\documentclass[leqno,12pt]{article}

\def\today{${\scriptscriptstyle\number\day-\number\month-\number\year}$}
\usepackage{graphicx}
\usepackage{fancyhdr}
\usepackage{amsmath}
\usepackage{amsthm}
\usepackage{amsfonts}
\newtheorem{theorem}{Theorem}[section]
\newtheorem*{maina}{Theorem A}
\newtheorem*{mainb}{Theorem B}

\newtheorem{lemma}[theorem]{Lemma}
\newtheorem{corollary}[theorem]{Corollary}
\theoremstyle{definition}

\newtheorem{remark}{Remark}[section]
\newtheorem{notation}[theorem]{Notation}

\def\address#1{{\center{#1}}}
\date{}
\def\m@th{\mathsurround=0pt}
\def\eqal#1{\null\,\vcenter{\openup\jot\m@th
 \ialign{\strut\hfil$\displaystyle{##}$&&$\displaystyle{{}##}$\hfil
 \crcr#1\crcr}}\,}
\def\matrix#1{\null\,\vcenter{\normalbaselines\m@th
 \ialign{\hfil$##$\hfil&&\quad\hfil$##$\hfil\crcr
 \mathstrut\crcr\noalign{\kern-\baselineskip}
 #1\crcr\mathstrut\crcr\noalign{\kern-\baselineskip}}}\,}
\def\N{{\Bbb N}}
\def\R{{\Bbb R}}
\def\divv{{\rm div}\,}
\def\rot{{\rm rot}\,}
\def\const{{\rm const}}
\def\diagin{\hbox{---}\hskip-14pt\intop}
\def\diagint{\hbox{--}\hskip-9pt\intop}
\def\diagintop{\mathop{\mathchoice
{{\diagin}}%
{{\diagint}}%
{{\diagint}}%
{{\diagint}}%
}\limits}

\numberwithin{equation}{section}

\title{Global regular periodic solutions to equations of weakly compressible barotropic fluid motions}
\author{Wojciech M. Zaj\c{a}czkowski}

\begin{document}
\input amssym.def
\input amssym.tex
\maketitle
\thispagestyle{fancy}

\address{$^1$  Institute of Mathematics, Polish Academy of Sciences,\\
\'Sniadeckich 8, 00-656 Warsaw, Poland\\
e-mail:wz@impan.gov.pl\\
Institute of Mathematics and Cryptology, 
Cybernetics Faculty, \\
Military University of Technology,\\
S. Kaliskiego 2, 00-908 Warsaw, Poland\\}

\begin{abstract}
We consider barotropic motions described by the compressible Navier-Stokes equations in a box with periodic boundary conditions. We are looking for density $\varrho$ in the form $\varrho=a+\eta$, where $a$ is a constant and $\eta|_{t=0}$ is sufficiently small in $H^2$-norm. We assume existence of potentials $\varphi$ and $\psi$ such that $v=\nabla\varphi+\rot\psi$. Next we assume that $\nabla\varphi|_{t=0}$ is sufficiently small in $H^2$-norm too. Finally, we assume that the second viscosity coefficient $\nu$ is sufficiently large. Then we prove long time existence of solutions such that $v\in L_\infty(0,T;H^2(\Omega))\cap L_2(0,T;H^3(\Omega))$, $v_{,t}\in L_\infty(0,T;H^1(\Omega))\cap L_2(0,T;H^2(\Omega))$, where the existence time $T$ is proportional to $\nu$. Next for $T$ sufficiently large we obtain that $v(T)$ is correspondingly small so global existence is proved using the methods appropriate for problems with small data.


\end{abstract}

\section{Introduction}\label{s1}

We are looking for existence of global regular periodic solutions to the following problem
\begin{equation}\eqal{
&\varrho v_t+\varrho v\cdot\nabla v-\mu\Delta v-\nu\nabla\divv v+\nabla p= \varrho f\quad &{\rm in}\ \ \Omega\times\R_+,\cr
&\varrho_t+\divv(\varrho v)=0\quad &{\rm in}\ \ \Omega\times\R_+,\cr
&v|_{t=0}=v_0,\ \ \varrho|_{t=0}=\varrho_0\quad &{\rm in}\ \ \Omega,\cr}
\label{1.1}
\end{equation}
where $v=v(x,t)=(v_1(x,t),v_2(x,t),v_3(x,t))\in\R^3$ is the velocity of the fluid, $\varrho=\varrho(x,t)\in\R_+$ is density, $f=f(x,t)=(f_1(x,t),f_2(x,t),f_3(x,t))\in\R^3$ is the external force field, $p=p(\varrho)$ and $\mu$, $\nu$ are constant positive viscosity coefficients. $\Omega\subset\R^3$ is a box and the periodic boundary conditions are assumed on $\partial\Omega$.

Looking for weakly compressible motions we assume that
\begin{equation}
\varrho=a+\eta,
\label{1.2}
\end{equation}
where $a$ is a positive constant, $\eta_0=\eta|_{t=0}$ is sufficiently small and the second viscosity coefficient $\nu$ is sufficiently large.

In view of (\ref{1.2}) we write (\ref{1.1}) in the form
\begin{equation}\eqal{
&(a+\eta)(v_t+v\cdot\nabla v)-\mu\Delta v-\nu\nabla\divv v+a_0\nabla\eta\cr
&=(p_\varrho(a))-p_\varrho(a+\eta)\nabla\eta+(a+\eta)f,\cr
&v|_{t=0}=v_0,\cr}
\label{1.3}
\end{equation}
where $p_\varrho={dp\over d\varrho}$, $a_0=p_\varrho(a)$ and
\begin{equation}\eqal{
&\eta_t+v\cdot\nabla\eta+a\divv v+\eta\divv v=0,\cr
&\eta|_{t=0}=\eta_0.\cr}
\label{1.4}
\end{equation}
Sometimes it is convenient to consider (\ref{1.3}) in the form
\begin{equation}\eqal{
&v_t+v\cdot\nabla v-{\mu\over a}\Delta v-{\nu\over a}\nabla\divv v+{a_0\over a}\nabla\eta\cr
&=-{\mu\over a}{\eta\over a+\eta}\Delta v-{\nu\over a}{\eta\over a+\eta}\nabla\divv v+{a_0\over a}{\eta\over a+\eta}\nabla\eta\cr
&\quad+{1\over a+\eta}(p_\varrho(a)-p_\varrho(a+\eta))\nabla\eta+f,\cr
&v|_{t=0}=v_0.\cr}
\label{1.5}
\end{equation}
Since we are looking for solutions to (\ref{1.1}) with large $\nu$ it is natural to introduce periodic potentials $\varphi$ and $\psi$ such that
\begin{equation}
v=\nabla\varphi+\rot\psi+G,
\label{1.6}
\end{equation}
where $G={1\over|\Omega|}\intop_\Omega vdx$.

\noindent
From $(\ref{1.1})_{1,2}$ and assumption (\ref{1.2}) we obtain
$$
{d\over dt}\intop_\Omega(a+\eta)vdx=\intop_\Omega(a+\eta)fdx.
$$
Hence
$$
\intop_\Omega vdx={1\over a}\bigg[-\intop_\Omega\eta vdx+\intop_{\Omega^t}(a+\eta)fdxdt'+\intop_\Omega(a+\eta_0)v_0dx\bigg].
$$
Therefore,
|$$
G={1\over a|\Omega|}\bigg[-\intop_\Omega\eta vdx+\intop_{\Omega^t}(a+\eta)fdxdt'+\intop_\Omega(a+\eta_0)v_0dx\bigg].
$$
In this case equations (\ref{1.3}), (\ref{1.4}), (\ref{1.5}) take the form
\begin{equation}\eqal{
&(a+\eta)(\nabla\varphi_t+\rot\psi_t+G_t+(\nabla\varphi+\rot\psi+G)\cdot\nabla (\nabla\varphi+\rot\psi))\cr
&\quad-\mu\Delta(\nabla\varphi+\rot\psi)-\nu\nabla\Delta\varphi+a_0\nabla\eta\cr
&= (p_\varrho(a)-p_\varrho(a+\eta))\nabla\eta+(a+\eta)f,\cr
&\nabla\varphi|_{t=0}=\nabla\varphi_0,\ \ \rot\psi|_{t=0}=\rot\psi_0.\cr}
\label{1.7}
\end{equation}
Moreover, we assume that
\begin{equation}
f=f_r+f_g,
\label{1.8}
\end{equation}
where $f_r$ is divergence free and $f_g$ is the gradient part. Then $\divv f=\divv f_g$, $\rot f=\rot f_r$. Next we have
\begin{equation}\eqal{
&\eta_t+v\cdot\nabla\eta+a\Delta\varphi+\eta\Delta\varphi=0\cr
&\eta|_{t=0}=\eta_0.\cr}
\label{1.9}
\end{equation}
Finally, (\ref{1.5}) takes the form
\begin{equation}\eqal{
&\nabla\varphi_t+\rot\psi_t+G_t+(\nabla\varphi+\rot\psi+G)\cdot \nabla(\nabla\varphi+\rot\psi)\cr
&\quad-{\mu\over a}\Delta(\nabla\varphi+\rot\psi)-{\nu\over a}\nabla\Delta\varphi+{a_0\over a}\nabla\eta=-{\mu\over a}{\eta\over a+\eta}\Delta v\cr
&\quad-\nu{\eta\over a+\eta}\nabla\divv v+a_0{\eta\over a+\eta}\nabla\eta+{1\over a+\eta}(p_\varrho(a)-p_\varrho(a+\eta))\nabla\eta+f_r+f_g,\cr
&v|_{t=0}=v_0.\cr}
\label{1.10}
\end{equation}
In this paper the following barotropic motions are considered
\begin{equation}
p=A\varrho^\varkappa,\ \ \varkappa>1,\ \ A-\const.
\label{1.11}
\end{equation}
We need an equation for $\nabla\eta$. To derive it we multiply (\ref{1.7}) by $a\over\nu$, apply operator $\nabla$ to (\ref{1.9}) and sum up the results. Then we have
\begin{equation}\eqal{
&\nabla\eta_t+{a_0\over\nu}\nabla\eta=-\nabla(v\cdot\nabla\eta)-\nabla (\eta\Delta\varphi)-{a\over\nu}(\varrho v_t+\varrho v\cdot\nabla v)\cr
&\quad+{\mu a\over\nu}\Delta v+{a\over\nu}(p_\varrho(a)-p_\varrho(a+\eta)) \nabla\eta+{a\over\nu}\varrho f,\cr
&\eta|_{t=0}=\eta_0.\cr}
\label{1.12}
\end{equation}
The aim of this paper is to prove existence of global regular periodic solutions to problem (\ref{1.7})--(\ref{1.9}). To show existence of such solutions we assume that the initial density is close to a constant assuming that the norms $\|\eta(0)\|_{H^2(\Omega)}$, $\|\eta_t(0)\|_{H^1(\Omega)}$ are sufficiently small.

\noindent
Moreover, we assume that the second viscosity coefficient $\nu$ is sufficiently large. This implies that velocity $v$ must be considered in form (\ref{1.6}) because divergence free and potential parts have to be treated differently. Therefore we are looking for such solutions that $\nabla\varphi$ in some norms is small but $\rot\psi$ in these norms not.

\noindent
The natural way to derive necessary estimates is the energy method. The method was first applied to equations of viscous compressible heat-con\-duc\-ting fluids by Matsumura and Nishida in \cite{MN1, MN2, MN3}. Next by Valli and Zaj\c{a}czkowski in \cite{V, VZ}. The free boundary barotropic case was considered in \cite{Z1, Z2}. Finally, a free boundary viscous compressible heat-conducting case was considered by Zadrzy\'nska \cite{Za}. The method is natural in the problem because $v$ is considered in form (\ref{1.6}) and the second viscosity $\nu$ is very large comparing to $\mu$ so the anisotropic approach to velocity is necessary.

\noindent
Since $\|\eta(0)\|_{H^2(\Omega)}$ and $\|\eta_t(0)\|_{H^1(\Omega)}$ are small we denote the motion a weakly compressible. However, we consider the periodic problem the proof can be extended to motions with different boundary conditions.

\noindent
In Section \ref{s2} some preliminary results are formulated and proved. In Section~\ref{s3} the main estimates for large $\nu$ and time are shown. The estimates are of an a priori type. However, for the local solutions they are real estimates. The estimates are made without smallness assumptions on norms of $\rot\psi$. This is possible thanks to the following two estimates: $\|v\|_{L_{6,\infty}(\Omega^t)}$ (see Lemma \ref{l2.2}, Remark \ref{r2.3} and Remark \ref{r2.4}) and $\|v_t\|_{L_{2,\infty}(\Omega^t)}$ (see Lemma \ref{l2.8}). These estimates are crucial for proofs of Lemma \ref{l4.1}, Corollary \ref{c4.2}, Theorem \ref{t4.3} and Theorem \ref{t5.9}.

\noindent
In Section \ref{s4} we prove existence of long time solutions, where the existence time $T$ is proportional to $\nu$. The existence is proved in the following way. Having long time esimate proved in Corollary \ref{c4.2} we have existence by time extension of local solutions.

\noindent
The proof of Theorem \ref{t4.3} is based on the time extension of a local solution and the derived long time estimate. The proof is understandable but not explicit. To have an explicit proof we have to use the method of successive approximations in the time interval $[0,T]$. But looking for considerations in Section \ref{s3} such proof will be very complicated.

\noindent
However, the long time solutions are not global. Therefore in Section \ref{s5} we prove existence of global regular solutions. For this we need some decay estimates (see (\ref{5.51}), (\ref{5.52}))which imply smallness of data at time $T$ (see (\ref{5.53})). Then considering problem (\ref{1.7})--(\ref{1.9}) with small initial data at time $T$ (see (\ref{5.53})) we use the technique from \cite{BSZ, VZ, Z1, Z2} to prove existence of global regular solutions (see Theorem \ref{t5.9}).

\noindent
Now, we formulate the main results of this paper
\goodbreak

\begin{maina}[local existence]
Let $\nu>0$ be given sufficiently large. Let $v=\nabla\varphi+\rot\psi$, $\varrho=a+\eta$, $a$-positive constant. Let $\eta(0)\in L_\infty(\Omega)$, $\eta(0),\nabla\varphi(0),\rot\psi(0)\in\Gamma_1^2(\Omega)$, $f\in L_2(0,T;\Gamma_1^1(\Omega))$, $\|\nabla\varphi(0)\|_{\Gamma_1^2(\Omega)}\le{c\over\sqrt{\nu}}$,\break $\|\rot\phi(0)\|_{\Gamma_1^2(\Omega)}\le c$, $\|\eta(0)\|_{\Gamma_1^2(\Omega)}\le{c\over\nu}$, $|f_g|_{6/5,2,\Omega^t}\le{c\over\nu}$, $f\in L_2(0,T;\Gamma_1^1(\Omega))$, $f\in L_6(0,T;L_3(\Omega))\cap L_1(0,T;L_\infty(\Omega))$. Assume that there exist positive constants $\varphi_*$ and $c$ such that $c\nu^\varkappa\le\varphi_*\le\varphi(0)$, where $\varkappa\in(1/2,1)$. Then there exists a regular long time solution to problem (\ref{1.1}) expressed in the form (\ref{1.6})--(\ref{1.9}) such that $\sqrt{\nu}\nabla\varphi,\rot\psi\in L_\infty(0,T;\Gamma_1^2(\Omega))\cap L_2(0,T;\Gamma_1^3(\Omega))\,\!,\!$ 
$\nu\nabla\varphi\in L_2(0,T;\Gamma_1^3(\Omega))$, $\nu\eta\in L_\infty(0,T;\Gamma_1^2(\Omega))$, and $v\in{\frak N}(\Omega^t)$, $t\le T\le\nu$, where $T$ is the time of local existence and
\begin{equation}\eqal{
&\|v\|_{{\frak N}(\Omega^t)}\le\phi(\nu\|\eta(0)\|_{\Gamma_1^2(\Omega)},\quad \nu^{1/2}\|\nabla\varphi(0)\|_{\Gamma_1^2(\Omega)},\cr
&\|\rot\psi(0)\|_{\Gamma_1^2(\Omega)},\quad \nu\|\varphi(0)\|_1,\quad \nu|f_g|_{6/5,2,\Omega^t},\quad \|f\|_{L_2(0,t;\Gamma_1^1(\Omega))},\cr
&\|f\|_{L_6(0,t;L_3(\Omega))\cap L_1(0,t;L_\infty(\Omega))}),\cr}
\label{1.13}
\end{equation}
where $t\le T$, $\phi$ is an increasing positive function of its arguments and the space ${\frak N}(\Omega^t)$ is defined by
$$\eqal{
\|v\|_{{\frak N}(\Omega^t)}&=\nu^{1/2}\|\nabla\varphi\|_{L_\infty(0,t;\Gamma_1^2(\Omega))}+ \|\rot\psi\|_{L_\infty(0,t;\Gamma_1^2(\Omega))}\cr
&\quad+\|\rot\psi\|_{L_2(0,t;\Gamma_1^3(\Omega))}+\nu \|\nabla\varphi\|_{L_2(0,t;\Gamma_1^3(\Omega))}.\cr}
$$
For $T>\nu$ we have a global existence of such solutions that.\\
To prove global existence of solutions to problem (\ref{1.1}) we use the step by step in time extension.
\end{maina}

The existence for $t\le T\le\nu$ is shown in Section \ref{s4} but for $t\ge T>\nu$ in Section~\ref{s5}.\\
\noindent
Therefore, we have

\begin{mainb}[Global existence]
Let the assumptions of the Theorem A hold. Let $\|f_g(t)\|_1\le ce^{-\alpha t}$, $\alpha>0$. Let $f\in L_2(kT,(k+1)T;\Gamma_1^2(\Omega))\cap L_6(kT,(k+1)T;L_3(\Omega))\cap L_1(kT,(k+1)T;L_\infty(\Omega))$.\\
Then $v\in{\frak N}(\Omega\times(kT,(k+1)T))$, $k\in\N_0$ and (\ref{1.13}) holds with interval $(0,T)$ replaced by $(kT,(k+1)T)$, $k\in\N_0=\N\cup\{0\}$.
\end{mainb}

\noindent
Now we describe the idea of proofs of Theorems A and B.

Our aim is to derive a global estimate for regular solutions to (\ref{1.3}), (\ref{1.4}) using presentation (\ref{1.9}), (\ref{1.10}). By the regular velocity we mean such velocity that $v\in L_2(0,T;\Gamma_1^3(\Omega))$. Then we have a corresponding regularity for $\eta\in L_\infty(0,T;\Gamma_1^2(\Omega))$. This kind of regularity is necessary to estimate nonlinear terms.

First we derive the inequality for $v\in L_\infty(0,T;L_6(\Omega))$. Multiplying $(\ref{1.3})_1$ by $v|v|^{r-2}$, integrating over $\Omega^t$ we get (see Lemma \ref{l2.2})
\begin{equation}\eqal{
&|v(t)|_r+\bigg(\intop_0^t|\nabla|v|^{r/2}|^2dxdt'\bigg)^{1/r}\le c|\eta|_{{3r\over r+1},r,\Omega^t}+c\nu|\Delta\varphi|_{{3r\over r+1},r,\Omega^t}\cr
&\quad+c|f|_{{3r\over2r+1},r,\Omega^t}+c|\varrho_0|_\infty^{1/r}|v_0|_r,\cr}
\label{1.14}
\end{equation}
where $r\le 6$.

\noindent
The second term on the r.h.s. of (\ref{1.14}) is not controlled for large $\nu$. Hence to control it we use the interpolation in the case $r=6$
\begin{equation}
|\Delta\varphi|_{18/7,6,\Omega^t}\le c|\nabla\Delta\varphi|_{2,\infty,\Omega^t}^{2/3} |\nabla\varphi|_{2,\Omega^t}^{1/3},
\label{1.15}
\end{equation}
where the mean values of $\Delta\varphi$, $\nabla\Delta\varphi$, $\nabla\varphi$ are zero in view of periodic boundary conditions. To estimate the last factor we derive the equation obtained from $(\ref{1.3})_1$ by applying the div operator
\begin{equation}\eqal{
&a\Delta\varphi_t-(\mu+\nu)\Delta^2\varphi+a_0\Delta\eta=-a\divv(v\cdot\nabla v)+\divv[-\eta v_t\cr
&\quad-\eta v\cdot\nabla v+(p_\varrho(a)-p_\varrho(a+\eta))\nabla\eta+ (a+\eta)f].\cr}
\label{1.16}
\end{equation}
Applying operator $\Delta^{-1}$ to (\ref{1.16}) yields
\begin{equation}\eqal{
&a\varphi_t-(\mu+\nu)\Delta\varphi=-a\Delta^{-1}\partial_{x_i} \partial_{x_j}(v_iv_j)+a\Delta^{-1}\partial_{x_i}(\Delta\varphi v_i)\cr
&\quad+\Delta^{-1}\divv[-\eta v_t-\eta v\cdot\nabla v+(p_\varrho(a)-p_\varrho(a+\eta))\nabla\eta+af_g\cr
&\quad+\eta f]-(\eta-\diagintop_\Omega\eta dx)+a\diagintop_\Omega\varphi_tdx\equiv D_1+D_2+F-\bar\eta+a\diagintop_\Omega\varphi_tdx,\cr}
\label{1.17}
\end{equation}
where $\diagintop_\Omega={1\over|\Omega|}\intop_\Omega$, $\bar\eta=\eta-\diagintop_\Omega\eta dx$.

The aim of this paper is to show that the quantity $\Psi=\nu|\nabla\varphi|_{3,1,2,\Omega^t}$ is bounded for any $t\in\R_+$. Then we see that $|\nabla\Delta\varphi|_{2,\infty,\Omega^t}\le c\|\Delta\varphi\|_{W_2^{2,1}(\Omega^t)}\le c|\nabla\varphi|_{3,1,2,\Omega^t}$. Hence we have to derive that $\nu^{1/3}|\nabla\varphi|_{2,\Omega^t}^{1/3}\le{\Phi^\alpha\over\nu^\beta}$, $\alpha$, $\beta$ positive constant numbers. This is the aim of Lemmas \ref{l2.3}, \ref{l2.4} and Remark \ref{r2.2}.

\noindent
First we show the idea of the proof of Lemma \ref{l2.3}. Multiplying (\ref{1.17}) by $\varphi$ and integrating over $\Omega$ yields
\begin{equation}\eqal{
&{a\over 2}{d\over dt}|\varphi|_2^2+(\mu+\nu)|\nabla\varphi|_2^2=\intop_\Omega D_1\varphi dx+\intop_\Omega D_2\varphi dx+\intop_\Omega F\varphi dx\cr
&\quad-\intop_\Omega\bar\eta\varphi dx+a\diagintop_\Omega\varphi_tdx\intop_\Omega\varphi dx.\cr}
\label{1.18}
\end{equation}
We need to obtain such estimate for $|\nabla\varphi|_{2,\Omega^t}$ that
\begin{equation}
|\nabla\varphi|_{2,\Omega^t}\le{c\over\nu^\alpha},\quad \alpha>1,
\label{1.19}
\end{equation}
where $c$ depends on such norms of $v$, $\nabla\varphi$, $\eta$ that $\nabla\varphi,v\in L_2(0,T;\Gamma_1^3(\Omega))\cap L_\infty(0,T;\Gamma_1^2(\Omega))$, $\eta\in L_\infty(0,T;\Gamma_1^2(\Omega))$.

\noindent
The dependence can be strongly nonlinear because small parameter $1/\nu^\alpha$ ($\nu$ is assumed to be large) helps to get an estimate by a perturbation argument. But the first and the last terms from the r.h.s. od (\ref{1.18}) will not imply estimate (\ref{1.19}). They must be treated in the following different way
$$\eqal{
\bigg|\intop_\Omega D_1\varphi dx\bigg|&\le\intop_\Omega {|D_1|\varphi^2\over
\varphi_*}dx\le{|D_1|_{p/(p-2)}\over\varphi_*}|\varphi|_p^2\cr
&\le{1\over2}(\mu+\nu)|\nabla\varphi|_2^2+ {c|D_1|_{p/(p-2)}^{1/(1-\varkappa)}\over [(\mu+\nu)^\varkappa\varphi_*]^{1/(1-\varkappa)}}|\varphi|_2^2,\quad \varkappa\in(1/2,1)\cr}
$$
and
$$
a\diagintop_\Omega\varphi_t dx\intop_\Omega\varphi dx=0,
$$
because we assume that $\varphi=\varphi'+L$, $\intop\varphi'dx=0$ and $L$ is some arbitrary constant, where $0<\varphi_*=\min_\Omega\varphi$. $\varphi$ can be chosen positive because it is determined up to an arbitrary positive constant $L$.

\noindent
Using that $f=\divv F'$ we get
$$
\bigg|\intop_\Omega F\varphi dx\bigg|=\bigg|\intop_\Omega F'\nabla\varphi dx\bigg|\le\varepsilon|\nabla\varphi|_2^2+c/\varepsilon|F'|_2^2.
$$
Since $|D_1|\le c|v|^2$ we derive (\ref{2.31}) in the form
\begin{equation}\eqal{
&a|\varphi(t)|_2^2+(\mu+\nu)|\nabla\varphi|_{2,\Omega^t}^2\le\exp \bigg({c|v|_{2p/(p-2),2/(1-\varkappa),\Omega^t}^{2/(1-\varkappa)}\over [(\mu+\nu)^\varkappa\varphi_*]^{1/(1-\varkappa)}}\bigg)\cr
&\quad\cdot\bigg[{c\over\mu+\nu}\{|\Delta\varphi|_{3/2,\infty,\Omega^t}^2A_1^2+ |\eta|_{6,\infty,\Omega^t}^2|v_t|_{2,\Omega^t}^2\cr
&\quad+|v|_{6,\infty,\Omega^t}^2A_1^2+|\nabla\eta|_{2,\Omega^t}^2+ |f|_{2,\Omega^t}^2+|f_g|_{2,\Omega^t}^2+|\eta|_{2,\Omega^t}^2\}\cr
&\quad+a|\varphi(0)|_2^2\bigg].\cr}
\label{1.20}
\end{equation}
Lemma \ref{l2.4} implies the bound for $\bar\varphi=\varphi-\diagintop_\Omega\varphi dx$.

\noindent
Let $\hat k=\max\{|\bar\varphi(0)|_\infty,1\}$, $\eta\in L_\infty(0,T;H^2(\Omega))$, $v\in L_\infty(\Omega^T)$, $f\in L_3(\Omega^T)$. Then
\begin{equation}\eqal{
|\bar\varphi|_{\infty,\Omega^t}&\le2\hat k\bigg[1+ct^{(1+\varkappa)/r_0} {\rm meas}^{(1+\varkappa)/p_0}\Omega\bigg({1\over\mu+\nu}\bigg)^{(2/q)(1+1/\varkappa)} \cdot\cr
&\quad\cdot\phi(\|\eta\|_{2,\infty,\Omega^t},|v|_{\infty,\Omega^t}, |f|_{3,\Omega^t})\bigg]\equiv\gamma_*,\cr}
\label{1.21}
\end{equation}
where $t\le T$, $\varkappa\in(0,4/3)$, $3/p_0+2/r_0=3/2$ and $\phi$ is an increasing positive function. Next Remark \ref{r2.2} gives
\begin{equation}
0<\varphi_*\le\diagintop_\Omega\varphi dx-\gamma_*\le\varphi,
\label{1.22}
\end{equation}
where $\varphi$ is defined up to an arbitrary constant $L$ but $\gamma_*$ not because $\gamma_*$ depends on $\hat k=\max\{|\bar\varphi(0)|_\infty,1\}$ which is independent of $L$.

Finally we have to estimate the second term on the r.h.s. of (\ref{1.14}). In view of (\ref{1.15}) we have
\begin{equation}
\nu|\Delta\varphi|_{18/7,6,\Omega^t}\le c\nu^{2/3}|\nabla\Delta\varphi|_{2,\infty,\Omega^t}^{2/3}\nu^{1/3} |\nabla\varphi|_{2,\Omega^t}^{1/3}\equiv I_1.
\label{1.23}
\end{equation}
Recall that $\Psi(t)=\nu|\nabla\varphi|_{3,1,2,\Omega^t}$. Our aim is to show that there exist positive constants $\alpha_k$, $\beta_k$, $k=1,2,\dots$, such that
\begin{equation}
I_1\le\sum_k{\Psi^{\alpha_k}\over\nu^{\beta_k}}.
\label{1.24}
\end{equation}
By imbedding we already have that
\begin{equation}
I_1\le c\Psi^{2/3}\nu^{1/3}|\nabla\varphi|_{2,\Omega^t}^{1/3}.
\label{1.25}
\end{equation}
Hence, we have to use (\ref{1.20}). First we have to estimate the argument of exp function. We assume that
\begin{equation}
c_1\le(\mu+\nu)^\varkappa\varphi_*.
\label{1.26}
\end{equation}
Then the term under exp in (\ref{1.20}) is bounded by
$$
c|v|_{\infty,\Omega^t}^{2/(1-\varkappa)},\ \ \varkappa=3/2-3/p\in(0,1)\quad {\rm for}\ \ p\in(2,6).
$$
Assuming that
\begin{equation}
|\eta(0)|\le{c\over\nu},|f_g|_{2,\Omega^t}\le{c\over\nu}
\label{1.27}
\end{equation}
we obtain (see (\ref{2.62}) and (\ref{2.64}))
\begin{equation}
|\eta|_{r,\infty,\Omega^t}\le{c\over\nu}(\Psi+c)
\label{1.28}
\end{equation}
Then (\ref{1.20}) implies
\begin{equation}\eqal{
|\nabla\varphi|_{2,\Omega^t}&\le\phi(|v|_{\infty,\Omega^t},|v_t|_{2,\Omega^t}, |\nabla\eta|_{2,\Omega^t},|f|_{2,\Omega^t})\cdot\cr
&\quad\cdot\bigg[{1\over\nu^2}(\Psi+c)+{c\over\nu^{1/2}}|\varphi(0)|_2\bigg].\cr}
\label{1.29}
\end{equation}
Assuming additionally that
\begin{equation}
|\varphi(0)|_2\le{c\over\nu^\varkappa},\quad \varkappa>1/2
\label{1.30}
\end{equation}
and using (\ref{1.30}) in (\ref{1.29}) and the result in (\ref{1.23}) yields
\begin{equation}\eqal{
I_1&\le\phi(|v|_{\infty,\Omega^t},|v_t|_{2,\Omega^t},|\nabla\eta|_{2,\Omega^t}, |f|_{2,\Omega^t})\cdot\cr
&\quad\cdot\bigg[{\Psi\over\nu^{1/3}}+{\Psi^{2/3}\over\nu^{1/3}}+ {\Psi^{2/3}\over\nu^{\varkappa/3-1/6}}\bigg].\cr}
\label{1.31}
\end{equation}
Then from (\ref{1.14}) for $r=6$ we get (see (\ref{2.82}))
\begin{equation}\eqal{
&|v(t)|_6+\bigg(\intop_0^t|\nabla|v|^3|^2dxdt'\bigg)^{1/6}\le c|\eta|_{18/7,6,\Omega^t}\cr
&\quad+\phi(|v|_{\infty,\Omega^t},|v_t|_{2,\Omega^t},|\nabla\eta|_{2,\Omega^t}, |f|_{2,\Omega^t})[\Psi/\nu^{1/3}+\Psi^{2/3}/\nu^{1/3}\cr
&\quad+\Psi^{2/3}/\nu^{\varkappa/3-1/6}]+c|f|_{18/7,6,\Omega^t}+c |\varrho_0|_\infty^{1/6}|v_0|_6\equiv D_1.\cr}
\label{1.32}
\end{equation}
Finally, Lemma \ref{l2.8} implies
\begin{equation}\eqal{
&|v_t(t)|_2^2+\mu|\nabla v_t|_{2,\Omega^t}^2+\nu|\Delta\varphi_t|_{2,\Omega^t}^2 \le\exp(|\eta_t|_{3,2,\Omega^t}^2\cr
&\quad+(\|\eta\|_{2,\infty,\Omega^t}^2+D_1^2+1)A_1^2)\cdot[|\eta_t|_{2,\Omega^t}^2 +\|\eta_t\|_{1,\infty,\Omega^t}^2(D_1^2A_1^2\cr
&\quad+|\nabla\varphi|_{3,1,2,\Omega^t}^2A_1^2+|f|_{2,\Omega^t}^2)+ |f_t|_{2,\Omega^t}^2+|\varrho_0|_\infty|v_t(0)|_2^2]\cr
&\equiv D_2^2.\cr}
\label{1.33}
\end{equation}
Let us introduce the quantities
$$\eqal{
&\Psi(t)=\nu|\nabla\varphi|_{3,1,2,\Omega^t},\cr
&\chi_1^2(t)=\nu|\nabla\varphi|_{2,1,\infty,\Omega^t}^2+ |\rot\psi|_{2,1,\infty,\Omega^t}^2,\cr
&\chi_2^2(t)=\nu|\nabla\varphi|_{3,1,2,\Omega^t}^2+ |\rot\psi|_{3,1,2,\Omega^t}^2,\cr
&\Phi_1^2(t)=\nu(|\nabla\varphi|_{2,1,\infty,\Omega^t}^2+ |\nabla\varphi|_{3,1,2,\Omega^t}^2),\cr
&\Phi_2^2(t)=|\rot\psi|_{2,1,\infty,\Omega^t}^2+|\rot\psi|_{3,1,2,\Omega^t}^2.\cr}
$$
From (\ref{4.2}) we have inequality
\begin{equation}\eqal{
&\chi_1^2(t)+\chi_2^2(t)+\Psi^2(t)\le\phi\bigg(D_1,D_2,A_1,{\Psi\over\nu}, {\Phi_1\over\sqrt{\nu}},{\Psi\over\nu},\Phi_2,\cr
&\quad|f|_{1,1,2,\Omega^t},|\eta|_{2,1,\infty,\Omega^t}, |\eta|_{2,1,2,\Omega^t}\bigg)\cr
&\quad+c|\eta|_{2,1,\infty,\Omega^t}^2(1+|\eta|_{2,1,\infty,\Omega^t}^2)\Psi^2+ \nu|\varphi(0)|_{2,1}^2+|\rot\psi(0)|_{2,1}^2.\cr}
\label{1.34}
\end{equation}
Hence for $\nu$ sufficiently large there exists a constant $A$ such that
\begin{equation}
X(t)\equiv\chi_1t)+\chi_2(t)+\Psi(t)\le A.
\label{1.35}
\end{equation}
there is such restriction on time that $t$ is proportional to some positive increasing function of $\nu$. But large time is not convenient because strong restrictions on $v$, $f$ follow from time-integrals norms. Therefore to derive estimate (\ref{1.35}) for all $t\in\R_+$ we perform the procedure step by step in time. In Theorem \ref{t5.9} we prove that
\begin{equation}
X(t)\le A,\quad t\in[kT,(k+1)T],\quad k\in\N_0=\|\cup\{0\},
\label{1.36}
\end{equation}
where $A$ does not depend on $k$.

\noindent
The derivation of such estimate is possible thanks to sufficiently large $\nu$, $T$ and sufficiently small $|\nabla\varphi(0)|_{2,1}|$, $|\eta(0)|_{2,1}$.

\noindent
In the proof is used natural dissipation mechanism in  the compressible Navier-Stokes equations connected with viscosity coefficients $\mu$ and $\nu$.

\section{Notation and auxiliary results}\label{s2}

We use the simplified notation
$$\eqal{
&\|u\|_{L_p(\Omega)}=|u|_p,\ \ \|u\|_{H^s(\Omega)}=\|u\|_s,\ \ \|u\|_{W_p^s(\Omega)}=\|u\|_{s,p},\cr
&|u|_{k,l}^2=\sum_{i=0}^l\|\partial_t^iu\|_{k-i}^2,\ \ |u|_{k,l,r,\Omega^t}=\bigg(\intop_0^t|u(t')|_{k,l}^rdt'\bigg)^{1/r},\cr
&|u|_{r,q,\Omega^t}=\bigg(\intop_0^t|u(t')|_r^qdt'\bigg)^{1/q},\cr
&\|u\|_{L_r(0,t;H^s(\Omega))}=\|u\|_{s,r,\Omega^t},\ \ \|u\|_{L_r(0,t;W_p^s(\Omega))}=\|u\|_{s,p,r,\Omega^t}.\cr}
$$
Introduce the spaces
$$
\Gamma_l^k(\Omega)=\{u:|u|_{k,l}<\infty\},\ \ l\le k,\ \ l,k\in\N_0.
$$
First we obtain the energy type estimate for solutions to problem (\ref{1.1}).

\begin{lemma}\label{l2.1}
Assume that $(\varrho,v)$ is a solution to problem (\ref{1.1}). Assume that $p=p(\varrho)=A\varrho^\varkappa$, $\varkappa>1$.
Assume the periodic boundary conditions and that $f\in L_2(\Omega\times(kT,(k+1)T)$, $\alpha(0)=\intop_\Omega\big({1\over2}\varrho_0v_0^2+ {A\over\varkappa-1}\varrho_0^\varkappa\big)dx<\infty$
$$\eqal{
&B_1(T)=\sup_{k\in\N_0}\intop_{kT}^{(k+1)T}|\eta(t)|_2^2dt,\quad 
B_2(T)=\sup_{k\in\N_0}\intop_{kT}^{(k+1)T}|f(t)|_2^2dt,\quad \varrho_0\in L_2(\Omega),\cr
&v_0\in L_2(\Omega),\quad B(T)=\exp B_1(T)(B_2(T)+\sup_k\intop_{kT}^{(k+1)T} |\varrho|_\varkappa^\varkappa dt+|\varrho_0|_2^2|v_0|_2^2).\cr}
$$
Then
\begin{equation}\eqal{
&\intop_\Omega\bigg({1\over2}\varrho v^2+{A\over\varkappa-1}\varrho^\varkappa\bigg)dx+\mu\intop_0^t\|v(t')\|_1^2 dt'+\nu\intop_0^t|\divv v(t')|_2^2dt'\cr
&\le c(1+B_1(T)+B_2(T))(1+T)\bigg({B(T)\over1-e^{-\mu T}}+\alpha(0)\bigg)\equiv A_1^2.\cr}
\label{2.1}
\end{equation}
\end{lemma}

\begin{proof}
Multiplying $(\ref{1.1})_1$ by $v$, integrating over $\Omega$ and using the periodic boundary conditions, we obtain
\begin{equation}\eqal{
&{1\over2}\intop_\Omega(\varrho\partial_tv^2+\varrho v\cdot\nabla v^2)dx+\mu |\nabla v|_2^2+\nu|\divv v|_2^2\cr
&\quad+\intop_\Omega\nabla p(\varrho)\cdot vdx=\intop_\Omega\varrho f\cdot vdx\cr}
\label{2.2}
\end{equation}
Adding the identity
$$
{1\over2}\intop_\Omega[\varrho_t+\divv(\varrho v)]v^2dx=0
$$
we derive from (\ref{2.2}) the equality
\begin{equation}
{1\over2}{d\over dt}\intop_\Omega\varrho v^2dx+\mu|\nabla v|_2^2+\nu|\divv v|_2^2+\intop_\Omega\nabla p(\varrho)\cdot vdx=\intop_\Omega\varrho f\cdot vdx.
\label{2.3}
\end{equation}
Using that $p=A\varrho^\varkappa$, $\varkappa>1$, the last term on the l.h.s. of (\ref{2.3}) equals
\begin{equation}
\intop_\Omega\nabla p(\varrho)\cdot vdx=A\intop_\Omega v\cdot\nabla\varrho^\varkappa dx.
\label{2.4}
\end{equation}
To use (\ref{2.4}) in (\ref{2.3}) we multiply $(\ref{1.1})_2$ by $\varrho^{\varkappa-1}$. Then we get
$$
\varrho^{\varkappa-1}(\varrho_t+v\cdot\nabla\varrho)+\varrho^\varkappa\divv v=0.
$$
Continuing, we have
\begin{equation}
{1\over\varkappa}\partial_t\varrho^\varkappa+{1\over\varkappa}v\cdot \nabla\varrho^\varkappa+\varrho^\varkappa\divv v=0.
\label{2.5}
\end{equation}
Adding ${\varkappa-1\over\varkappa}v\cdot\nabla\varrho^\varkappa$ to both sides of (\ref{2.5}) yields
$$
{1\over\varkappa}{\partial\over\partial t}\varrho^\varkappa+\divv(v\varrho^\varkappa)= {\varkappa-1\over\varkappa}v\cdot\nabla\varrho^\varkappa.
$$
Integrating the equality over $\Omega$ and using boundary conditions gives
$$
{A\over\varkappa}{d\over dt}\intop_\Omega\varrho^\varkappa dx={A(\varkappa-1)\over\varkappa}\intop_\Omega v\cdot\nabla\varrho^\varkappa dx.
$$
Hence
\begin{equation}
{A\over\varkappa-1}{d\over dt}\intop_\Omega\varrho^\varkappa dx=A\intop_\Omega v\cdot\nabla\varrho^\varkappa dx.
\label{2.6}
\end{equation}
In view of (\ref{2.4}) and (\ref{2.6}) equality (\ref{2.3}) takes the form
\begin{equation}
{d\over dt}\intop_\Omega\bigg({1\over2}\varrho v^2+{A\over\varkappa-1}\varrho^\varkappa\bigg)dx+\mu|\nabla v|_2^2+\nu|\divv v|_2^2=\intop_\Omega\varrho f\cdot vdx.
\label{2.7}
\end{equation}
Using that $\diagintop_\Omega vdx=G$, where $G$ is defined between (\ref{1.6}) and (\ref{1.7}), and applying the Poincar\'e inequality we get
\begin{equation}\eqal{
&{d\over dt}\intop_\Omega\bigg({1\over2}\varrho v^2+{A\over\varkappa-1} \varrho^\varkappa\bigg)dx+\mu\|v\|_1^2+\nu|\divv v|_2^2\le\intop_\Omega\varrho f\cdot vdx\cr
&\quad+|G|^2\le\intop_\Omega\varrho f\cdot vdx+|\eta|_2^2|v|_2^2+\bigg|\intop_{\Omega^t}\varrho fdxdt'\bigg|^2+\bigg|\intop_\Omega\varrho_0v_0dx\bigg|^2.\cr}
\label{2.8}
\end{equation}
Exploiting that $a/2\le\varrho\le{3\over 2}a$ and introducing the quantity
$$
\alpha=\intop_\Omega\bigg({1\over2}\varrho v^2+{A\over\varkappa-1}\varrho^\varkappa\bigg)dx
$$
we have
\begin{equation}
{d\over dt}\alpha+\mu\alpha\le c|\eta|_2^2\alpha+c\beta,
\label{2.9}
\end{equation}
where
$$
\beta=|f|_2^2+|f|_{2,\Omega^t}^2+|\varrho_0|_2^2|v_0|_2^2+ |\varrho|_\varkappa^\varkappa
$$
Looking for the energy estimate with the bound independent of time we use the step by step in time argument. Then we consider (\ref{2.8}) in the time interval $[kT,(k+1)T]$, $k\in\N$. Then the second term in $\beta$ equals $|f|_{2,\Omega\times(kT,t)}^2$.

\noindent
From (\ref{2.8}) we have
$$
{d\over dt}\bigg[\alpha(t)\exp\bigg(\mu t-c\intop_{kT}^t|\eta(t')|_2^2dt'\bigg) \bigg]\le\beta\exp\bigg(\mu t-c\intop_{kT}^t|\eta(t')|_2^2dt'\bigg).
$$
Integrating with respect to time we finally get
$$\eqal{
\alpha(t)&\le\exp\bigg(-\mu t+c\intop_{kT}^t|\eta(t')|_2^2dt'\bigg)\intop_{kT}^t\beta\exp\bigg(\mu t'-c\intop_{kT}^{t'}|\eta(t'')|_2^2dt''\bigg)dt'\cr
&\quad+\alpha(kT)\exp(-\mu(t-kT)).\cr}
$$
Taking into account that some parts of $\beta$ are bounded in time we obtain the inequality
\begin{equation}\eqal{
\alpha(t)&\le\exp\bigg[c\intop_{kT}^t|\eta(t')|_2^2dt'\bigg]\bigg(\intop_{kT}^t |f|_{2,\Omega}^2dt'+|\varrho_0|_2^2|v_0|_2^2+ |\varrho|_{\varkappa,\infty,\Omega\times(kT,t)}^\varkappa\bigg)\cr
&\quad+\alpha(kT)\exp(-\mu(t-kT)).\cr}
\label{2.10}
\end{equation}
Setting $t=(k+1)T$ yields
$$\eqal{
\alpha((k+1)T)&\le\exp\bigg[c\intop_{kT}^{(k+1)T}|\eta(t')|_2^2dt'\bigg] \bigg[\intop_{kT}^{(k+1)T}(|f|_{2,\Omega}^2+|\varrho|_\varkappa^\varkappa)dt+|\varrho_0|_2^2|v_0|_2^2\bigg]\cr
&\quad+\alpha(kT)\exp(-\mu T)\le B(T)+\alpha(kT)\exp(-\mu T),\cr}
$$
where
$$\eqal{
B(T)&=\sup_{k\in\N_0}\exp\bigg[c\intop_{kT}^{(k+1)T}|\eta(t)|_2^2dt\bigg] \bigg[\intop_{kT}^{(k+1)T}(|f|_2^2+|\varrho|_\varkappa^\varkappa)dx\cr
&\quad+|\varrho_0|_2^2|v_0|_2^2\bigg].\cr}
$$
By interation we have
$$
\alpha(kT)\le{B(T)\over1-e^{-\mu T}}+\alpha(0)e^{-\mu kT},\quad k\in\N_0.
$$
Using the estimate in (\ref{2.10}) yields
\begin{equation}
\alpha(t)\le{2B(T)\over1-e^{-\mu T}}+\alpha(0),\quad t\in[kT,(k+1)T].
\label{2.11}
\end{equation}
Integrating (\ref{2.8}) with respect to time from $t=kT$ to $t\in(kT,(k+1)T]$, $k\in\N_0$, using (\ref{2.11}) and that $\varrho\in(a/2,3a/2)$ we obtain
$$\eqal{
&\intop_\Omega\bigg({1\over2}\varrho v^2+{A\over\varkappa-1} \varrho^\varkappa\bigg)dx+\mu\intop_{kT}^t\|v(t')\|_1^2dt'+\nu\intop_{kT}^t |\divv v(t')|_2^2dt'\cr
&\le c\bigg(\intop_{kT}^t|f(t')|_2^2dt'+\intop_{kT}^t|\eta(t')|_2^2dt'\bigg) \alpha(t)+cT\intop_{kT}^t|f(t')|_2^2dt'\cr
&\quad+T|\varrho_0|_2^2|v_0|_2^2+{B(T)\over1-e^{-\mu T}}+\alpha(0),\quad t\in[kT,(k+1)T].\cr}
$$
The above inequality implies (\ref{2.1}). This concludes the proof.
\end{proof}

\begin{lemma}\label{l2.2}
Assume that $r>2$, $\varrho f\in L_{3r\over2r+1}(\Omega)$, $p(\varrho)-p(a)\in L_{3r\over r+1}(\Omega)$, $\divv v\in L_{3r\over r+1}$, $\nu=\nu_1+\nu_2$, ${\nu_1^2\over\nu}$ sufficiently small. Then
\begin{equation}\eqal{
&{1\over r}{d\over dt}\intop_\Omega\varrho|v|^rdx+{k_0\over2}\intop_\Omega |\nabla|v|^{r/2}|^2dx+{\mu\over2}\intop_\Omega|v|^r \bigg|\nabla{v\over|v|}\bigg|^2dx\cr
&\quad+{\nu\over2}\intop_\Omega|\divv v|^2|v|^{r-2}dx\le{c\over\nu^{r/2}} |p(\varrho)-p(a)|_{3r\over r+1}^r+c|p(\varrho)\cr
&\quad-p(a)-\nu_2\divv v|_{3r\over r+1}^r+c|\varrho f|_{3r\over2r+1}^r,\cr}
\label{2.12}
\end{equation}
where $k_0=k_0\big(\mu,r,{\nu_1^2\over\nu}\big)$, $c=c(r,k_0)$.
\end{lemma}

\begin{proof}
We express $(\ref{1.1})_1$ in the form
\begin{equation}
\varrho v_t+\varrho v\cdot\nabla v-\mu\Delta v-\nu_1\nabla\divv v+\nabla(p-p(a)-\nu_2\divv v)=\varrho f,
\label{2.13}
\end{equation}
where $\nu=\nu_1+\nu_2$, $\nu_i>0$, $i=1,2$. Multiplying (\ref{2.13}) by $v|v|^{r-2}$, integrating over $\Omega$, using $(\ref{1.1})_2$ and boundary conditions yields
\begin{equation}\eqal{
&{1\over r}{d\over dt}\intop_\Omega\varrho|v|^rdx+\mu\intop_\Omega\nabla v\cdot\nabla(v|v|^{r-2})dx+\nu_1\intop_\Omega\divv v\divv(v|v|^{r-2})dx\hskip-.4cm\cr
&\quad-\intop_\Omega(p-p(a)-\nu_2\divv v)\divv(v|v|^{r-2})dx=\intop_\Omega\varrho fv|v|^{r-2}dx.\cr
&\cr}
\label{2.14}
\end{equation}
First, we consider
$$\eqal{
J_1&=\mu\intop_\Omega\nabla v\cdot\nabla(v|v|^{r-2})dx\cr
&=\mu\intop_\Omega|\nabla v|^2|v|^{r-2}dx+\mu\intop_\Omega v_k\nabla v_k\cdot(r-2)|v|^{r-3}\nabla|v|dx\equiv I_1+I_2.}
$$
Using that $v_k\nabla v_k={1\over2}\nabla|v|^2=|v|\nabla|v|$ we have
$$\eqal{
I_2&=(r-2)\mu\intop_\Omega|v|^{r-2}|\nabla|v||^2dx=(r-2)\mu\intop_\Omega ||v|^{{r\over2}-1}\nabla|v||^2dx\cr
&={4(r-2)\over r^2}\mu\intop_\Omega|\nabla|v|^{r/2}|^2dx.\cr}
$$
To examine $I_1$ we use the formula
\begin{equation}
|\nabla u|^2=|u|^2\bigg|\nabla{u\over|u|}\bigg|^2+\big|\nabla|u|\big|^2.
\label{2.15}
\end{equation}
Then $I_1$ takes the form
$$\eqal{
I_1&=\mu\intop_\Omega\bigg(|v|^2\bigg|\nabla{v\over|v|}\bigg|^2+ \big|\nabla|v|\big|^2\bigg)|v|^{r-2}dx\cr
&=\mu\intop_\Omega\bigg[|v|^r\bigg|\nabla{v\over|v|}\bigg|^2+|v|^{r-2} \big|\nabla|v|\big|^2\bigg]dx\cr
&=\mu\intop_\Omega\bigg[|v|^r\bigg|\nabla{v\over|v|}\bigg|^2+ \big||v|^{{r\over2}-1}\nabla|v|\big|^2\bigg]dx\cr
&=\mu\intop_\Omega|v|^r\bigg|\nabla{v\over|v|}\bigg|^2dx+{4\mu\over r^2} \intop_\Omega\big|\nabla|v|^{r/2}\big|^2dx.\cr}
$$
Hence
$$
J_1={4(r-1)\mu\over r^2}\intop_\Omega\big|\nabla|v|^{r/2}\big|^2dx+\mu\intop_\Omega |v|^r\bigg|\nabla{v\over|v|}\bigg|^2dx.
$$
Next, we consider
$$\eqal{
J_2&=\nu_1\intop_\Omega\divv v\divv(v|v|^{r-2})dx\cr
&=\nu_1\intop_\Omega|\divv v|^2|v|^{r-2}dx+\nu_1\intop_\Omega\divv vv\cdot\nabla|v|^{r-2}dx\equiv I_3+I_4,\cr}
$$
where
$$
I_4=\nu_1(r-2)\intop_\Omega\divv vv\cdot\nabla|v|\,|v|^{r-3}dx.
$$
Then
$$
|I_4|\le\nu_1(r-2)\intop_\Omega|\divv v|\,|v|^{r-2}|\nabla|v||dx.
$$
Employing the above expressions in (\ref{2.14}) one gets
\begin{equation}\eqal{
&{1\over r}{d\over dt}\intop_\Omega\varrho|v|^rdx+{4(r-1)\mu\over r^2}\intop_\Omega\big|\nabla|v|^{r/2}\big|^2dx+ \mu\intop_\Omega|v|^r \bigg|\nabla{v\over|v|}\bigg|^2dx\cr
&\quad+\nu_1\intop_\Omega|\divv v|^2|v|^{r-2}dx\le\nu_1(r-2)\intop_\Omega|\divv v|\,|v|^{r-2}\big|\nabla|v|\big|dx\cr
&\quad+\intop_\Omega[p(\varrho)-p(a)-\nu_2\divv v][\divv v|v|^{r-2}+(r-2)|v|^{r-3}v\cdot\nabla|v|]dx\cr
&\quad+\bigg|\intop_\Omega\varrho f\cdot v|v|^{r-2}dx\bigg|.\cr}
\label{2.16}
\end{equation}
The first term on the r.h.s. of (\ref{2.16}) is bounded by
$$
{\varepsilon_1\over2}\nu_1\intop_\Omega|\divv v|^2|v|^{r-2}dx+{1\over2\varepsilon_1}\nu_1(r-2)^2\intop_\Omega|v|^{r-2} \big|\nabla|v|\big|^2dx,
$$
where the second integral equals
$$
{\nu_1\over2\varepsilon_1}(r-2)^2\intop_\Omega\big||v|^{{r\over2}-1}\nabla|v|\big|^2dx ={4\nu_1(r-2)^2\over2\varepsilon_1r^2}\intop_\Omega|\nabla|v|^{r/2}|^2dx.
$$
The second term on the r.h.s. of (\ref{2.16}) can be expressed in the form
$$\eqal{
&\intop_\Omega(p(\varrho)-p(a))\divv v|v|^{r-2}dx-\nu_2\intop_\Omega|\divv v|^2 |v|^{r-2}dx\cr
&\quad+\intop_\Omega[p(\varrho)-p(a)-\nu_2\divv v](r-2)|v|^{r-3}v\cdot\nabla |v|dx\equiv J_1+J_2+J_3,\cr}
$$
where
$$
|J_1|\le{\varepsilon_2\over2}\intop_\Omega|\divv v|^2|v|^{r-2}dx+ {1\over2\varepsilon_2}\intop_\Omega|p(\varrho)-p(a)|^2|v|^{r-2}dx
$$
and
$$\eqal{
|J_3|&\le(r-2)\intop_\Omega|p(\varrho)-p(a)-\nu_2\divv v|\,|v|^{r-2}|\nabla |v||dx\cr
&\le{\varepsilon_3\over2}(r-2)\intop_\Omega|v|^{r-2}|\nabla|v||^2dx\cr
&\quad+{r-2\over2\varepsilon_3}\intop_\Omega|p(\varrho)-p(a)-\nu_2\divv v|^2|v|^{r-2}dx\equiv K_1.\cr}
$$
Hence we have
$$\eqal{
K_1&=2\varepsilon_3{r-2\over r^2}\intop_\Omega\big|\nabla|v|^{r/2}\big|^2dx\cr
&\quad+{r-2\over2\varepsilon_3}\intop_\Omega|p(\varrho)-p(a)-\nu_2\divv v|^2|v|^{r-2}dx.\cr}
$$
Employing the above estimates in (\ref{2.16}) yields
\begin{equation}\eqal{
&{1\over r}{d\over dt}\intop_\Omega\varrho|v|^rdx+\bigg[{4(r-1)\mu\over r^2}- {4\nu_1(r-2)^2\over2\varepsilon_1r^2}-{2\varepsilon_3(r-2)\over r^2}\bigg] \intop_\Omega\big|\nabla|v|^{r/2}\big|^2dx\cr
&\quad+\mu\intop_\Omega|v|^r \bigg|\nabla{v\over|v|}\bigg|^2dx+\!\bigg[\nu-{\nu_1\varepsilon_1\over2}- {\varepsilon_2\over2}\bigg] \intop_\Omega|\divv v|^2|v|^{r-2}dx\cr
&\le{1\over2\varepsilon_2}\intop_\Omega|p(\varrho)-p(a)|^2|v|^{r-2}dx+ {r-2\over2\varepsilon_3}\intop_\Omega|p(\varrho)-p(a)-\nu_2\divv v|^2 |v|^{r-2}dx\cr
&\quad+\intop_\Omega\varrho f\cdot v|v|^{r-2}dx.\cr}
\label{2.17}
\end{equation}
We set
\begin{equation}
\varepsilon_1={\nu\over2\nu_1},\quad \varepsilon_2={\nu\over2},\quad \varepsilon_3={\mu\over r-2}
\label{2.18}
\end{equation}
then coefficient near the second term on the l.h.s. of (\ref{2.17}) equals
\begin{equation}\eqal{
&{4(r-1)\mu\over r^2}-{4\nu_1^2(r-2)^2\over\nu r^2}-{2\mu\over r^2}= {2\mu(2r-3)\over r^2}-{4\nu_1^2(r-2)^2\over\nu r^2}\cr
&={4\over r^2}\bigg[{(2r-3)\mu\over2}-{\nu_1^2\over\nu}(r-2)^2\bigg]\equiv k_0\bigg(\mu,r,{\nu_1^2\over\nu}\bigg)}
\label{2.19}
\end{equation}
which is positive for $r\ge2$ and $\nu_1^2/\nu$ small.

\noindent
Coefficients near the third and fourth terms are equal, respectively,
$$
\mu\quad {\rm and}\quad {\nu\over 2}.
$$
To have $k_0>0$ we obtain the restriction on $\nu_1$
$$
{\nu_1^2\over\nu}<{\mu(2r-3)\over 2(r-2)^2}\equiv d^2,\quad {\rm so}\quad \nu_1<d\sqrt{\nu}
$$
Then
$$
\nu_2>\nu-d\sqrt{\nu}.
$$
Hence for $\nu$ large $\nu_2$ is close to $\nu$

\noindent
Then (\ref{2.17}) takes the form
\begin{equation}\eqal{
&{1\over r}{d\over dt}\intop_\Omega\varrho|v|^rdx+k_0\intop_\Omega\big|\nabla|v|^{r/2}\big|^2dx+ \mu\intop_\Omega|v|^r\bigg|\nabla{v\over|v|}\bigg|^2dx\cr
&\quad+{\nu\over2}\intop_\Omega|\divv v|^2|v|^{r-2}dx\le{1\over\nu}\intop_\Omega |p(\varrho)-p(a)|^2|v|^{r-2}dx\cr
&\quad+{(r-2)^2\over2\mu}\intop_\Omega|p(\varrho)-p(a)-\nu_2\divv v|^2|v|^{r-2}dx+\intop_\Omega|\varrho f\cdot v|\,|v|^{r-2}dx.\cr}
\label{2.20}
\end{equation}
First we use the estimate
$$
\intop_\Omega|\nabla|v|^{r/2}|^2dx\ge c_0|v|_{3r}^r.
$$
Next, we have
$$\eqal{
&{1\over\nu}\intop_\Omega|p(\varrho)-p(a)|^2|v|^{r-2}dx\le{1\over\nu} |p(\varrho)-p(a)|_{3r\over r+1}^2|v|_{3r}^{r-2}\cr
&\le\bigg[{\varepsilon_4^{r\over r-2}\over{r\over r-2}}|v|_{3r}^r+{1\over{r\over2}\varepsilon_4^{r/2}\nu^{r/2}}|p(\varrho)- p(a)|_{3r\over r+1}^r\bigg]\equiv L_1,\cr}
$$
$$\eqal{
&{(r-2)^2\over2\mu}\intop_\Omega|p(\varrho)-p(a)-\nu_2\divv v|^2|v|^{r-2}dx\cr
&\le {(r-2)^2\over2\mu}|p(\varrho)-p(a)-\nu_2\divv v|_{3r\over r+1}^2 |v|_{3r}^{r-2}\cr
&\le{(r-2)^2\over2\mu}\bigg[{\varepsilon_5^{r\over r-2}\over{r\over r-2}}|v|_{3r}^r+{1\over{r\over 2}\varepsilon_5^{r/2}}|p(\varrho)-p(a)-\nu_2\divv v|_{3r\over r+1}^r\bigg]\equiv L_2,\cr}
$$
$$\eqal{
&\bigg|\intop_\Omega\varrho f\cdot v|v|^{r-2}dx\bigg|\le\intop_\Omega|\varrho f|\,|v|^{r-1}dx\le|\varrho f|_{3r\over2r+1}|v|_{3r}^{r-1}\cr
&\le{\varepsilon_6^{r/(r-1)}\over r/(r-1)}|v|_{3r}^r+{1\over r\varepsilon_6^r}|\varrho f|_{3r\over2r+1}^r\equiv L_3.\cr}
$$
We set
$$
{\varepsilon_4^{r/(r-2)}\over r/(r-2)}{1\over\nu}={k_0c_0\over 6},\quad {(r-2)^2\over2\mu}{\varepsilon_5^{r/(r-2)}\over r/(r-2)}={k_0c_0\over6},\quad {\varepsilon_6^{r/(r-1)}\over r/(r-1)}={k_0c_0\over 6}
$$
Then
$$\eqal{
&L_1\le c_0{k_0\over 6}|v|_{3r}^r+{c(r,k_0)\over\nu^{r/2}}|p(\varrho)-p(a)|_{3r\over r+1}^r,\cr
&L_2\le c_0{k_0\over6}|v|_{3r}^r+c(r,k_0)|p(\varrho)-p(a)-\nu_2\divv v|_{3r\over r+1}^r,\cr
&L_3\le c_0{k_0\over6}|v|_{3r}^2+c(r,k_0)|\varrho f|_{3r\over2r+1}^r.\cr}
$$
Employing the estimates in (\ref{2.20}) implies
\begin{equation}\eqal{
&{1\over r}{d\over dt}\intop_\Omega\varrho|v|^rdx+{k_0\over2}\intop_\Omega |\nabla|v|^{r/2}|^2dx+\mu\intop_\Omega|v|^r \bigg|\nabla{v\over|v|}\bigg|^2dx\cr
&\quad+{\nu\over2}\intop_\Omega|\divv v|^2|v|^{r-2}dx\le{c\over\nu^{r/2}} |p(\varrho)-p(a)|_{3r\over r+1}^r\cr
&\quad+c|p(\varrho)-p(a)-\nu_2\divv v|_{3r\over r+1}^r+c|\varrho f|_{3r\over2r+1}^r.\cr}
\label{2.21}
\end{equation}
This inequality implies (\ref{2.12}) and concludes the proof.
\end{proof}

\begin{remark}\label{r2.1}
We integrate (\ref{2.12}) with respect to time. Then we get
\begin{equation}\eqal{
&{1\over r}\intop_\Omega\varrho|v|^rdx+{k_0\over2}\intop_{\Omega^t} \big|\nabla|v|^{r/2}\big|^2dxdt'+\mu\intop_{\Omega^t}|v|^r \bigg|\nabla{v\over|v|}\bigg|^2dxdt'\cr
&\quad+{\nu\over2}\intop_{\Omega^t}|\divv v|^2|v|^{r-2}dxdt'\le {c\over\nu^{r/2}}\intop_0^t|p(\varrho)-p(a)|_{3r\over r+1}^rdt'\cr
&\quad+c\intop_0^t|p(\varrho)-p(a)-\nu_2\divv v|_{3r\over r+1}^rdt'+c\intop_0^t |\varrho f|_{3r\over2r+1}^rdt'+{1\over r}\intop_\Omega\varrho_0|v_0|^rdx\cr
&\equiv A_{2,r}^r(t).\cr}
\label{2.22}
\end{equation}
Since $\varrho=a+\eta$, we obtain from (\ref{2.22}) for $\eta$ so small that $|\eta|<a/2$ the inequality
\begin{equation}
|v(t)|_r\le cA_{2,r}(t).
\label{2.23}
\end{equation}
Since $\varrho$ is bounded from below and above and since $\nu_2$ is close to $\nu$ we obtain from (\ref{2.22}) the inequality
\begin{equation}\eqal{
&{1\over r}|v|_r^r+{k_0\over2}\intop_{\Omega^t}|\nabla|v|^{r/2}|^2dxdt'\le c\intop_0^t|\eta|_{3r\over r+1}^rdt'+c\nu^r\intop_0^t|\Delta\varphi|_{3r\over r+1}^rdt'\cr
&\quad+c\intop_0^t|f|_{3r\over2r+1}^rdt'+{1\over r} |\varrho_0|_\infty|v_0|_r^r.\cr}
\label{2.24}
\end{equation}
Simplifying, we get
\begin{equation}\eqal{
&{1\over\root{r}\of{r}}|v|_r+\root{r}\of{k_0\over2} \bigg(\intop_{\Omega^t}|\nabla|v|^{r/2}|^2dxdt'\bigg)^{1/r}\le c|\eta|_{{3r\over r+1},r,\Omega^t}+c\nu|\Delta\varphi|_{{3r\over r+1},r,\Omega^t}\cr
&\quad+c|f|_{{3r\over2r+1},r,\Omega^t}+{1\over\root{r}\of{r}} |\varrho_0|_\infty^{1/r}|v_0|_r.\cr}
\label{2.25}
\end{equation}
Since $\Delta\varphi\in W_2^{2,1}(\Omega^t)$ then $\Delta\varphi\in L_{{3r\over r+1},r}(\Omega^t)$ with arbitrary $r$, because
$$
{5\over2}-{3(r+1)\over3r}-{2\over r}\le 2\quad {\rm so}\quad {3\over2}-{3\over r}\le 2\ {\rm which\ holds\ for\ any}\ r.
$$
To derive a global estimate for solutions to problem (\ref{1.1}) we need that the first two terms on the r.h.s. of (\ref{2.25}) are estimated by quantities multiplied by the small parameter $\big({1\over\nu}\big)^\alpha$, $\alpha>0$.

\noindent
However, the second term on the r.h.s. of (\ref{2.25}) contains the coefficient $\nu$. This means that we need more delicate estimate to get the factor $\big({1\over\nu}\big)^\alpha$, $\alpha>0$.

\noindent
For this purpose we consider the interpolation
\begin{equation}
|\Delta\varphi|_{3r\over r+1}\le c|\nabla\Delta\varphi|_2^\theta|\nabla\varphi|_2^{1-\theta},
\label{2.26}
\end{equation}
where $\theta$ is a solution to the equation
$$
{3(r+1)\over3r}-1={3\over2}-2\theta.
$$
Then $\theta={3\over4}-1/2r$, $1-\theta={1\over4}+{1\over2r}$. Therefore
\begin{equation}
|\Delta\varphi|_{{3r\over r+1},r,\Omega^t}\le c|\nabla\Delta\varphi|_{2,\infty,\Omega^t}^{{3\over4}-1/2r}\bigg(\intop_0^t |\nabla\varphi|_2^{(1/4+1/2r)r}dt\bigg)^{1/r},
\label{2.27}
\end{equation}
where $(1/4+1/2r)r\le 2$ for $r\le 6$. For $r=6$, (\ref{2.27}) takes the form
\begin{equation}
|\Delta\varphi|_{18/7,6,\Omega^t}\le c|\nabla\Delta\varphi|_{2,\infty,\Omega^t}^{2/3} |\nabla\varphi|_{2,\Omega^t}^{1/3}.
\label{2.28}
\end{equation}
To estimate the last factor on the r.h.s. of (\ref{2.28}) we need the following equation derived from $(\ref{1.1})_1$ by applying the div operator
\begin{equation}\eqal{
&a\Delta\varphi_t-(\mu+\nu)\Delta^2\varphi+a_0\Delta\eta=-a\divv(v\cdot\nabla v)+ \divv[-\eta v_t-\eta v\cdot\nabla v\cr
&\quad+(p_\varrho(a)-p_\varrho(a+\eta))\nabla\eta+(a+\eta)f].\cr}
\label{2.29}
\end{equation}
Applying operator $\Delta^{-1}$ to (\ref{2.29}) yields
\begin{equation}\eqal{
&a\varphi_t-(\mu+\nu)\Delta\varphi=-a\Delta^{-1}\partial_{x_i}\partial_{x_j} (v_iv_j)+a\Delta^{-1}\partial_{x_i}(\Delta\varphi v_i)\cr
&\quad+\Delta^{-1}\divv[-\eta v_t-\eta v\cdot\nabla v+(p_\varrho(a)-p_\varrho(a+\eta))\nabla\eta+(a+\eta)f]\cr
&\quad-(\eta-\diagintop_\Omega\eta dx)+a\diagintop_\Omega\varphi_tdx\cr
&\equiv D_1+D_2+F-(\eta-\diagintop_\Omega\eta dx)+ a\diagintop_\Omega\varphi_tdx,\cr}
\label{2.30}
\end{equation}
where $\diagintop_\Omega={1\over|\Omega|}\intop$.
\end{remark}

Integrating (\ref{2.30}) over $\Omega$ we obtain identity in view of the periodic boundary conditions.

To obtain an estimate of the term $\nu|\Delta\varphi|_{{3r\over r+1},r,\Omega^t}$, which appears on the r.h.s. of (\ref{2.24}), in terms of the function $\Psi^\alpha/\nu^\beta$, where $\alpha>0$, $\beta>0$, $\Psi=\nu|\nabla\varphi|_{3,1,2,\Omega^t}$, we need the result

\begin{lemma}\label{l2.3}
Let the assumptions of Lemma \ref{l2.1} hold. Let $A_1$ be defined in Lemma \ref{l2.1}. Let $|\eta|\le a/2$. Let $0<\varphi_*=\min_\Omega\varphi$. Let $v\in L_{2p/(p-2),2/(1-x)}(\Omega^t)$, $\Delta\varphi\in L_{3/2,\infty}(\Omega^t)$, $\eta\in L_{6,\infty}(\Omega^t)$, $\nabla\eta\in L_2(\Omega^t)$, $v_t\in L_2(\Omega^t)$, $v\in L_{6,\infty}(\Omega^t)$, $f_g\in L_{6/5,2}(\Omega^t)$, $f\in L_2(\Omega^t)$, $\varphi(0)\in L_2(\Omega)$, $p\in(2,6)$, $\varkappa={3\over2}-{3\over p}$.\\
Then
\begin{equation}\eqal{
&a|\varphi(t)|_2^2+(\mu+\nu)|\nabla\varphi|_{2,\Omega^t}^2\le\exp\bigg( {c|v|_{2p/(p-2),2/(1-\varkappa),\Omega^t}^{2/(1-\varkappa)}\over [(\mu+\nu)^\varkappa\varphi_*]^{1/(1-\varkappa)}}\cr
&\quad+{\intop_0^t|\diagintop_\Omega\varphi_tdx|dt'\over\varphi_*}\bigg)\bigg[ {c\over\mu+\nu}(|\Delta\varphi|_{3/2,\infty,\Omega^t}^2A_1^2+ |\eta|_{3,\infty,\Omega^t}^2|v_t|_{2,\Omega^t}^2\cr
&\quad+|\eta|_{6,\infty,\Omega^t}^2|v|_{6,\infty,\Omega^t}^2A_1^2+ |\eta|_{6,\infty,\Omega^t}^2|\nabla\eta|_{2,\Omega^t}^2+|f_g|_{6/5,2,\Omega^t}^2+ |\eta|_{2,\Omega^t}^2\cr
&\quad+|\eta|_{3,\infty,\Omega^t}^2|f|_{2,\Omega^t}^2)+ a|\varphi(0)|_2^2\bigg].\cr}
\label{2.31}
\end{equation}
\end{lemma}

\begin{proof}
Multiplying (\ref{2.30}) by $\varphi$ and integrating over $\Omega$ yields
\begin{equation}\eqal{
&{a\over2}{d\over dt}|\varphi|_2^2+(\mu+\nu)|\nabla\varphi|_2^2=\intop_\Omega D_1\varphi dx+\intop_\Omega D_2\varphi dx\cr
&\quad+\intop_\Omega F\varphi dx-\intop_\Omega\bar\eta\varphi dx+a\diagintop_\Omega\varphi_tdx\intop_\Omega\varphi dx,\cr}
\label{2.32}
\end{equation}
where $\bar\eta=\eta-\diagintop_\Omega\eta dx$.

\noindent
Now we estimate the particular terms from the r.h.s. of (\ref{2.32}). The first term is bounded by
\begin{equation}
\bigg|\intop_\Omega D_1\varphi dx\bigg|\le\intop_\Omega {|D_1|\varphi^2\over\varphi_*}dx\le{|D_1|_{p/(p-2)}\over\varphi_*}|\varphi|_p^2 \equiv I_1,
\label{2.33}
\end{equation}
where $2<p<6$. Let $\alpha={|D_1|_{p/(p-2)}\over\varphi_*}$. Then we use the interpolation
$$
I_1^{1/2}=\alpha^{1/2}|\varphi|_p\le\alpha^{1/2}(\varepsilon^{1/\varkappa} |\nabla\varphi|_2+c\varepsilon^{-{1\over1-\varkappa}}|\varphi|_2),
$$
where $\varkappa=3/2-3/p$.

\noindent
Setting $\varepsilon^{1/\varkappa}\alpha^{1/2}=(\mu+\nu)^{1/2}$ we have $\varepsilon=\big({\mu+\nu\over\alpha}\big)^{\varkappa/2}$.

\noindent
Then
$$
\alpha^{1/2}\varepsilon^{-1/(1-\varkappa)}={\alpha^{1/2(1-\varkappa)}\over (\mu+\nu)^{\varkappa/2(1-\varkappa)}}
$$
Therefore
\begin{equation}\eqal{
I_1&\le{1\over2}(\mu+\nu)|\nabla\varphi|_2^2+{c\alpha^{1/(1-\varkappa)}\over (\mu+\nu)^{\varkappa/(1-\varkappa)}}|\varphi|_2^2\cr
&={1\over2}(\mu+\nu)|\nabla\varphi|_2^2+{c|D_1|_{p/(p-2)}^{1/(1-\varkappa)}\over ((\mu+\nu)^\varkappa\varphi_*)^{1/(1-\varkappa)}}|\varphi|_2^2,\cr}
\label{2.34}
\end{equation}
where $|D_1|_q\le c\sum_{i,j=1}^3|v_lv_j|_q$ for any $q\in(1,\infty)$.

\noindent
Consider the second term on the r.h.s. of (\ref{2.32}). Integration by parts yields
$$
\intop_\Omega D_2\varphi dx=-\intop_\Omega\Delta^{-1}(\Delta\varphi v)\cdot\nabla\varphi dx\equiv I_2.
$$
Hence
\begin{equation}\eqal{
|I_2|&\le\varepsilon_1|\nabla\varphi|_2^2+c/\varepsilon_1|\Delta^{-1} (\Delta\varphi v)|_2^2\le\varepsilon_1|\nabla\varphi|_2^2+c/\varepsilon_1|\Delta\varphi v|_{6/5}^2\cr
&\le\varepsilon_1|\nabla\varphi|_2^2+c/\varepsilon_1|\Delta\varphi|_{3/2}^2 |v|_6^2.\cr}
\label{2.35}
\end{equation}
Consider the third term on the r.h.s. of (\ref{2.32}). Using that $F=\divv F'$ we get
$$
\bigg|\intop_\Omega F\cdot\varphi dx\bigg|=\bigg|\intop_\Omega F'\cdot\nabla\varphi\bigg|dx\le\varepsilon_2|\nabla\varphi|_2^2+c/\varepsilon_2 |F'|_2^2,
$$
where
\begin{equation}\eqal{
|F'|_2^2&=|\Delta^{-1}[-\eta v_t-\eta v\cdot\nabla v+(p_\varrho(a)-p_\varrho(a+\eta))\nabla\eta+af_g+\eta f]|_2^2\cr
&\le c(|\eta v_t|_{6/5}^2+|\eta v\cdot\nabla v|_{6/5}^2+|\eta\nabla\eta|_{6/5}^2+|f_g|_{6/5}^2+|\eta f|_{6/5}^2)\cr
&\le c(|\eta|_3^2|v_t|_2^2+|\eta|_6^2|v|_6^2|\nabla v|_2^2+|\eta|_6^2|\nabla\eta|_2^2+|f_g|_{6/5}^2+|\eta|_3^2|f|_2^2).\cr}
\label{2.36}
\end{equation}
We express the fourth term on the r.h.s. of (\ref{2.32}) in the form
$$
\intop_\Omega\bar\eta\varphi dx=\intop_\Omega\bar\eta\bar\varphi dx\equiv I_3.
$$
Hence
$$
|I_3|\le\varepsilon_3|\nabla\varphi|_2^2+c/\varepsilon_3|\bar\eta|_2^2\le \varepsilon_3|\nabla\varphi|_2^2+c/\varepsilon_3|\eta|_2^2.
$$
Finally, the last term on the r.h.s. of (\ref{2.32}) is bounded by
$$
{\left|\diagintop\varphi_tdx\right|\over\varphi_*}|\varphi|_2^2.
$$
Using the above estimates in (\ref{2.32}) and assuming that $\varepsilon_1-\varepsilon_3$ are sufficiently small we derive the inequality
\begin{equation}\eqal{
&a{d\over dt}|\varphi|_2^2+(\mu+\nu)|\nabla\varphi|_2^2\le\bigg[ {c|D_1|_{p/(p-2)}^{1/(1-\varkappa)}\over [(\mu+\nu)^\varkappa\varphi_*]^{1/(1-\varkappa)}}+ {|\diagintop\varphi_tdx|\over\varphi_*}\bigg]|\varphi|_2^2\cr
&\quad+{c\over\mu+\nu}[|\Delta\varphi|_{3/2}^2|v|_6^2+|F'|_2^2+|\eta|_2^2]\cr
&\equiv cd^2|\varphi|_2^2+{c\over\mu+\nu}[|\Delta\varphi|_{3/2}^2|v|_6^2+ |F'|_2^2+|\eta|_2^2].\cr}
\label{2.37}
\end{equation}
From (\ref{2.37}) we have
\begin{equation}\eqal{
&a{d\over dt}(|\varphi|_2^2\exp\bigg[-c\intop_0^td^2(t')dt'\big]\bigg)+(\mu+\nu) |\nabla\varphi|_2^2\exp\bigg[-c\intop_0^td^2(t')dt'\bigg]\cr
&\le{c\over\mu+\nu}[|\Delta\varphi|_{3/2}^2|v|_6^2+|F'|_2^2+|\eta|_2^2]\exp \bigg[-c\intop_0^td^2(t')dt'\bigg].\cr}
\label{2.38}
\end{equation}
Integrating (\ref{2.38}) with respect to time implies
\begin{equation}\eqal{
&a|\varphi(t)|_2^2+(\mu+\nu)|\nabla\varphi|_{2,\Omega^t}^2\le\exp\bigg[ c\intop_0^td^2(t')dt'\bigg]\cdot\cr
&\quad\cdot\bigg[{c\over\mu+\nu}(|\Delta\varphi|_{3/2,\infty,k\Omega^t}^2A_1^2+ |F'|_{2,\Omega^t}^2+t|\eta|_{2,\infty,\Omega^t}^2)+a|\varphi(0)|_2^2\bigg].\cr}
\label{2.39}
\end{equation}
This inequality implies (\ref{2.31}) and concludes the proof.
\end{proof}

\noindent
Now we obtain bounds of $\varphi$ from below and from above. We follow considerations from \cite[Ch. 2, Sections 5, 6]{LSU}.

\begin{lemma}\label{l2.4}
Let $\bar\varphi=\varphi-\diagintop_\Omega\varphi dx$. Let $\bar\varphi(0)\in L_\infty(\Omega)$. Assume that $\hat k=|\bar\varphi(0)|_\infty$ for $\bar\varphi(0)\ge 1$ and $\hat k=1$ for $\bar\varphi(0)<1$. Assume that $\eta\in L_\infty(\Omega^t)$, $\nabla\eta\in L_6(\Omega^t)$, $v_t\in L_{30/(22-9\varkappa)}(\Omega^t)$, $v\in L_{20/(4-3\varkappa)}(\Omega^t)\cap L_{60/(17-9\varkappa)}(\Omega^t)$, $f_g,f\in L_{30/(22-9\varkappa)}(\Omega^t)$, $\varkappa\in(0,4/3)$, $t\le T$. Then
\begin{equation}\eqal{
&|\bar\varphi|_{\infty,\Omega^t}\le 2\hat k[1+2^{2/\varkappa+1/\varkappa^2} (\beta\gamma)^{1+1/\varkappa}t^{(1+\varkappa)/r_0}{\rm meas}^{(1+\varkappa)/p_0}(\Omega)]\cr
&\equiv\gamma_*,\ \ t\le T,\cr
&{3\over p_0}+{2\over r_0}={3\over2},\quad \beta={c\over(\mu+\nu)^{1/q}},\quad 
{3\over p}+{2\over q}={3\over2},\cr
&\gamma={c\over(\mu+\nu)^{1/q}} (|v|_{20/(4-3\varkappa),\Omega^t}^2+G(\varkappa,t))\equiv {c\over(\mu+\nu)^{1/q}}G_0(\varkappa,t),\cr}
\label{2.40}
\end{equation}
where $G$ is defined in (\ref{2.50}).
\end{lemma}

\begin{proof}
From (\ref{2.30}) we have
\begin{equation}
a\bar\varphi_t-(\mu+\nu)\Delta\bar\varphi=-a\Delta^{-1}\partial_{x_i} \partial_{x_j}(v_iv_j)+a\Delta^{-1}\partial_{x_i}(\Delta\varphi v_i)+F-\bar\eta
\label{2.41}
\end{equation}
Let $\bar\varphi^{(k)}=\max\{\bar\varphi^{(k)}(x,t)-k,0\}$. Multiply (\ref{2.41}) by $\bar\varphi^{(k)}$ and integrate over $\Omega$. Then we have
\begin{equation}\eqal{
&{a\over2}{d\over dt}|\bar\varphi^{(k)}|_2^2+(\mu+\nu)|\nabla\bar\varphi^{(k)}|_2^2\cr
&= -\intop_\Omega\Delta^{-1}\partial_{x_i}\partial_{x_j}(v_iv_j) \bar\varphi^{(k)}dx\cr
&\quad+a\intop_\Omega\Delta^{-1}\partial_{x_i}(\Delta\varphi v_i)\bar\varphi^{(k)}dx+ \intop_\Omega F\bar\varphi^{(k)}dx-\intop_\Omega\bar\eta\bar\varphi^{(k)}dx.\cr}
\label{2.42}
\end{equation}
Assume that
$$
\hat k< k.
$$
Integrating (\ref{2.42}) with respect to time and using $k>\hat k$ gives
\begin{equation}\eqal{
&\|\bar\varphi^{(k)}\|_{V(\Omega^t)}^2\equiv a|\bar\varphi^{(k)}(t)|_2^2+(\mu+\nu)|\nabla\bar\varphi^{(k)}|_{2,\Omega^t}^2\le c|vv|_{p',q',A_k^t(t)}|\bar\varphi^{(k)}|_{p,q,\Omega^t}\cr
&\quad+c(|\Delta\varphi v|_{p',q',A_k^t(t)} +c|F|_{p',q',A_k^t(t)}) |\bar\varphi^{(k)}|_{p,q,\Omega^t}+ |\bar\eta|_{p',q',A_k^t(t)} |\bar\varphi^{(k)}|_{p,q,\Omega^t},\cr}
\label{2.43}
\end{equation}
where ${3\over p}+{2\over q}={3\over2}$ and ${1\over p}+{1\over p'}=1$, ${1\over q}+{1\over q'}=1$ and $A_k(t)=\{x\in\Omega:\bar\varphi(x,t)>k\}$. Then ${3\over p'}+{2\over q'}=7/2$.

\noindent
Now we have to estimate $|\bar\varphi^{(k)}|_{p,q,\Omega^t}$ by the norm on the l.h.s. of (\ref{2.43}), where we have to take under account the coefficient $\mu+\nu$ which is assumed to be large. From \cite[Ch. 2, Sect. 3]{LSU} we have
\begin{equation}
|\bar\varphi^{(k)}|_{p,q,\Omega^t}\le c|\bar\varphi^{(k)}|_{2,\infty,\Omega^t}^{1-2/q} |\nabla\bar\varphi^{(k)}|_{2,\Omega^t}^{2/q}.
\label{2.44}
\end{equation}
Let $\beta$ be a power function of $(\mu+\nu)$. Then we have
$$
\beta|\bar\varphi^{(k)}|_{2,\infty,\Omega^t}^{1-2/q} |\nabla\bar\varphi^{(k)}|_{2,\Omega^t}^{2/q}\le c\beta^{q/2}|\bar\varphi_x^{(k)}|_{2,\Omega^t}+c |\bar\varphi^{(k)}|_{2,\infty,\Omega^t}.
$$
Comparing this with the norm from the l.h.s. of (\ref{2.43}) we have
$$
\beta^{q/2}=(\mu+\nu)^{1/2}\quad {\rm so}\quad \beta=(\mu+\nu)^{1/q}.
$$
Then, (\ref{2.43}) takes the form
\begin{equation}\eqal{
&|\bar\varphi^{(k)}(t)|_{2,\infty,\Omega^t}+(\mu+\nu)^{1/2} |\nabla\bar\varphi^{(k)}|_{2,\Omega^t}\le{c\over(\mu+\nu)^{1/q}} |vv|_{p',q',A_k^t(t)}\cr
&\quad+{1\over(\mu+\nu)^{1/q}}|F|_{p',q',A_k^t(t)}+{1\over(\mu+\nu)^{1/q}} (|\Delta\varphi v|_{p',q',A_k^t(t)}+|\bar\eta|_{p',q',A_k^t(t)}).\cr}
\label{2.45}
\end{equation}
Now, we examine the terms from the r.h.s. of (\ref{2.45}). Examine the first term. Let $h=v\cdot v$. Then we have
$$\eqal{
&{c\over(\mu+\nu)^{1/q}}|h|_{p',q',A_k^t(t)}\equiv
{c\over(\mu+\nu)^{1/q}}\bigg(\intop_0^t\bigg(\intop_{A_k(t')} |h|^{p'}dx\bigg)^{q'/p'}dt'\bigg)^{1/q'}\cr
&\le{c\over(\mu+\nu)^{1/q}}\bigg(\intop_0^t\bigg[\bigg( \intop_{A_k(t')}1dx\bigg)^{1/\lambda p'}\bigg(\intop_{A_k(t')}|h|^{p'\lambda'} \bigg)^{1/p'\lambda'}\bigg]^{q'}dt'\bigg)^{1/q'}\cr
&={c\over(\mu+\nu)^{1/q}}\bigg(\intop_0^t|A_k(t')|^{q'/\lambda p'}
|h|_{p'\lambda',\Omega}^{q'}dt'\bigg)^{1/q'}\cr
&\le{c\over(\mu+\nu)^{1/q}}\bigg(\intop_0^t|A_k(t')|^{\gamma q'/\lambda p'}dt' \bigg)^{1/\gamma q'}\bigg(\intop_0^t|h|_{p'\lambda',\Omega}^{\gamma'q'}dt' \bigg)^{1/\gamma'q'}\equiv I_1,\cr}
$$
where $1/\lambda+1/\lambda'=1$, $1/\gamma+1/\gamma'=1$, 
$|A_k(t)|={\rm meas}A_k(t)$, $\mu(k,t)=\intop_0^t|A_k(t')|^{r_0/p_0}dt'$, ${3\over p_0}+{2\over r_0}={3\over2}$.

\noindent
Let ${r_0\over p_0}={\gamma q'\over\lambda p'}$, 
$1/\gamma q'={1+\varkappa\over r_0}$, where $\varkappa>0$. Then $\gamma q'={r_0\over1+\varkappa}$, $\lambda p'={p_0\over1+\varkappa}$.

\noindent
Since $p'={p_0\over(1+\varkappa)\lambda}$, $q'={r_0\over(1+\varkappa)\gamma}$ we have the following two equations for $p_0$, $r_0$,
\begin{equation}
{3\over p_0}+{2\over r_0}={3\over2},\quad 
{3(1+\varkappa)\lambda\over p_0}+{2(1+\varkappa)\gamma\over r_0}={7\over2}.
\label{2.46}
\end{equation}
Hence, we obtain
\begin{equation}\eqal{
&{2\over r_0}(1+\varkappa)(\gamma-\lambda)={7\over2}-{3\over2}(1+\varkappa)\lambda,\cr
&{3\over p_0}(1+\varkappa)(\lambda-\gamma)={7\over2}-{3\over2}(1+\varkappa)\gamma.\cr}
\label{2.47}
\end{equation}
Consider the case $\gamma=\lambda$. Then $\gamma'=\lambda'$, $\lambda={7\over3(1+\varkappa)}$ and $\lambda'={7\over4-3\varkappa}$. Therefore
$$
I_1\le{c\over(\mu+\nu)^{1/q}}\mu(k,t)^{1+\varkappa\over r_0}\bigg(\intop_0^t |h|_{p'\lambda'}^{q'\lambda'}dt'\bigg)^{1/q'\lambda'}\equiv I_2.
$$
Since $3/p'+2/q'=7/2$ we have that $3/p'\lambda'+2/q'\lambda'={7\over2\lambda'}={4-3\varkappa\over2}$. Let $p_*=p'\lambda'$, $q_*=q'\lambda'$. Then
$$
{3\over p_*}+{2\over q_*}={4-3\varkappa\over 2},
$$
where $0<\varkappa<4/3$. For $p_*=q_*$ we have that $p_*={10\over4-3\varkappa}$. But $h\sim v^2$ so to have $I_2$ bounded we need that
\begin{equation}
v\in L_{20/(4-3\varkappa)}(\Omega^t),\quad {\rm so}\quad I_2\le{c\over(\mu+\nu)^{1/q}}\mu(k,t)^{1+\varkappa\over r_0} |v|_{20/(4-3\varkappa),\Omega^t}^2
\label{2.48}
\end{equation}
Looking for solutions such that $v\in L_\infty(0,t;H^2(\Omega))$ we see that (\ref{2.48}) may hold.

\noindent
Looking for the second term on the r.h.s. of (\ref{2.45}) and using the above considerations we have to find an estimate for
$$
|F|_{10/(4-3\varkappa),\Omega^t}.
$$
Using the form of $F$ we calculate
\begin{equation}\eqal{
&|F|_{10/(4-3\varkappa),\Omega^t}=\bigg(\intop_{\Omega^t}|\Delta^{-1}\divv [-\eta v_t-\eta v\cdot\nabla v\cr
&\quad+(p_\varrho(a)-p_\varrho(a+\eta))\nabla\eta+af_g+\eta f]|^{10/(4-3\varkappa)}dxdt'\bigg)^{(4-3\varkappa)/10}.\cr}
\label{2.49}
\end{equation}
Now, we examine the particular terms from (\ref{2.49})
$$\eqal{
&|\Delta^{-1}\divv(\eta v_t)|_{10/(4-3\varkappa),\Omega^t}
\le c|\eta v_t|_{30/(22-9\varkappa),10/(4-3\varkappa),\Omega^t}\cr
&\le|\eta|_{\infty,\Omega^t}|v_t|_{30/(22-9\varkappa),10/(4-3\varkappa),\Omega^t},\cr
&|\Delta^{-1}\divv(\eta v\cdot\nabla v)|_{10/(4-3\varkappa),\Omega^t}\le |\Delta^{-1}\partial_{x_i}\partial_{x_j}(\eta v_iv_j)|_{10/(4-3\varkappa),\Omega^t}\cr
&\quad+|\Delta^{-1}\partial_{x_j}(\partial_{x_i}\eta v_iv_j)|_{10/(4-3\varkappa),\Omega^t}\cr
&\le c|\eta|_{\infty,\Omega^t}|v^2|_{10/(4-3\varkappa),\Omega^t}+c|\nabla\eta v^2|_{30/(22-9\varkappa),\Omega^t}\cr
&\le c|\eta|_{\infty,\Omega^t}|v|_{20/(4-3\varkappa),\Omega^t}^2+c |\nabla\eta|_{30\lambda_1/(22-9\varkappa),\Omega^t} |v^2|_{30\lambda_2/(22-9\varkappa),\Omega^t}\cr
&\equiv J_1+J_2,\cr}
$$
where $1/\lambda_1+1/\lambda_2=1$.

\noindent
Since
$$
{30\lambda_1\over22-9\varkappa}=6\ \ {\rm we\ have\ that}\ \ \lambda_1={22-9\varkappa\over 5}
$$
so $\lambda_2={22-9\varkappa\over17-9\varkappa}$.

\noindent
Therefore
$$\eqal{
J_2&\le c|\nabla\eta|_{6,\infty,\Omega^t}|v^2|_{30/(17-9\varkappa),10/(4-3\varkappa), \Omega^t}\cr
&\le c|\nabla\eta|_{6,\infty,\Omega^t} |v|_{60/(17-9\varkappa),20/(4-3\varkappa),\Omega^t}^2.\cr}
$$
Continuing,
$$\eqal{
&|\Delta^{-1}\divv(p_\varrho(a)- p_\varrho(a+\eta))\nabla\eta|_{10/(4-3\varkappa),\Omega^t}\le c|\eta\nabla\eta|_{30/(22-9\varkappa),10/(4-3\varkappa),\Omega^t}\cr
&\le c|\eta|_{\infty,\Omega^t} |\nabla\eta|_{30/(22-9\varkappa),10/(4-3\varkappa),\Omega^t},\cr}
$$
where ${30\over22-9\varkappa}\le 6$ so $5\le 22-9\varkappa$. Hence, $\varkappa\le{17\over 9}$. Finally, we have
$$
|\Delta^{-1}\divv f_g|_{10/(4-3\varkappa),\Omega^t}\le c|f_g|_{30/(22-9\varkappa),10/(4-3\varkappa),\Omega^t}
$$
and
$$
|\Delta^{-1}\divv(\eta f)|_{10/(4-3\varkappa),\Omega^t}\le c|\eta|_{\infty,\Omega^t}|f|_{30/(22-9\varkappa),10/(4-3\varkappa),\Omega^t}.
$$
Using the above estimates in (\ref{2.49}) yields
\begin{equation}\eqal{
&|F|_{10/(4-3\varkappa),\Omega^t}\le c[|\eta|_{\infty,\Omega^t} |v_t|_{30/(22-9\varkappa),10/(4-3\varkappa),\Omega^t}\cr
&\quad+|\eta|_{\infty,\Omega^t}|v|_{20/(4-3\varkappa),\Omega^t}^2+ |\nabla\eta|_{6,\infty,\Omega^t}|v|_{60/(17-9\varkappa),20/(4-3\varkappa),\Omega^t}^2\cr
&\quad+|\eta|_{\infty,\Omega^t} |\nabla\eta|_{30/(22-9\varkappa),10/(4-3\varkappa),\Omega^t}+ |f_g|_{30/(22-9\varkappa),10/(4-3\varkappa),\Omega^t}\cr
&\quad+|\eta|_{\infty,\Omega^t} |f|_{30/(22-9\varkappa),10/(4-3\varkappa),\Omega^t}]\cr
&\equiv G(\varkappa,t)\cr}
\label{2.50}
\end{equation}
In view of (\ref{2.50}) the second term on the r.h.s. of (\ref{2.45}) is bounded by
\begin{equation}\eqal{
&{1\over(\mu+\nu)^{1/q}}|F|_{p',q',\Omega^t}\le{c\over(\mu+\nu)^{1/q}}\mu (k,t)^{1+\varkappa\over r_0}G_1(\varkappa,t),\cr
&|\Delta\varphi v|_{10/(4-3\varkappa),\Omega^t}\le |\Delta\varphi|_{20/(4-3\varkappa),\Omega^t} |v|_{20/(4-3\varkappa),\Omega^t} \equiv G_2,\cr
&|\bar\eta|_{10/(4-3\varkappa),\Omega^t}\equiv G_3,\cr
&G=G_1+G_2+G_3.\cr}
\label{2.51}
\end{equation}
Employing estimates (\ref{2.48}) and (\ref{2.51}) in (\ref{2.45}) implies the inequality
\begin{equation}\eqal{
&|\bar\varphi^{(k)}(t)|_{2,\infty,\Omega^t}+(\mu+\nu)^{1/2} |\nabla\bar\varphi^{(k)}(t)|_{2,\Omega^t}\cr
&\le{c\over(\mu+\nu)^{1/q}}\mu(k,t)^{1+\varkappa\over r_0} (|v|_{20/(4-3\varkappa),\Omega^t}^2+G(\varkappa,t)),\cr}
\label{2.52}
\end{equation}
where $q$ and $r_0$ follow from the relations
$$
{3\over p}+{2\over q}={3\over2},\ \ {3\over p_0}+{2\over r_0}={3\over2}\quad {\rm and}\quad 0<\varkappa<4/3.
$$
We apply Lemma 6.1 from \cite[Ch. 2, Sect. 6]{LSU}. Then for $k\ge\hat k$ we obtain from (\ref{2.52}) the inequality
\begin{equation}
\|\bar\varphi^{(k)}\|_{V(\Omega^t)}\le\gamma k\mu^{1+\varkappa\over r_0}(k),
\label{2.53}
\end{equation}
where the norm of $V(\Omega^t)$ is determined by the l.h.s. of (\ref{2.52}) and
$$
\gamma={c\over(\mu+\nu)^{1/q}}(|v|_{20/(4-3\varkappa),\Omega^t}^2+ G(\varkappa,t)),\ \ t\le T.
$$
Moreover, (\ref{2.44}) is used in the form
\begin{equation}
|\bar\varphi^{(k)}|_{p,q,\Omega^t}\le\beta(\mu+\nu)^{1/q} |\bar\varphi^{(k)}|_{2,\infty,\Omega^t}^{1-2/q}
|\nabla\bar\varphi^{(k)}|_{2,\Omega^t}^{2/q}\le\beta \|\bar\varphi^{(k)}\|_{V(\Omega^t)},
\label{2.54}
\end{equation}
where $\beta=c/(\mu+\nu)^{1/q}$ appears in formula (\ref{3.2}) from \cite[Ch. 2, Sect. 3]{LSU}. The $\beta$ appears also in Theorem 6.1 from \cite[Ch. 2, Sect. 6]{LSU}. Then Theorem 6.1 yields the estimate
\begin{equation}
|\bar\varphi|_{\infty,\Omega^t}\le 2\hat k[1+2^{2/\varkappa+1/\varkappa^2} (\beta\gamma)^{1+1/\varkappa}t^{1+\varkappa\over r_0}{\rm meas}^{1+\varkappa\over p_0}\Omega].
\label{2.55}
\end{equation}
This implies (\ref{2.40}) and concludes the proof.
\end{proof}

\begin{remark}\label{r2.2}
To obtain an estimate for $\varphi$ from below we assume that $\varphi>0$. This can be always derived by adding a positive constant $L$ to $\varphi$ because $|\bar\varphi|_\infty\le\gamma_*$ holds with $\gamma_*$ independent of $L$. Then we have
$$\eqal{
&\diagintop_\Omega\varphi dx=\bigg|\diagintop_\Omega\varphi dx\bigg|= \bigg|\diagintop_\Omega\varphi dx-\varphi+\varphi\bigg|\le \bigg|\diagintop_\Omega\varphi dx-\varphi\bigg|+|\varphi|\cr
&\le\gamma_*+|\varphi|=\gamma_*+\varphi.\cr}
$$
Hence we have
\begin{equation}
0<\varphi_*\le\diagintop_\Omega\varphi dx-\gamma_*\le\varphi,
\label{2.56}
\end{equation}
where positiveness follows from the above explanation.
\end{remark}

\noindent
The fact that $\varphi$ is defined up to an arbitrary constant, say $L$, is connected with the considered periodic boundary conditions. Therefore, we have some freedom with determining the magnitude of $\varphi$.

\begin{remark}\label{r2.3}
To estimate the second term on the r.h.s. of (\ref{2.25}) we need (\ref{2.28}). Then we examine
\begin{equation}
\nu|\Delta\varphi|_{18/7,6,\Omega^t}\le c\nu^{2/3}|\nabla\Delta\varphi|_{2,\infty,\Omega^t}^{2/3}\nu^{1/3} |\nabla\varphi|_{2,\Omega^t}^{1/3}\equiv I.
\label{2.57}
\end{equation}
Our aim is the following estimate for $I$
\begin{equation}
I\le c{\Psi^\alpha\over\nu^\beta},
\label{2.58}
\end{equation}
where $\alpha$, $\beta$ are positive numbers and $\Psi=\nu|\nabla\varphi|_{3,1,2,\Omega^t}$.
\end{remark}

\noindent
Hence
\begin{equation}
I\le c\Psi^{2/3}\nu^{1/3}|\nabla\varphi|_{2,\Omega^t}^{1/3}\equiv I_1.
\label{2.59}
\end{equation}
To derive the bound (\ref{2.58}) for $I_1$ in the case of large $\nu$ we need to estimate $|\nabla\varphi|_{2,\Omega^t}$. For this purpose we use (\ref{2.31}). To derive bound (\ref{2.58}) from (\ref{2.31}) we need to know that the coefficient with exponent is independent of $\nu$.

\noindent
For this purpose we assume that there exist positive constants $c_1$, $c_2$, $c_1<c_2$ such that
\begin{equation}
c_1\le(\mu+\nu)^\varkappa\varphi_*\le c_2,
\label{2.60}
\end{equation}
where $\varkappa=3/2-3/p$, $2<p<6$.

\noindent
Using (\ref{2.56}) the second component under exponent in the r.h.s. of (\ref{2.31}) equals
$$
I_2={\intop_0^t\left|{d\over dt'}\intop_\Omega\varphi dx\right|\over\varphi_*}.
$$
Let $\varphi=\varphi'+L$, where $\diagintop_\Omega\varphi'dx=0$ and $L=\const$.
Then $I_2=0$.

\noindent
Then (\ref{2.39}), which is a simpler version of (\ref{2.31}), takes the form
\begin{equation}\eqal{
&a|\varphi(t)|_2^2+(\mu+\nu)|\nabla\varphi|_{2,\Omega^t}^2\le c\exp(c|v|_{2p/(p-2),2/(1-\varkappa),\Omega^t}^{2/(1-\varkappa)})\cdot\cr
&\quad\cdot\bigg[{c\over\mu+\nu}(A_1^2|\Delta\varphi|_{3/2,\infty,\Omega^t}^2+ |F'|_{2,\Omega^t}^2+t|\eta|_{2,\infty,\Omega^t}^2)\cr
&\quad+a|\varphi(0)|_2^2\bigg].\cr}
\label{2.61}
\end{equation}
From the problem for $\eta$
$$
\eta_t+v\cdot\nabla\eta=-a\Delta\varphi-\eta\Delta\varphi,\quad \eta|_{t=0}=\eta(0),
$$
we have
\begin{equation}\eqal{
|\eta(t)|_r&\le\exp(c|\Delta\varphi|_{\infty,1,\Omega^t}) \bigg(\intop_0^t|\Delta\varphi|_rdt'+|\eta(0)|_r\bigg)\cr
&\le\exp\bigg(ct^{1/2}{\Psi\over\nu}\bigg)\bigg(t^{1/2}{\Psi\over\nu}+ |\eta(0)|_r\bigg),\quad r\le\infty.\cr}
\label{2.62}
\end{equation}
Assuming that
\begin{equation}
|\eta(0)|_r\le{c_3\over\nu}
\label{2.63}
\end{equation}
we obtain that
\begin{equation}
|\eta(t)|_r\le{c(t)\over\nu}(\Psi+c_3).
\label{2.64}
\end{equation}
In view of the assumptions of Lemma \ref{l2.4} we have
\begin{equation}
|F'|_{2,\Omega^t}\le{c(t)\over\nu}(\Psi+c_3).
\label{2.65}
\end{equation}
Then (\ref{2.61}) implies
\begin{equation}
|\nabla\varphi|_{2,\Omega^t}\le{c(t)(1+t)^{1/2}\over(\mu+\nu)\nu}(\Psi+c_3)+ {c\over(\mu+\nu)^{1/2}}|\varphi(0)|_2.
\label{2.66}
\end{equation}
Inserting estimate (\ref{2.66}) in (\ref{2.57}) yields
\begin{equation}\eqal{
I&\le c\Psi^{2/3}\bigg[\nu^{1/3}\bigg[(1+t)^{1/2}\bigg({\Psi\over\nu^2}+ {c_3\over\nu^2}\bigg)\bigg]^{1/3}\cr
&\quad+\nu^{1/3}{1\over\nu^{1/6}}|\varphi(0)|_2^{1/3}\bigg]\cr
&\le c(1+t)^{1/6}\bigg[{\Psi\over\nu^{1/3}}+{\Psi^{2/3}c_3^{1/3}\over\nu^{1/3}}\bigg] +c\Psi^{2/3}\nu^{1/6}|\varphi(0)|_2^{1/3}.\cr}
\label{2.67}
\end{equation}
We see that (\ref{2.67}) does not have from (\ref{2.58}).

\noindent
In view of restriction (\ref{2.60}) we can assume that
\begin{equation}
\nu^\varkappa|\varphi(0)|_2\le c_4.
\label{2.68}
\end{equation}
The restriction is compatible with (\ref{2.60}).

\noindent
In view of (\ref{2.68}) the last element on the r.h.s. of (\ref{2.67}) is estimated in the following way
\begin{equation}
c\Psi^{2/3}\nu^{1/6}|\varphi(0)|_2^{1/3}\le c\Psi^{2/3}\nu^{1/6-\varkappa/3} (\nu^\varkappa|\varphi(0)|_2)^{1/3}= {c\Psi^{2/3}c_5^{1/3}\over\nu^{\varkappa/3-1/6}}.
\label{2.69}
\end{equation}
To have estimate (\ref{2.58}) we need
$$
{\varkappa\over3}-{1\over6}={1\over3}(\varkappa-1/2)>0.
$$
Hence
\begin{equation}
\varkappa>1/2\ \ {\rm so}\ \ {3\over2}-{3\over p}>{1\over2}\ \ {\rm implies\ that}\ \ p>3.
\label{2.70}
\end{equation}
Therefore, we can formulate the Corollary.

\noindent
Assume that there exist positive constants $c_1-c_4$ such that
\begin{equation}\eqal{
&c_1\le(\mu+\nu)^\varkappa\varphi_*\le c_2,\cr
&
|\eta(0)|_r\le{c_3\over\nu}, \quad |\varphi(0)|_2\le{c_4\over\nu^\varkappa},\cr}
\label{2.71}
\end{equation}
where $\varkappa=3/2-3/p>1/2$, $3<p<6$.

\noindent
Assume that
$$\eqal{
&v\in L_{2p/(p-2),2/(1-\varkappa)}(\Omega^t),\quad &F'\in L_2(\Omega^t),\cr
&|F'|_{2,\Omega^t}\le{c\over\nu}(\Psi+c_3),\quad &|\eta(t)|_r\le{c\over\nu}(\Psi+c_3).\cr}
$$
Then
\begin{equation}
\nu|\Delta\varphi|_{18/7,6,\Omega^t}\le c\bigg({\Psi\over\nu^{1/3}}+ {\Psi^{2/3}\over\nu^{1/3}}+{\Psi^{2/3}\over\nu^{\varkappa/3-1/6}}\bigg).
\label{2.72}
\end{equation}

\begin{lemma}\label{l2.5}
Let $\eta$ be a solution to (\ref{1.4}). Assume also that $\eta(0)\in L_r(\Omega)$, $\divv v\in L_r(\Omega)\cap L_\infty(\Omega)$, $r\in[1,\infty]$. Then
\begin{equation}
|\eta(t)|_r\le\exp\bigg[\bigg(1-{1\over r}\bigg)\intop_0^t|\divv v(t')|_\infty dt'\bigg]\bigg[\intop_0^ta|\divv v|_rdt'+|\eta(0)|_r\bigg].
\label{2.73}
\end{equation}
\end{lemma}

\begin{proof}
Multiplying $(\ref{1.4})_1$ by $\eta|\eta|^{r-2}$ and integrating over $\Omega$ yields
$$
{1\over r}{d\over dt}|\eta|_r^r+{1\over r}\intop_\Omega v\cdot\nabla|\eta|^rdx+ \intop_\Omega|\eta|^r\divv vdx+\intop_\Omega a\eta|\eta|^{r-2}\divv vdx=0.
$$
Integrating by parts we have
$$
{1\over r}{d\over dt}|\eta|_r^r+\bigg(1-{1\over r}\bigg)\intop_\Omega\divv v|\eta|_r^rdx+a\intop_\Omega\divv v\eta|\eta|^{r-2}dx=0.
$$
Continuing,
$$
{1\over r}{d\over dt}|\eta|_r^r\le\bigg(1-{1\over r}\bigg)|\divv v|_\infty|\eta|_r^r+a|\divv v|_r|\eta|_r^{r-1}.
$$
Simplifying,
$$
{d\over dt}|\eta|_r\le(1-1/r)|\divv v|_\infty|\eta|_r+a|\divv v|_r.
$$
Integrating with respect to time yields (\ref{2.73}). This concludes the proof.
\end{proof}

Next, we obtain estimates for derivatives of $\eta$.

\begin{lemma}\label{l2.6}
Let $\eta$ be a solution to (\ref{1.4}). Let $\nabla\varphi,\rot\psi\in L_1(0,t;H^3(\Omega))$, $\eta(0)\in H^2(\Omega)$. Then
\begin{equation}\eqal{
&\|\eta(t)\|_2\cr
&\le\exp\bigg[c\intop_0^t(\|\nabla\varphi(t')\|_3+\|\rot\psi(t')\|_3) dt'\bigg]\bigg[c\intop_0^t\|\nabla\varphi(t')\|_3dt'+ \|\eta(0)\|_2\bigg].\cr}
\label{2.74}
\end{equation}
\end{lemma}

\begin{proof}
Multiplying $(\ref{1.5})_1$ by $\eta$, integrating over $\Omega$ and by parts we get
$$
{1\over2}{d\over dt}|\eta|_2^2+{1\over2}\intop_\Omega\Delta\varphi\eta^2dx+ a\intop_\Omega\Delta\varphi\eta dx=0.
$$
By the H\"older inequality we have
\begin{equation}
{d\over dt}|\eta|_2\le{1\over2}|\Delta\varphi|_\infty|\eta|_2+a|\Delta\varphi|_2.
\label{2.75}
\end{equation}
Differentiating $(\ref{1.4})_1$ with respect to $x$, multiplying by $\eta_{,x}$ and integrating over $\Omega$ implies
$$\eqal{
&{1\over2}{d\over dt}|\eta_{,x}|_2^2+a\intop_\Omega\Delta\varphi_{,x}\eta_{,x}dx+ {1\over2}\intop_\Omega\Delta\varphi\eta_{,x}^2dx\cr
&\quad+\intop_\Omega v_{,x}\cdot\nabla\eta\eta_{,x}dx+ \intop_\Omega\eta\Delta\varphi_{,x}\eta_{,x}dx.\cr}
$$
By the H\"older inequality we get
\begin{equation}
{d\over dt}|\eta_{,x}|_2\le c(|\rot\psi_{,x}|_\infty+|\nabla^2\varphi|_\infty) |\eta_{,x}|_2+|\Delta\varphi_{,x}|_2+c\|\eta\|_1\|\Delta\varphi_{,x}\|_1.
\label{2.76}
\end{equation}
Finally, we differentiate $(\ref{1.4})_1$ twice with respect to $x$, multiply by $\eta_{,xx}$ and integrate over $\Omega$. Then we have
$$\eqal{
&{1\over2}{d\over dt}|\eta_{,xx}|_2^2\le{1\over2}|\Delta\varphi|_\infty |\eta_{,xx}|_2^2+|v_{,x}|_\infty|\eta_{,xx}|_2^2+|v_{,xx}|_4|\eta_{,x}|_4 |\eta_{,xx}|_2\cr
&\quad+a|\Delta\varphi_{,xx}|_2|\eta_{,xx}|_2+|\Delta\varphi_{,x}|_4|\eta_{,x}|_4 |\eta_{,xx}|_2+|\eta|_\infty|\Delta\varphi_{,xx}|_2|\eta_{,xx}|_2.\cr}
$$
Simplifying we get
\begin{equation}
{d\over dt}|\eta_{,xx}|_2\le c(\|\nabla\varphi\|_3+\|\rot\psi\|_3)\|\eta\|_2 +|\nabla\varphi_{,xxx}|_2.
\label{2.77}
\end{equation}
Adding (\ref{2.75}), (\ref{2.76}) and (\ref{2.77}) yields
\begin{equation}
{d\over dt}\|\eta\|_2\le c(\|\nabla\varphi\|_3+\|\rot\psi\|_3)\|\eta\|_2+ c\|\nabla\varphi\|_3.
\label{2.78}
\end{equation}
From (\ref{2.78}) we have
\begin{equation}\eqal{
&{d\over dt}\bigg[\|\eta\|_2\exp\bigg(-c\intop_0^t(\|\nabla\varphi\|_3+ \|\rot\psi\|_3)dt'\bigg)\bigg]\cr
&\le c\|\nabla\varphi\|_3\exp\bigg(-c\intop_0^t(\|\nabla\varphi\|_3+ \|\rot\psi\|_3)dt'\bigg).\cr}
\label{2.79}
\end{equation}
Integrating (\ref{2.79}) with respect to time yields (\ref{2.74}). This concludes the proof.
\end{proof}

\begin{lemma}\label{l2.7}
Assume that $\rot\psi,\nabla\varphi\in L_1(0,t;\Gamma_1^3(\Omega))$, $\eta(0)\in\Gamma_1^2(\Omega)$, $t\le T$. Then
\begin{equation}\eqal{
|\eta(t)|_{2,1}&\le\exp\bigg[c\intop_0^t(|\nabla\varphi(t')|_{3,1}+ |\rot\psi(t')|_{3,1})dt'\bigg]\cr
&\quad\cdot\bigg[c\intop_0^t|\nabla\varphi(t')|_{3,1}dt'+\|\eta(0)\|_2+ \|\eta_t(0)\|_1\bigg].\cr}
\label{2.80}
\end{equation}
\end{lemma}

\begin{proof}
We consider the equation
\begin{equation}
\eta_t=-v\cdot\nabla\eta-a\Delta\varphi-\eta\Delta\varphi
\label{2.81}
\end{equation}
From (\ref{2.81}) we have
$$\eqal{
&\intop_\Omega\eta_{xtt}\eta_{xt}dx=-\intop_\Omega (v_{xt}\cdot\nabla\eta+v_x\nabla\eta_t+v_t\nabla\eta_x+v\cdot\nabla\eta_{xt}) \eta_{xt}dx\cr
&\quad-a\intop_\Omega\Delta\varphi_{xt}\eta_{xt}dx-\intop_\Omega(\eta_{xt} \Delta\varphi+\eta_x\Delta\varphi_t+\eta_t\Delta\varphi_x+\eta\Delta\varphi_{xt}) \eta_{xt}dx.\cr}
$$
Hence
$$
{d\over dt}|\eta_{xt}|_2^2\le c|v|_{3,1}(\|\nabla\eta\|_1^2+\|\eta_{xt}\|_0^2)+ |\Delta\varphi_{xt}|_2|\eta_{xt}|_2+c|\nabla\varphi|_{3,1}(\|\nabla\eta\|_1^2+ \|\eta_t\|_1^2)
$$
and
$$
{d\over dt}|\eta|_2^2\le c(|v_t|_\infty+|\nabla\varphi|_{3,1})(|\eta_t|_2^2+ \|\eta\|_2^2)+c|\Delta\varphi_t|_2|\eta_t|_2.
$$
Using (\ref{2.78}) yields
$$\eqal{
&{d\over dt}(\|\eta\|_2^2+\|\eta_t\|_1^2)\le c|v|_{3,1}(\|\eta\|_2^2+\|\eta_t\|_1^2)+c|\nabla\varphi|_{3,1}(\|\eta\|_2^2+ \|\eta_t\|_1^2)\cr
&\quad+|\nabla\varphi|_{3,1}(\|\eta_t\|_1+\|\eta\|_2)\cr}
$$
This implies (\ref{2.80}) and concludes the proof.
\end{proof}

\begin{remark}\label{r2.4}
From (\ref{2.25}) for $r=6$ and from (\ref{2.72}) we have
\begin{equation}\eqal{
|v|_6&\le c|\eta|_{18/7,6,\Omega^t}+c\Psi^{2/3}/\nu^{1/3}+c\Psi/\nu^{1/3}+ 
+c{\Psi^{2/3}\over\nu^{\varkappa/3-1/6}}\cr
&\quad+c|f|_{18/7,6,\Omega^t}+c |\varrho_0|_\infty^{1/6}|v_0|_6\equiv D_1,\cr}
\label{2.82}
\end{equation}
where $\varkappa>1/2$ appears in (\ref{2.71}) and $\Psi$ is introduced in (\ref{2.58}).
\end{remark}

\noindent
Differentiating (\ref{2.13}) with respect to $t$ yields
\begin{equation}\eqal{
&\varrho v_{tt}+\varrho_tv_t+\varrho v\cdot\nabla v_t+\varrho_tv\cdot\nabla v+\varrho v_t\cdot\nabla v-\mu\Delta v_t\cr
&\quad-\nu_1\nabla\divv v_t+\nabla(p_t-\nu_2\divv v_t)=\varrho f_t+\varrho_tf,\cr}
\label{2.83}
\end{equation}
where $\nu=\nu_1+\nu_2$, $\nu_i>0$, $i=1,2$.

\noindent
Next we derive the result

\begin{lemma}\label{l2.8} Assume $\eta\in L_\infty(0,T;\Gamma_1^2(\Omega))$, $v\in L_\infty(0,T;L_6(\Omega))\cap\break\cap\;  L_\infty(0,T;L_2(\Omega))$, $\Delta\varphi\in L_2(0,T;L_6(\Omega))$, $\Delta\varphi_t\in L_2(0,T;L_3(\Omega))$, $f_t\in L_2(0,T;L_{6/5}(\Omega))$, $f\in L_2(0,T;L_{3/2}(\Omega))$.\\
Recall the estimates
$$
|v(t)|_2\le A_1\ \ ({\rm see}\ (\ref{2.1})),\quad |v(t)|_6\le D_1\ \ ({\rm see}\ (\ref{2.82}))
$$
Then
\begin{equation}\eqal{
&|v_t(t)|_2^2+\mu|\nabla v_t|_{2,2,\Omega^t}^2+\nu|\Delta\varphi_t|_{2,2,\Omega^t}^2\cr
&\le c(a)\exp[cB_1(t)] [B_2(t)+|\varrho_0|_\infty|v_t(0)|_2^2]\equiv D_2^2\cr}
\label{2.84}
\end{equation}
where
$$
B_1(t)=\intop_0^t|\eta_t|_3^2dt'+\sup_t\|\eta(t)\|_2^2A_1^2+D_1^2A_1^2+A_1^2
$$
and
$$\eqal{
&B_2(t)=|\eta_t|_{2,2,\Omega^t}^2+ \|\eta_t\|_{1,\infty,\Omega^t}^2|\Delta\varphi|_{6,2,\Omega^t}^2A_1^2\cr
&\quad+\|\eta_t\|_{1,\infty,\Omega^t}^2D_1^2A_1^2+ |\Delta\varphi_t|_{3,2,\Omega^t}^2A_1^2+|f_t|_{6/5,2,\Omega^t}^2+ \|\eta_t\|_{1,\infty,\Omega^t}^2|f|_{3/2,2,\Omega^t}^2.\cr}
$$
\end{lemma}

\begin{proof}
Multiplying (\ref{2.83}) by $v_t$, integrating over $\Omega$, using $(\ref{1.1})_2$ and boundary conditions we get
\begin{equation}\eqal{
&{1\over2}{d\over dt}\intop_\Omega\varrho|v_t|^2dx+\mu|\nabla v_t|_2^2+\nu_1 |\divv v_t|_2^2=\intop_\Omega(p_t-\nu_2\divv v_t)\divv v_tdx\cr
&\quad-\intop_\Omega\varrho_tv_t^2dx-\intop_\Omega\varrho_tv\cdot\nabla v\cdot v_tdx-\intop_\Omega\varrho v_t\cdot\nabla v\cdot v_tdx\cr
&\quad+\intop_\Omega(\varrho f_t+\varrho_tf)\cdot v_tdx.\cr}
\label{2.85}
\end{equation}
From (\ref{2.85}) we have
\begin{equation}\eqal{
&{1\over2}{d\over dt}\intop_\Omega\varrho|v_t|^2dx+\mu|\nabla v_t|_2^2+\nu |\divv v_t|_2^2\le{\varepsilon\over2}|\divv v_t|_2^2\cr
&\quad+{1\over2\varepsilon}|\eta_t|_2^2-\intop_\Omega\varrho_tv_t^2dx- \intop_\Omega\varrho_tv\cdot\nabla v\cdot v_tdx-\intop_\Omega\varrho v_t\cdot\nabla v\cdot v_tdx\cr
&\quad+\intop_\Omega(\varrho f_t+\varrho_tf)\cdot v_tdx.\cr}
\label{2.86}
\end{equation}
Now we estimate non-positive terms from the r.h.s. of (\ref{2.86}). We estimate the third term by
$$
\bigg|\intop_\Omega\eta_tv_t^2dx\bigg|\le|\eta_t|_3|v_t|_2|v_t|_6\le \varepsilon/2|v_t|_6^2+c/\varepsilon|\eta_t|_3^2|v_t|_2^2
$$
We intergate by parts in the fourth term on the r.h.s. of (\ref{2.86}). Then we have
$$\eqal{
&-\intop_\Omega\varrho_tv\cdot\nabla vv_tdx=\intop_\Omega v\cdot\nabla\varrho_tv\cdot v_tdx+\intop_\Omega\varrho_t\divv vv\cdot v_tdx\cr
&\quad+\intop_\Omega\varrho_tv\cdot\nabla v_t\cdot vdx\equiv\sum_{i=1}^3K_i,\cr}
$$
where
$$\eqal{
&|K_1|\le\varepsilon|v_t|_6^2+c/\varepsilon|v|_6^4|\eta_{xt}|_2^2,\cr
&|K_2|\le\varepsilon|v_t|_6^2+c/\varepsilon|\eta_t|_6^2 |\Delta\varphi|_6^2|v|_2^2,\cr
&|K_3|\le\varepsilon|\nabla v_t|_2^2+c/\varepsilon|\eta_t|_6^2|v|_6^4.\cr}
$$
Integrating by parts in the fifth term on the r.h.s. of (\ref{2.86}) yields
$$\eqal{
&-\intop_\Omega\varrho v_t\cdot\nabla v\cdot v_tdx=\intop_\Omega v_t\cdot\nabla\varrho v\cdot v_tdx+\intop_\Omega\varrho\Delta\varphi_tv\cdot v_tdx\cr
&\quad+\intop_\Omega\varrho v_t\cdot\nabla v_t\cdot vdx\equiv\sum_{i=1}^3L_i,\cr}
$$
where
$$\eqal{
&|L_1|\le\varepsilon|v_t|_6^2+c/\varepsilon|\eta_x|_6^2|v|_6^2|v_t|_2^2,\cr
&|L_2|\le\varepsilon|v_t|_6^2+c/\varepsilon|\varrho|_\infty^2 |\Delta\varphi_t|_3^2|v|_2^2,\cr
&|L_3|\le\varepsilon|\nabla v_t|_2^2+c/\varepsilon|\varrho|_\infty^2|v_t|_3^2 |v|_6^2\equiv L_3^1+L_3^2.\cr}
$$
To estimate $L_3^2$ we use the interpolation
$$
|v_t|_3\le c|\nabla v_t|_2^{1/2}|v_t|_2^{1/2}+|v_t|_2.
$$
Hence
$$
L_3^2\le\varepsilon|\nabla v_t|_2^2+c/\varepsilon|\varrho|_\infty^4|v|_6^4 |v_t|_2^2+c/\varepsilon|\varrho|_\infty^2|v|_6^2|v_t|_2^2.
$$
Finally, we estimate the last term on the r.h.s. of (\ref{2.86}) by
$$
\bigg|\intop_\Omega(\varrho f_t+\varrho_tf)\cdot v_tdx\bigg|\le\varepsilon |v_t|_6^2+c/\varepsilon|\varrho|_\infty^2|f_t|_{6/5}^2+c/\varepsilon |\varrho_t|_6^2|f|_{3/2}^2.
$$
Employing the above estimates in (\ref{2.86}) and assuming that $\varepsilon$ is sufficiently small we have
\begin{equation}\eqal{
&{1\over2}{d\over dt}\intop_\Omega\varrho|v_t|^2dx+{\mu\over2}|\nabla v_t|_2^2+ {\nu\over2}|\divv v_t|_2^2\cr
&\le c[|\eta_t|_2^2+|\eta_t|_3^2|v_t|_2^2+|v|_6^4|\eta_{xt}|_2^2+|\eta_t|_6^2 |\Delta\varphi|_6^2|v|_2^2+|\eta_t|_6^2|v|_6^4\cr
&\quad+|\eta_x|_6^2|v|_6^2|v_t|_2^2+|\varrho|_\infty^2|\Delta\varphi_t|_3^2 |v|_2^2+|\varrho|_\infty^4|v|_6^4|v_t|_2^2+|\varrho|_\infty^2|v|_6^2|v_t|_2^2\cr
&\quad+|\varrho|_\infty^2|f_t|_{6/5}^2+|\varrho_t|_6^2|f|_{3/2}^2].\cr}
\label{2.87}
\end{equation}
Using Lemma \ref{l2.1} and that $a/2\le\varrho\le3a/2$ we obtain
\begin{equation}\eqal{
&{d\over dt}\intop_\Omega\varrho|v_t|^2dx+\mu|\nabla v_t|_2^2+\nu|\divv v_t|_2^2\cr
&\le c(|\eta_t|_3^2+|\eta_x|_6^2|v|_6^2+|v|_6^4+|v|_6^2) \intop_\Omega\varrho|v_t|^2dx\cr
&\quad+c(|\eta_t|_2^2+\|\eta_t\|_1^2|v|_6^4+\|\eta_t\|_1^2|\Delta\varphi|_6^2 A_1^2+\|\eta_t\|_1^2|v|_6^4\cr
&\quad+|\Delta\varphi_t|_3^2A_1^2+|f_t|_{6/5}^2+\|\eta_t\|_1^2|f|_{3/2}^2).\cr}
\label{2.88}
\end{equation}
Introduce the quantities
$$\eqal{
&B'_1(t)=|\eta_t(t)|_3^2+|\eta_x(t)|_6^2|v(t)|_6^2+|v(t)|_6^4+|v(t)|_6^2,\cr
&B'_2(t)=|\eta_t(t)|_2^2+\|\eta_t(t)\|_1^2|v(t)|_6^4+\|\eta_t(t)\|_1^2 |\Delta\varphi(t)|_6^2A_1^2\cr
&\quad+|\Delta\varphi_t|_3^2A_1^2+|f_t(t)|_{6/5}^2+\|\eta_t(t)\|_1^2 |f(t)|_{3/2}^2.\cr}
$$
Using the quantities in (\ref{2.88}) it takes the form
\begin{equation}
{d\over dt}\intop_\Omega\varrho|v_t|^2dx+\mu|\nabla v_t|_2^2+\nu|\divv v_t|_2^2\le cB'_1\intop_\Omega\varrho|v_t|^2dx+cB'_2.
\label{2.89}
\end{equation}
Integrating (\ref{2.89}) with respect to time yields
\begin{equation}\eqal{
&\intop_\Omega\varrho|v_t|^2dx+\mu\intop_0^t|\nabla v_t|_2^2dt'+\nu\intop_0^t|\divv v_t|^2dt'\cr
&\le\exp\bigg[c\intop_0^tB'_1(t')dt'\bigg]\bigg[c\intop_0^tB'_2(t')dt'+ |\varrho_0|_\infty|v_t(0)|_2^2\bigg].\cr}
\label{2.90}
\end{equation}
In view of Lemma \ref{l2.1} and (\ref{2.82}) we have
\begin{equation}\eqal{
B_1&=\intop_0^tB'_1dt'\le|\eta_t|_{3,2,\Omega^t}^2+ \|\eta\|_{2,\infty,\Omega^t}^2A_1^2+D_1^2A_1^2+A_1^2,\cr
B_2&=\intop_0^tB'_2dt'\le|\eta_t|_{2,2,\Omega^t}^2+ \|\eta_t\|_{1,\infty,\Omega^t}^2D_1^2A_1^2+\|\eta_t\|_{1,\infty,\Omega^t}^2 |\Delta\varphi|_{6,2,\Omega^t}^2A_1^2\cr
&\quad+|\Delta\varphi_t|_{3,2,\Omega^t}^2A_1^2+|f_t|_{6/5,2,\Omega^t}^2\cr
&\quad+\|\eta_t\|_{1,\infty,\Omega^t}^2|f|_{3/2,2,\Omega^t}^2.\cr}
\label{2.91}
\end{equation}
Exployting the estimates in (\ref{2.90}) implies (\ref{2.84}). This concludes the proof.
\end{proof}

To formulate Section \ref{s3} more explicitly we introduce the notation

\begin{notation}\label{n2.13}
We introduce the quantities
$$\eqal{
&\chi_1^2(t)=\nu|\nabla\varphi|_{2,1,\infty,\Omega^t}^2+ |\rot\psi|_{2,1,\infty,\Omega^t}^2,\cr
&\chi_2^2(t)=\nu|\nabla\varphi|_{3,1,2,\Omega^t}^2+ |\rot\psi|_{3,1,2,\Omega^t}^2,\cr
&\Psi^2(t)=\nu^2|\nabla\varphi|_{3,1,2,\Omega^t}^2,\cr
&\Phi_1^2(t)=\nu(U_1^2+V_1^2)\equiv\nu(|\nabla\varphi|_{2,1,\infty,\Omega^t}^2+ |\nabla\varphi|_{3,1,2,\Omega^t}^2)\equiv\Phi_{11}^2(t)+\Phi_{12}^2(t),\cr
&\Phi_2^2(t)=\nu(U_2^2+V_2^2)\equiv|\rot\psi|_{2,1,\infty,\Omega^t}^2+ |\rot\psi|_{3,1,2,\Omega^t}^2\equiv\Phi_{21}^2(t)+\Phi_{22}^2(t),\cr
&\chi_0^2=\nu|\nabla\varphi|_{2,1,\infty,\Omega^t}^2\equiv\Phi_{11}^2(t),\cr
&\Phi_0^2(t)=\nu|\nabla\varphi(t)|_{2,1}^2+|\rot\psi(t)|_{2,1}^2,\cr
&\Phi_*^2(t)=|\nabla\varphi|_{3,1,2,\Omega^t}^2+|\rot\psi|_{3,1,2,\Omega^t}^2.\cr}
$$
From (\ref{2.71}) we have
\begin{equation}
|\eta(t)|_r\le\exp\bigg[(1-1/r)\intop_0^t|\Delta\varphi(t')|_\infty dt'\bigg] \bigg[\intop_0^t|\Delta\varphi|_rdt'+|\eta(0)|_r\bigg].
\label{2.92}
\end{equation}
From (\ref{2.92}) we have
\begin{equation}
|\eta(t)|_r\le\exp\bigg(t^{1/2}{\Psi\over\nu}\bigg)\bigg[t^{1/2}{\Psi\over\nu}+ |\eta(0)|_r\bigg]
\label{2.93}
\end{equation}
From (\ref{2.78}) we derive the estimate
\begin{equation}\eqal{
|\eta(t)|_{2,1}&\le\exp\bigg[ct^{1/2}\bigg(\bigg(\intop_0^t |\nabla\varphi(t')|_{3,1}^2dt'\bigg)^{1/2}+\bigg(\intop_0^t|\rot\psi(t')|_{3,1}^2 dt'\bigg)^{1/2}\bigg)\bigg]\cr
&\quad\cdot\bigg[t^{1/2}\bigg(\intop_0^t|\nabla\varphi(t')|_{3,1}^2 dt'\bigg)^{1/2}+ |\eta(0)|_{2,1}\bigg]\cr
&\le\exp\bigg[ct^{1/2}\bigg({\Psi\over\nu}+\Phi_2\bigg)\bigg]\bigg[t^{1/2} {\Psi\over\nu}+|\eta(0)|_{2,1}\bigg].\cr}
\label{2.94}
\end{equation}
From (\ref{2.82}) and (\ref{2.93}) we have
\begin{equation}
D_1^2=c\bigg(\exp\bigg(t^{1/2}{\Psi\over\nu}\bigg) \bigg(t^{1/2}{\Psi\over\nu}+|\eta(0)|_3^2\bigg)+ {\Psi^{4/3}\over\nu^{2/3}}+{\Psi^2\over\nu^{2/3}}+ {\Psi^{4/3}\over\nu^{2\varkappa/3-1/3}}\bigg)+A_2^2,
\label{2.95} 
\end{equation}
where $A_2=|f|_{3,6,\Omega^t}+|\varrho_0|_\infty^{1/6}|v_0|_6$ (see also (\ref{2.82})), and from (\ref{2.83}) we obtain
\begin{equation}\eqal{
D_2^2&\le\|\eta_t\|_{1,2,\Omega^t}^2\bigg(1+D_1^2A_1^2+{\Psi^2\over\nu^2}A_1^2+ |f|_{3/2,2,\Omega^t}^2\bigg)\cr
&\quad+{\Psi^2\over\nu^2}A_1^2+|f_t|_{6/5,2,\Omega^t}^2,\cr}
\label{2.96}
\end{equation}
where $A_1^2=2|\varrho_0|_1|f|_{\infty,1,\Omega^t}^2+{3\over2}\intop_\Omega \big({1\over2}\varrho_0v_0^2+{A\over\varkappa-1}\varrho_0^\varkappa\big)dx$ (see (\ref{2.1})).
\end{notation}

\section{Differential inequality}\label{s3}

We recall the considered equations
\begin{equation}
\eta_t+v\cdot\nabla\eta+a\Delta\varphi+\eta\Delta\varphi=0,\quad \eta|_{t=0}=\eta(0),
\label{3.1}
\end{equation}
and
\begin{equation}\eqal{
&(a+\eta)v_t-\mu\Delta v-\nu\nabla\Delta\varphi+a_0\nabla\eta+(a+\eta)v\cdot \nabla v\cr
&=[p_\varrho(a)-p_\varrho(a+\eta)]\nabla\eta+(a+\eta)f,\quad v|_{t=0}=v(0),\cr}
\label{3.2}
\end{equation}
where $a_0=p_\varrho(a)$ and
\begin{equation}
v=\nabla\varphi+\rot\psi+G,
\label{3.3}
\end{equation}
where $G$ is defined below (\ref{1.6}).

\begin{lemma}\label{l3.1}
Assume that $|v|_2\le A_1$ (see Lemma \ref{l2.1}), $|v|_6\le D_1$ (see (\ref{2.82})), $|v_t|\le D_2$ (see (\ref{2.84})). Then for sufficiently regular solutions to (\ref{3.1})--(\ref{3.3}) we have
\begin{equation}\eqal{
&{a\over2}{d\over dt}\bigg(|\nabla\varphi|_2^2+{1\over\nu}|\rot\psi|_2^2\bigg)+ \mu|\nabla^2\varphi|_2^2+\nu|\Delta\varphi|_2^2+{\mu\over\nu} |\nabla\rot\psi|_2^2\cr
&\le{c\over\nu}[|\eta|_2^2+|\eta|_3^2|v_t|_2^2+(1+|\eta|_\infty^2) |\Delta\varphi|_3^2A_1^2+(1+|\eta|_\infty^2)^2D_1^4A_1^2\cr
&\quad+|\eta|_3^2|\nabla\eta|_2^2+(1+|\eta|_3^2)|f|_2^2],\cr
&{a\over2}\bigg(|\nabla\varphi(t)|_2^2+{1\over\nu}|\rot\psi(t)|_2^2\bigg)+\mu |\nabla^2\varphi|_{2,2,\Omega^t}^2+\nu|\Delta\varphi|_{2,2,\Omega^t}^2+ {\mu\over\nu}|\nabla\rot\psi|_{2,2,\Omega^t}^2\cr
&\le{c\over\nu}[|\eta|_{2,2,\Omega^t}^2+|\eta|_{3,\infty,\Omega^t}^2D_2^2+ (1+|\eta|_{\infty,\infty,\Omega^t}^2)|\Delta\varphi|_{3,2,\Omega^t}^2A_1^2\cr
&\quad+(1+\|\eta\|_{2,\infty,\Omega^t}^2)D_1^2A_1^2+|\eta|_{3,\infty,\Omega^t}^2+ |\nabla\eta|_{2,2,\Omega^t}^2+(1+|\eta|_{3,\infty,\Omega^t}^2) |f|_{2,2,\Omega^t}^2]\cr
&\quad+{a\over2}\bigg(|\nabla\varphi(0)|_2^2+{1\over\nu}|\rot\psi(0)|_2^2 \bigg).\cr}
\label{3.4}
\end{equation}
\end{lemma}

\begin{proof}
Multiplying (\ref{3.2}) by $\nabla\varphi$ and integrating over $\Omega$ yields
\begin{equation}\eqal{
&{a\over2}{d\over dt}|\nabla\varphi|_2^2+\mu|\nabla^2\varphi|_2^2+\nu |\Delta\varphi|_2^2=-a_0\intop_\Omega\nabla\eta\cdot\nabla\varphi dx\cr
&\quad-\intop_\Omega\eta v_t\cdot\nabla\varphi dx-\intop_\Omega(a+\eta)v\cdot\nabla v\cdot\nabla\varphi dx\cr
&\quad+\intop_\Omega[p_\varphi(a)-p_\varphi(a+\eta)]\nabla\eta\cdot\nabla\varphi dx+\intop_\Omega(a+\eta)f\cdot\nabla\varphi dx.\cr}
\label{3.5}
\end{equation}
Integrating by parts in the first term on the r.h.s. of (\ref{3.5}) we estimate it by
$$
\varepsilon|\Delta\varphi|_2^2+c/\varepsilon|\eta|_2^2.
$$
The second term on the r.h.s. of (\ref{3.5}) is bounded by
$$
\varepsilon|\nabla\varphi|_6^2+c/\varepsilon|\eta|_3^2|v_{,t}|_2^2\le \varepsilon|\nabla\varphi|_6^2+c/\varepsilon|\eta|_3^2D_2^2,
$$
where (\ref{2.67}) is used.

\noindent
The third term on the r.h.s. of (\ref{3.5}) is expressed in the form
$$
-a\intop_\Omega v\cdot\nabla v\cdot\nabla\varphi dx-\intop_\Omega\eta v\cdot\nabla v\cdot\nabla\varphi dx\equiv I_1+I_2.
$$
We drop the factor $a$ in $I_1$ for simplicity. Integrating by parts we express it in the form
$$
I_1=\intop_\Omega\divv vv\cdot\nabla\varphi dx+\intop_\Omega v\cdot v\cdot\nabla\nabla\varphi dx\equiv I_{11}+I_{12}.
$$
Then
$$
|I_{11}|\le\varepsilon|\nabla\varphi|_6^2+c/\varepsilon|\Delta\varphi|_3^2 |v|_2^2\le\varepsilon|\nabla\varphi|_6^2+c/\varepsilon|\Delta\varphi|_3^2A_1^2.
$$
and
$$
|I_{12}|\le\varepsilon|\nabla^2\varphi|_2^2+c/\varepsilon|v|_6^2|v|_3^2\le \varepsilon|\nabla^2\varphi|_2^2+c/\varepsilon D_1^2|v|_3^2.
$$
Next we estimate $I_2$. Integration by parts yields
$$\eqal{
I_2&=\intop_\Omega\nabla\eta\cdot vv\cdot\nabla\varphi dx+\intop_\Omega\eta\Delta\varphi v\cdot\nabla\varphi dx+\intop_\Omega\eta vv\cdot\nabla^2\varphi dx\cr
&\equiv I_2^1+I_2^2+I_2^3,\cr}
$$
where
$$\eqal{
&|I_2^1|\le\varepsilon|\nabla\varphi|_6^2+c/\varepsilon |\nabla\eta|_6^2|v|_6^2|v|_2^2\le\varepsilon|\nabla\varphi|_6^2+c/\varepsilon \|\eta\|_2^2D_1^2A_1^2,\cr
&|I_2^2|\le\varepsilon|\nabla\varphi|_6^2+c/\varepsilon|\eta|_\infty^2 |\Delta\varphi|_3^2|v|_2^2\le\varepsilon|\nabla\varphi|_6^2+c/\varepsilon |\eta|_\infty^2 |\Delta\varphi|_3^2A_1^2,\cr
&|I_2^2|\le\varepsilon|\nabla^2\varphi|_2^2+c/\varepsilon|\eta|_\infty^2D_1^2 |v|_3^2.\cr}
$$
The fourth term on the r.h.s. of (\ref{3.5}) is bounded by
$$
c\intop_\Omega|\eta|\,|\nabla\eta|\,|\nabla\varphi|dx\le\varepsilon |\nabla\varphi|_6^2+c/\varepsilon|\eta|_3^2|\nabla\eta|_2^2.
$$
Finally, the last term on the r.h.s. of (\ref{3.5}) equals
$$
I_3=a\intop_\Omega f_g\cdot\nabla\varphi dx+\intop_\Omega\eta f\cdot\nabla\varphi dx.
$$
Hence, it is bounded by
$$
|I_3|\le\varepsilon|\nabla\varphi|_2^2+c/\varepsilon|f_g|_2^2+\varepsilon |\nabla\varphi|_6^2+c/\varepsilon|\eta|_3^2|f|_2^2.
$$
Employing the above estimates in (\ref{3.5}) and assuming that $\varepsilon$ is sufficiently small yields
\begin{equation}\eqal{
&a{d\over dt}|\nabla\varphi|_2^2+\mu|\nabla^2\varphi|_2^2+\nu|\Delta\varphi|_2^2 \le{c\over\nu}[|\eta|_2^2+|\eta|_3^2|v_t|_2^2\cr
&\quad+(1+|\eta|_\infty^2)|\Delta\varphi|_3^2A_1^2+(1+|\eta|_\infty^2)D_1^2 |v|_3^2+\|\eta\|_2^2D_1^2|v|_2^2\cr
&\quad+|\eta|_3^2|\nabla\eta|_2^2+(1+|\eta|_3^2)|f|_2^2].\cr}
\label{3.6}
\end{equation}
Multiplying (\ref{3.2}) by $\rot\psi$ and integrating the result over $\Omega$ implies
\begin{equation}\eqal{
&{a\over2}{d\over dt}|\rot\psi|_2^2+\mu|\nabla\rot\psi|_2^2=-\intop_\Omega\eta v_t\cdot\rot\psi dx\cr
&\quad-\intop_\Omega(a+\eta)v\cdot\nabla v\cdot\rot\psi dx+\intop_\Omega(a+\eta)f\cdot\rot\psi dx.\cr}
\label{3.7}
\end{equation}
We estimate the first term on the r.h.s. of (\ref{3.7}) by
$$
\varepsilon|\rot\psi|_6^2+c/\varepsilon|\eta|_3^2|v_t|_2^2.
$$
The second term on the r.h.s. of (\ref{3.7}) is expressed in the form
$$
-a\intop_\Omega v\cdot\nabla v\cdot\rot\psi dx-\intop_\Omega\eta v\cdot\nabla v\cdot\rot\psi dx\equiv J_1+J_2.
$$
Integrating by parts in $J_1$ yields
$$
J_1=\intop_\Omega\Delta\varphi v\cdot\rot\psi dx+\intop_\Omega v\cdot v\cdot\nabla\rot\psi dx\equiv J_1^1+J_1^2,
$$
where
$$\eqal{
&|J_1^1|\le\varepsilon|\rot\psi|_6^2+c/\varepsilon|\Delta\varphi|_3^2|v|_2^2\le \varepsilon|\rot\psi|_6^2+c/\varepsilon|\Delta\varphi|_3^2A_1^2,\cr
&|J_1^2|\le\varepsilon|\nabla\rot\psi|_2^2+c/\varepsilon|v|_6^2|v|_3^2\le \varepsilon|\nabla\rot\psi|_2^2+c/\varepsilon D_1^2|v|_3^2.\cr}
$$
Next we examine $J_2$. Integration by parts gives
$$\eqal{
J_2&=\intop_\Omega\nabla\eta\cdot vv\cdot\rot\psi dx+\intop_\Omega\eta\Delta\varphi v\cdot\rot\psi dx+\intop_\Omega\eta v\cdot v\cdot\nabla\rot\psi dx\cr
&\equiv J_2^1+J_2^2+J_2^3,\cr}
$$
where
$$\eqal{
&|J_2^1|\le\varepsilon|\rot\psi|_6^2+c/\varepsilon|\nabla\eta|_6^2|v|_6^2|v|_2^2 \le\varepsilon|\rot\psi|_6^2+c/\varepsilon\|\eta\|_2^2D_1^2|v|_2^2,\cr
&|J_2^2|\le\varepsilon|\rot\psi|_6^2+c/\varepsilon|\eta|_\infty^2 |\Delta\varphi|_3^2|v|_2^2\le\varepsilon|\rot\psi|_6^2+c/\varepsilon |\eta|_\infty^2|\Delta\varphi|_3^2A_1^2,\cr
&|J_2^3|\le\varepsilon|\nabla\rot\psi|_2^2+c/\varepsilon|\eta|_\infty^2|v|_6^2 |v|_3^2\le\varepsilon|\nabla\rot\psi|_2^2+c/\varepsilon|\eta|_\infty^2D_1^2 |v|_3^2.\cr}
$$
Finally, the last term on the r.h.s. of (\ref{3.7}) has the form
$$
a\intop_\Omega f_r\cdot\rot\psi dx+\intop_\Omega\eta f\cdot\rot\psi dx\equiv J_3.
$$
Hence
$$
|J_3|\le\varepsilon|\rot\psi|_2^2+c/\varepsilon|f_r|_2^2+\varepsilon |\rot\psi|_6^2+c/\varepsilon|\eta|_3^2|f|_2^2.
$$
Employing the above estimates in (\ref{3.7}) and using that $\varepsilon$ is sufficiently small we derive the inequality
\begin{equation}\eqal{
&{a\over2}{d\over dt}|\rot\psi|_2^2+\mu|\nabla\rot\psi|_2^2\le c[|\eta|_3^2|v_t|_2^2+(1+|\eta|_\infty^2)|\Delta\varphi|_3^2A_1^2\cr
&\quad+(1+|\eta|_\infty^2)D_1^2|v|_3^2+(1+|\eta|_3^2)|f|_2^2].\cr}
\label{3.8}
\end{equation}
Multiplying (\ref{3.8}) by $1/\nu$ and adding to (\ref{3.6}) yields
\begin{equation}\eqal{
&{a\over2}{d\over dt}\bigg(|\nabla\varphi|_2^2+{1\over\nu}|\rot\psi|_2^2\bigg)+ \mu|\nabla^2\varphi|_2^2+\nu|\Delta\varphi|_2^2+{\mu\over\nu} |\nabla\rot|_2^2\cr
&\le{c\over\nu}[|\eta|_2^2+|\eta|_3^2|v_t|_2^2+(1+|\eta|_\infty^2) |\Delta\varphi|_3^2A_1^2+(1+|\eta|_\infty^2)D_1^2|v|_3^2\cr
&\quad+|\eta|_3^2|\nabla\eta|_2^2+(1+|\eta|_3^2)|f|_2^2].\cr}
\label{3.9}
\end{equation}
Using that
$$
(1+|\eta|_\infty^2)D_1^2|v|_3^2\le\varepsilon|\nabla v|_2^2+c/\varepsilon (1+|\eta|_\infty^2)^2D_1^4A_1^2
$$
we obtain (\ref{3.4}).

\noindent
Integrating (\ref{3.9}) with respect to time and using (\ref{2.1}), (\ref{2.67}) gives
\begin{equation}\eqal{
&{a\over2}\bigg(|\nabla\varphi(t)|_2^2+{1\over\nu}|\rot\psi(t)|_2^2\bigg)+\mu |\nabla^2\varphi|_{2,2,\Omega^t}^2+\nu|\Delta\varphi|_{2,2,\Omega^t}^2\cr
&\quad+{\mu\over\nu}|\nabla\rot\psi|_{2,2,\Omega^t}^2\le{c\over\nu} [|\eta|_{2,2,\Omega^t}^2+|\eta|_{3,\infty,\Omega^t}^2D_2^2\cr
&\quad+(1+|\eta|_{\infty,\infty,\Omega^t}^2)|\Delta\varphi|_{3,2,\Omega^t}^2A_1^2 +(1+\|\eta\|_{2,\infty,\Omega^t}^2)D_1^2A_1^2\cr
&\quad+|\eta|_{3,\infty,\Omega^t}^2|\nabla\eta|_{2,2,\Omega^t}^2+(1+ |\eta|_{3,\infty,\Omega^t}^2)|f|_{2,2,\Omega^t}^2]\cr
&\quad+{a\over2}\bigg(|\nabla\varphi(0)|_2^2+{1\over\nu}|\rot\psi(0)|_2^2 \bigg).\cr}
\label{3.10}
\end{equation}
The above inequalities imply (\ref{3.4}) and conclude the proof.
\end{proof}

\begin{remark}\label{r3.1}
To simplify considerations we introduce the quantity
\begin{equation}
\phi_1=|\eta|_2^2+|\eta|_3^2D_2^2+(1+|\eta|_\infty^2)^2D_1^4A_1^2+\|\eta\|_1^4
\label{3.11}
\end{equation}
Then (\ref{3.4}) takes the form
\begin{equation}\eqal{
&{a\over2}{d\over dt}\bigg(|\nabla\varphi|_2^2+{1\over\nu}|\rot\psi|_2^2\bigg)+ \mu|\nabla^2\varphi|_2^2+\nu|\Delta\varphi|_2^2+{\mu\over\nu} |\nabla\rot\psi|_2^2\cr
&\le{c\over\nu}[\phi_1+(1+|\eta|_\infty^2)|\Delta\varphi|_3^2A_1^2+(1+ |\eta|_3^2)|f|_2^2].\cr}
\label{3.12}
\end{equation}
To obtain estimates for time derivatives we express (\ref{3.2}) in the form
\begin{equation}\eqal{
&\nabla\varphi_t+\rot\psi_t+G_t-{\mu\over a}\Delta v-{\nu\over a}\nabla\Delta\varphi +{a_0\over a+\eta}\nabla\eta\cr
&=-{\mu\over a}{\eta\over a+\eta}\Delta v-{\nu\over a}{\eta\over a+\eta}\nabla \Delta\varphi-v\cdot\nabla v\cr
&\quad+{1\over a+\eta}[p_\varrho(a)-p_\varrho(a+\eta)]\nabla\eta+f.\cr}
\label{3.13}
\end{equation}
\end{remark}

\begin{lemma}\label{l3.2}
Let the assumptions of Lemma \ref{l3.1} hold. Let $\phi_2$ be defined by (\ref{3.21}). Then
\begin{equation}\eqal{
&{d\over dt}\bigg(|\nabla\varphi_t|_2^2+{1\over\nu}|\rot\psi_t|_2^2\bigg)+ {\mu\over a}\|\nabla\varphi_t\|_1^2+{\mu\over\nu a}\|\rot\psi_t\|_1^2+{\nu\over a}|\Delta\varphi_t|_2^2\cr
&\le{c\over\nu}[\phi_2(1+|\eta|_{2,1}^2+|\nabla\rot\psi|_3^2+ |\nabla^2\varphi|_3^2)+(|\nabla\rot\psi_t|_2^2+\|\nabla\varphi_{,t}\|_1^2)\cr
&\quad+|\Delta\varphi|_3^2(|\nabla\rot\psi_t|_2^2+\|\nabla\varphi_{,t}\|_1^2+D_2^2) +\|\nabla\varphi_{,t}\|_1^2(|\nabla v|_2^2+D_1^2)\cr
&\quad+|v_t|_3^2D_1^2+|f_t|_2^2]\cr
&\quad+c\nu[|\eta|_{2,1}^2(1+|\eta|_{2,1}^2)|\Delta\varphi|_3^2+\|\eta\|_2^2 \|\nabla\varphi_t\|_1^2].\cr}
\label{3.14}
\end{equation}
\begin{equation}\eqal{
&|\nabla\varphi_t(t)|_2^2+{1\over\nu}|\rot\psi_t(t)|_2^2+{\mu\over a}\|\nabla\varphi_t\|_{1,2,\Omega^t}^2+{\nu\over a} |\Delta\varphi_t|_{2,2,\Omega^t}^2+{\mu\over\nu a} \|\rot\psi_t\|_{1,2,\Omega^t}^2\cr
&\le{c\over\nu}\bigg[|\eta|_{2,1,2,\Omega^t}^2(1+ |\eta|_{2,1,\infty,\Omega^t}^4)+|\eta|_{2,1,\infty,\Omega^t}^2 (|\eta|_{2,1,\infty,\Omega^t}^2+1)\bigg(|\nabla\rot\psi|_{3,2,\Omega^t}^2\cr
&\quad+ |\nabla\rot\psi_t|_{2,2,\Omega^t}^2+{\Psi^2\over\nu^2}\bigg)+{\chi_0^2\over\nu} \bigg(|\nabla\rot\psi_t|_{2,2,\Omega^t}^2+D_2^2+A_1^2+D_1^2\cr
&\quad+{\psi^2\over\nu^2} D_2^2D_1^2\bigg)+ |f_t|_{2,2,\Omega^t}^2\bigg]+c\nu|\eta|_{2,1,\infty,\Omega^t}^2 (1+|\eta|_{2,1,\infty,\Omega^t}^2) {\Psi^2\over\nu^2}+ |\nabla\varphi_t(0)|_2^2\cr
&\quad+{1\over\nu}|\rot\psi_t(0)|_2^2.\cr}
\label{3.15}
\end{equation}
\end{lemma}
\begin{proof}
Differentiate (\ref{3.13}) with respect to $t$, multiply by $\nabla\varphi_t$ and integrate over $\Omega$. Then we have
\begin{equation}\eqal{
&{1\over2}{d\over dt}|\nabla\varphi_t|_2^2+{\mu\over a}|\nabla^2\varphi_t|_2^2+ {\nu\over a}|\Delta\varphi_t|_2^2=-\intop_\Omega\bigg({a_0\over a+\eta}\nabla\eta\bigg)_t\cdot\nabla\varphi_tdx\cr
&\quad-{\mu\over a}\intop_\Omega\bigg({\eta\over a+\eta}\Delta v\bigg)_{,t}\cdot\nabla\varphi_t dx-{\nu\over a}\intop_\Omega\bigg({\eta\over a+\eta}\nabla\Delta\varphi\bigg)_{,t}\cdot\nabla\varphi_tdx\cr
&\quad-\intop_\Omega(v\cdot\nabla v)_{,t}\cdot\nabla\varphi_tdx+\intop_\Omega \bigg[{1\over a+\eta}(p_\varrho(a)-p_\varrho(a+\eta))\nabla\eta\bigg]_{,t}\cdot \nabla\varphi_tdx\cr
&\quad+\intop_\Omega f_{gt}\cdot\nabla\varphi_tdx.\cr}
\label{3.16}
\end{equation}
Now, we estimate the terms from the r.h.s. of (\ref{3.16}). The first term on the r.h.s. of (\ref{3.16}) is bounded by
$$\eqal{
&c\intop_\Omega(|\nabla\eta_t|\,|\nabla\varphi_t|+|\eta_t|\,|\nabla\eta|\, |\nabla\varphi_t|)dx\le\varepsilon|\nabla\varphi_t|_2^2+c/\varepsilon |\nabla\eta_t|_2^2\cr
&\quad+\varepsilon|\nabla\varphi_t|_6^2+c/\varepsilon|\eta_t|_3^2| \nabla\eta|_2^2.\cr}
$$
The second term on the r.h.s. of (\ref{3.16}) equals
$$
-{\mu\over a}\intop_\Omega\bigg({\eta\over a+\eta}\Delta\rot\psi\bigg)_{,t}\cdot \nabla\varphi_tdx-{\mu\over a}\intop_\Omega\bigg({\eta\over a+\eta} \Delta\nabla\varphi\bigg)_{,t}\cdot\nabla\varphi_tdx\equiv I_1+I_2,
$$
where
$$\eqal{
I_1&=-{\mu\over a}\intop_\Omega\bigg({\eta\over a+\eta}\bigg)_{,\eta}\eta_t \Delta\rot\psi\cdot\nabla\varphi_tdx-{\mu\over a}\intop_\Omega{\eta\over a+\eta} \Delta\rot\psi_t\cdot\nabla\varphi_tdx\cr
&\equiv I_{11}+I_{12}.\cr}
$$
Integrating by parts in $I_{11}$ yields
$$\eqal{
I_{11}&={\mu\over a}\intop_\Omega\bigg({\eta\over a+\eta}\bigg)_{,\eta\eta} \nabla\eta\eta_t\nabla\rot\psi\cdot\nabla\varphi_tdx\cr
&\quad+{\mu\over a}\intop_\Omega\bigg({\eta\over a+\eta}\bigg)_{,\eta} \nabla\eta_t\cdot\nabla\rot\psi\cdot\nabla\varphi_tdx\cr
&\quad+{\mu\over a}\intop_\Omega\bigg({\eta\over a+\eta}\bigg)_{,\eta}\eta_t \nabla\rot\psi\cdot\nabla\nabla\varphi_tdx\equiv I_{11}^1+I_{11}^2+I_{11}^3.\cr}
$$
Continuing, we have
$$\eqal{
&|I_{11}^1|\le\varepsilon|\nabla\varphi_t|_6^2+c/\varepsilon|\nabla\eta|_6^2 |\eta_t|_6^2|\nabla\rot\psi|_2^2,\cr
&|I_{11}^2|\le\varepsilon|\nabla\varphi_t|_6^2+c/\varepsilon|\nabla\eta_t|_2^2 |\nabla\rot\psi|_3^2,\cr
&|I_{11}^3|\le\varepsilon|\nabla^2\varphi_t|_2^2+c/\varepsilon|\eta_t|_6^2 |\nabla\rot\psi|_3^2.\cr}
$$
Next we examine $I_{12}$. Integration by parts yields
$$\eqal{
I_{12}&={\mu\over a}\intop_\Omega\bigg({\eta\over a+\eta}\bigg)_{,\eta}\nabla\eta\cdot\nabla\rot\psi_t\cdot\nabla\varphi_tdx\cr
&\quad+{\mu\over a}\intop_\Omega{\eta\over a+\eta}\nabla\rot\psi_t\cdot\nabla^2\varphi_tdx\equiv I_{12}^1+I_{12}^2.\cr}
$$
Continuing, we have
$$\eqal{
&|I_{12}^1|\le\varepsilon|\nabla\varphi_t|_6^2+c/\varepsilon|\nabla\eta|_3^2 |\nabla\rot\psi_t|_2^2,\cr
&|I_{12}^2|\le\varepsilon|\nabla^2\varphi_t|_2^2+c/\varepsilon|\eta|_\infty^2 |\nabla\rot\psi_t|_2^2.\cr}
$$
Summarizing the estimates yields
$$\eqal{
|I_1|&\le\varepsilon\|\nabla\varphi_t\|_1^2+c/\varepsilon|\eta|_{2,1}^4 |\nabla\rot\psi|_2^2+c/\varepsilon|\eta|_{2,1}^2|\nabla\rot\psi|_3^2\cr
&\quad+c/\varepsilon\|\eta\|_2^2|\nabla\rot\psi_{,t}|_2^2.\cr}
$$
Next, we estimate $I_2$. Performing differentiation with respect to time implies
$$
I_2=-{\mu\over a}\intop_\Omega\bigg({\eta\over a+\eta}\bigg)_{,\eta}\eta_t \Delta\nabla\varphi\cdot\nabla\varphi_tdx-{\mu\over a}\intop_\Omega{\eta\over a+\eta}\Delta\nabla\varphi_t\nabla\varphi_tdx\equiv I_{21}+I_{22}.
$$
Consider $I_{21}$. Integration by parts gives
$$\eqal{
I_{21}&={\mu\over a}\intop_\Omega\bigg({\eta\over a+\eta}\bigg)_{,\eta\eta} \nabla\eta\eta_t\Delta\varphi\cdot\nabla\varphi_tdx+{\mu\over a}\intop_\Omega \bigg({\eta\over a+\eta}\bigg)_{,\eta}\nabla\eta_t\Delta\varphi\cdot\nabla\varphi_tdx\cr
&\quad+{\mu\over a}\intop_\Omega\bigg({\eta\over a+\eta}\bigg)_{,\eta}\eta_t \Delta\varphi\cdot\Delta\varphi_tdx\equiv I_{21}^1+I_{21}^2+I_{21}^3.\cr}
$$
Continuing, we have
$$\eqal{
&|I_{21}^1|\le\varepsilon|\nabla\varphi_t|_6^2+c/\varepsilon|\nabla\eta|_6^2 |\eta_t|_6^2|\Delta\varphi|_2^2,\cr
&|I_{21}^2|\le\varepsilon|\nabla\varphi_t|_6^2+c/\varepsilon|\nabla\eta_t|_2^2 |\Delta\varphi|_3^2,\cr
&|I_{21}^3|\le\varepsilon|\Delta\varphi_t|_2^2+c/\varepsilon|\eta_t|_6^2 |\Delta\varphi|_3^2.\cr}
$$
Integration by parts in $I_{22}$ implies
$$
I_{22}={\mu\over a}\intop_\Omega\bigg({\eta\over a+\eta}\bigg)_{,\eta}\nabla\eta \nabla^2\varphi_t\cdot\nabla\varphi_tdx+{\mu\over a}\intop_\Omega{\eta\over a+\eta}|\nabla^2\varphi_t|^2dx\equiv I_{22}^1+I_{22}^2,
$$
where $I_{22}^2$ is qabsorbed by the second term on the l.h.s. of (\ref{3.16}) in such way that
$$\eqal{
&{\mu\over a}\intop_\Omega|\nabla^2\varphi_t|^2dx-{\mu\over a}\intop_\Omega{\eta\over a+\eta}|\nabla^2\varphi_t|^2dx\cr
&={\mu\over a}\intop_\Omega{a\over a+\eta}|\nabla^2\varphi_t|^2dx\ge{2\over3}{\mu\over a} \intop_\Omega|\nabla^2\varphi_t|^2dx,\cr}
$$
where we used that $|\eta|_\infty\le a/2$. Finally,
$$
|I_{22}^1|\le\varepsilon|\nabla^2\varphi_t|_2^2+c/\varepsilon|\nabla\eta|_6^2 |\nabla\varphi_t|_3^2.
$$
Summarizing the above estimates implies
$$
|I_2-I_{22}^2|\le\varepsilon\|\nabla\varphi_t\|_1^2+c/\varepsilon(|\eta|_{2,1}^4 |\nabla^2\varphi|_2^2+|\eta|_{2,1}^2|\nabla^2\varphi|_3^2+\|\eta\|_2^2 |\nabla\varphi_t|_3^2).
$$
Next, we examine the third term on the r.h.s. of (\ref{3.16}). First, we write it in the form
$$
I_3=-{\nu\over a}\intop_\Omega\bigg({\eta\over a+\eta}\bigg)_{,\eta}\eta_t \nabla\Delta\varphi\cdot\nabla\varphi_tdx-{\nu\over a}\intop_\Omega{\eta\over a+\eta}\nabla\Delta\varphi_t\cdot\nabla\varphi_tdx\equiv I_{31}+I_{32}.
$$
Integration by parts in $I_{31}$ yields
$$\eqal{
I_{31}&={\nu\over a}\intop_\Omega\bigg({\eta\over a+\eta}\bigg)_{,\eta\eta} \nabla\eta\eta_t\Delta\varphi\cdot\nabla\varphi_tdx+{\nu\over a}\intop_\Omega \bigg({\eta\over a+\eta}\bigg)_{,\eta}\nabla\eta_t\Delta\varphi\nabla\varphi_t dx\cr
&\quad+{\nu\over a}\intop_\Omega\bigg({\eta\over a+\eta}\bigg)_{,\eta}\eta_t\Delta\varphi\Delta\varphi_tdx\equiv I_{31}^1+I_{31}^2+I_{31}^3,\cr}
$$
where
$$\eqal{
&|I_{31}^1|\le\nu\varepsilon|\nabla\varphi_t|_6^2+{\nu c\over\varepsilon} |\nabla\eta|_6^2|\eta_t|_6^2|\Delta\varphi|_2^2,\cr
&|I_{31}^2|\le\nu\varepsilon|\nabla\varphi_t|_6^2+{\nu c\over\varepsilon} |\nabla\eta_t|_2^2|\Delta\varphi|_3^2,\cr
&|I_{31}^3|\le\nu\varepsilon|\Delta\varphi_t|_2^2+{\nu c\over\varepsilon} |\eta_t|_6^2|\Delta\varphi|_3^2.\cr}
$$
Finally, we examine $I_{32}$. Integration by parts yields
$$
I_{32}={\nu\over a}\intop_\Omega\bigg({\eta\over a+\eta}\bigg)_{,\eta} \nabla\eta\Delta\varphi_t\cdot\nabla\varphi_tdx+{\nu\over a}\intop_\Omega {\eta\over a+\eta}|\Delta\varphi_t|^2dx\equiv I_{32}^1+I_{32}^2,
$$
where $I_{32}^2$ is absorbed by the third term on the l.h.s. of (\ref{3.16}) in the following way
$$\eqal{
&{\nu\over a}\intop_\Omega|\Delta\varphi_t|^2dx-{\nu\over a}\intop_\Omega {\eta\over a+\eta}|\Delta\varphi_t|^2dx\cr
&={\nu\over a}\intop_\Omega{a\over a+\eta} |\Delta\varphi_t|^2dx\ge{2\over3}{\nu\over a}\intop_\Omega|\Delta\varphi_t|^2dx\cr}
$$
which holds for $|\eta|\le a/2$ and
$$
|I_{32}^1|\le\nu\varepsilon|\Delta\varphi_t|_2^2+{\nu c\over\varepsilon} |\nabla\eta|_6^2|\nabla\varphi_t|_3^2.
$$
Summarizing, we have
$$
|I_3-I_{32}^2|\le\nu\varepsilon\|\nabla\varphi_t\|_1^2+{\nu c\over\varepsilon} (|\eta|_{2,1}^4|\Delta\varphi|_2^2+|\eta|_{2,1}^2|\Delta\varphi|_3^2+ \|\eta\|_2^2|\nabla\varphi_t|_3^2).
$$
We write the fourth term on the r.h.s. of (\ref{3.16}) in the form
$$\eqal{
I_4&=\intop_\Omega v_t\cdot\nabla v\cdot\nabla\varphi_tdx+\intop_\Omega v\cdot\nabla v_t\cdot\nabla\varphi_tdx=-\intop_\Omega\Delta\varphi_tv\cdot\nabla\varphi_tdx\cr
&\quad- \intop_\Omega v_tv\cdot\nabla^2\varphi_tdx-\intop_\Omega\Delta\varphi v_t\cdot\nabla\varphi_tdx-\intop_\Omega vv_t\cdot\nabla^2\varphi_tdx.\cr}
$$
Hence, we have
$$\eqal{
|I_4|&\le\varepsilon|\nabla\varphi_t|_6^2+c/\varepsilon(|\Delta\varphi_t|_2^2 |v|_3^2+|\Delta\varphi|_3^2|v_t|_2^2)\cr
&\quad+\varepsilon|\nabla^2\varphi_t|_2^2+c/\varepsilon|v_t|_3^2|v|_6^2\le \varepsilon|\nabla\varphi_t|_6^2+c/\varepsilon(|\Delta\varphi_t|_2^2D_1^2+ |\Delta\varphi|_3^2D_2^2)\cr
&\quad+\varepsilon|\nabla^2\varphi_t|_2^2+c/\varepsilon|v_t|_3^2D_1^2.\cr}
$$
We estimate the fifth term on the r.h.s. of (\ref{3.16}) by
$$\eqal{
|I_5|&\le c\intop_\Omega(|\eta_|\,|\eta|\,|\nabla\eta|+|\eta_t|\,|\nabla\eta|+ |\eta|\nabla\eta_t|)|\nabla\varphi_t|dx\cr
&\le \varepsilon|\nabla\varphi_t|_6^2+c/\varepsilon|(|\eta|_{2,1}^6+\eta|_{2,1}^4).\cr}
$$
Finally, we estimate the last term on the r.h.s. of (\ref{3.16}) by
$$
\varepsilon|\nabla\varphi_t|_6^2+c/\varepsilon|f_{gt}|_2^2.
$$
Employing the above estimates in (\ref{3.16}) and assuming that $\varepsilon$ is sufficiently small we have
\begin{equation}\eqal{
&{d\over dt}|\nabla\varphi_t|_2^2+{\mu\over a}\|\nabla\varphi_t\|_1^2+{\nu\over a}|\Delta\varphi_t|_2^2\cr
&\le{c\over\nu}[|\eta|_{2,1}^2+|\eta|_{2,1}^4+|\eta|_{2,1}^6+ |\eta|_{2,1}^4(|\nabla\rot\psi|_2^2+|\nabla^2\varphi|_2^2)\cr
&\quad+|\eta|_{2,1}^2(|\nabla\rot\psi|_3^2+|\nabla^2\varphi|_3^2)+\|\eta\|_2^2 (|\nabla\rot\psi_t|_2^2+|\nabla\varphi_t|_3^2)\cr
&\quad+|\Delta\varphi_t|_2^2D_1^2+|\Delta\varphi|_3^2D_2^2+|v_t|_3^2D_1^2+ |f_{gt}|_2^2]\cr
&\quad+c\nu[|\eta|_{2,1}^4|\Delta\varphi|_2^2+|\eta|_{2,1}^2|\Delta\varphi|_3^2+ \|\eta\|_2^2|\nabla\varphi_t|_3^2].\cr}
\label{3.17}
\end{equation}
Differentiate (\ref{3.13}) with respect to $t$, multiply by $\rot\psi_t$ and integrate over $\Omega$. Then we have
\begin{equation}\eqal{
&{1\over2}{d\over dt}|\rot\psi_t|_2^2+{\mu\over a}|\nabla\rot\psi_t|_2^2= -{\mu\over a}\intop_\Omega\bigg({\eta\over a+\eta}\Delta v\bigg)_{,t}\cdot \rot\psi_tdx\cr
&\quad-{\nu\over a}\intop_\Omega\bigg({\eta\over a+\eta}\Delta\nabla\varphi\bigg)_{,t}\cdot\rot\psi_tdx-\intop_\Omega(v\cdot\nabla v)_{,t}\cdot\rot\psi_tdx\cr
&\quad+\intop_\Omega f_{r,t}\cdot\rot\psi_tdx.\cr}
\label{3.18}
\end{equation}
Now we examine the particular terms from the r.h.s. of (\ref{3.18}). The first term is written in the form
$$
-{\mu\over a}\intop_\Omega\bigg({\eta\over a+\eta}\Delta\rot\psi\bigg)_{,t}\cdot \rot\psi_tdx-{\mu\over a}\intop_\Omega\bigg({\eta\over a+\eta}\Delta\nabla \varphi\bigg)_{,t}\cdot\rot\psi_tdx\equiv I_1+I_2.
$$
Performing differentiation in $I_1$ yields
$$\eqal{
I_1&=-{\mu\over a}\intop_\Omega\bigg({\eta\over a+\eta}\bigg)_{,\eta}\eta_t\Delta \rot\psi\cdot\rot\psi_tdx-{\mu\over a}\intop_\Omega{\eta\over a+\eta}\Delta\rot\psi_t\cdot\rot\psi_tdx\cr
&\equiv I_{11}+I_{12}.\cr}
$$
Integration by parts in $I_{11}$ implies
$$\eqal{
I_{11}&={\mu\over a}\intop_\Omega\bigg({\eta\over a+\eta}\bigg)_{,\eta\eta} \nabla\eta\eta_t\nabla\rot\psi\cdot\rot\psi_tdx\cr
&\quad+{\mu\over a}\intop_\Omega\bigg({\eta\over a+\eta}\bigg)_{,\eta} \nabla\eta_t\nabla\rot\psi\cdot\rot\psi_tdx\cr
&\quad+{\mu\over a}\intop_\Omega\bigg({\eta\over a+\eta}\bigg)_{,\eta}\eta_t \nabla\rot\psi\cdot\nabla\rot\psi_tdx\equiv I_{11}^1+I_{11}^2+I_{11}^3.\cr}
$$
Continuing, we have the estimates
$$\eqal{
&|I_{11}^1|\le\varepsilon|\rot\psi_t|_6^2+c/\varepsilon|\nabla\eta|_6^2 |\eta_t|_6^2|\nabla\rot\psi|_2^2,\cr
&|I_{11}^2|\le\varepsilon|\rot\psi_t|_6^2+c/\varepsilon|\nabla\eta_t|_2^2 |\nabla\rot\psi|_3^2,\cr
&|I_{11}^3|\le\varepsilon|\nabla\rot\psi_t|_2^2+c/\varepsilon|\eta_t|_6^2 |\nabla\rot\psi|_3^2.\cr}
$$
Integrating by parts in $I_{12}$ gives
$$\eqal{
I_{12}&={\mu\over a}\intop_\Omega\bigg({\eta\over a+\eta}\bigg)_{,\eta} \nabla\eta\nabla\rot\psi_t\cdot\rot\psi_tdx+{\mu\over a}\intop_\Omega{\eta\over a+\eta}|\nabla\rot\psi_t|^2dx\cr
&\equiv I_{12}^1+I_{12}^2,\cr}
$$
where
$$
|I_{12}^1|\le\varepsilon|\nabla\rot\psi_t|_2^2+c/\varepsilon|\nabla\eta|_6^2 |\rot\psi_t|_3^2
$$
and $I_{12}^2$ is absorbed by the second term on the l.h.s. of (\ref{3.18}) in such way that
$$\eqal{
&{\mu\over a}\intop_\Omega|\nabla\rot\psi_t|^2dx-{\mu\over a}\intop_\Omega{\eta\over a+\eta}|\nabla\rot\psi_t|^2dx\cr
&={\mu\over a}\intop_\Omega{a\over a+\eta}|\nabla\rot\psi_t|^2dx\ge{2\over3}{\mu\over a}|\nabla\rot\psi_t|_2^2\quad {\rm for}\ |\eta|\le a/2.\cr}
$$
Next, we examine $I_2$. Performing differentiation with respect to time gives
$$\eqal{
I_2&=-{\mu\over a}\intop_\Omega\bigg({\eta\over a+\eta}\bigg)_{,\eta}\eta_t \Delta\nabla\varphi\cdot\rot\psi_tdx\cr
&\quad-{\mu\over a}\intop_\Omega{\eta\over a+\eta}\Delta\nabla\varphi_t\cdot \rot\psi_tdx\equiv I_{21}+I_{22}.\cr}
$$
Integrating by parts in $I_{21}$ implies
$$\eqal{
I_{21}&={\mu\over a}\intop_\Omega\bigg({\eta\over a+\eta}\bigg)_{,\eta\eta} \nabla\eta\eta_t\Delta\varphi\cdot\rot\psi_tdx\cr
&\quad+{\mu\over a}\intop_\Omega\bigg({\eta\over a+\eta}\bigg)_{,\eta} \nabla\eta_t\Delta\varphi\cdot\rot\psi_tdx\equiv I_{21}^1+I_{21}^2,\cr}
$$
where
$$\eqal{
&|I_{21}^1|\le\varepsilon|\rot\psi_t|_6^2+c/\varepsilon|\nabla\eta|_6^2 |\eta_t|_6^2|\Delta\varphi|_2^2,\cr
&|I_{21}^2|\le\varepsilon|\rot\psi_t|_6^2+c/\varepsilon|\nabla\eta_t|_2^2 |\Delta\varphi|_3^2.\cr}
$$
Integration by parts in $I_{22}$ gives
$$
I_{22}={\mu\over a}\intop_\Omega\bigg({\eta\over a+\eta}\bigg)_{,\eta}\nabla\eta \Delta\varphi_t\cdot\rot\psi_tdx.
$$
Hence,
$$
|I_{22}|\le c\intop_\Omega|\nabla\eta|\,|\Delta\varphi_t|\,|\rot\psi_t|dx\le \varepsilon|\rot\psi_t|_6^2+c/\varepsilon|\nabla\eta|_3^2|\Delta\varphi_t|_2^2.
$$
Collecting the above estimates yields
$$\eqal{
&|I_1|\le\varepsilon\|\rot\psi_t\|_1^2+c/\varepsilon(|\eta|_{2,1}^4 |\nabla\rot\psi|_2^2+|\eta|_{2,1}^2|\nabla\rot\psi|_3^2+\|\eta\|_2^2 |\rot\psi_t|_3^2),\cr
&|I_2|\le\varepsilon|\rot\psi_t|_6^2+c/\varepsilon(|\eta|_{2,1}^4 |\Delta\varphi|_2^2+|\nabla\eta|_3^2|\Delta\varphi_t|_2^2+|\eta|_{2,1}^2 |\Delta\varphi|_3^2).\cr}
$$
The second term on the r.h.s. of (\ref{3.18}) equals
$$
I_3=-{\nu\over a}\intop_\Omega\bigg({\eta\over a+\eta}\bigg)_{,\eta}\eta_t \Delta\nabla\varphi\cdot\rot\psi_tdx-{\nu\over a}\intop_\Omega{\eta\over a+\eta} \Delta\nabla\varphi_t\cdot\rot\psi_tdx.
$$
Integrating by parts in the above integrals yields
$$\eqal{
I_3&={\nu\over a}\intop_\Omega\bigg({\eta\over a+\eta}\bigg)_{,\eta\eta} \nabla\eta\eta_t\Delta\varphi\cdot\rot\psi_tdx\cr
&\quad+{\nu\over a}\intop_\Omega\bigg({\eta\over a+\eta}\bigg)_{,\eta}\nabla \eta_t\Delta\varphi\cdot\rot\psi_tdx\cr
&\quad+{\nu\over a}\intop_\Omega\bigg({\eta\over a+\eta}\bigg)_{,\eta}\nabla \eta\Delta\varphi_t\cdot\rot\psi_tdx.\cr}
$$
Hence, we have
$$
|I_3|\le\varepsilon|\rot\psi_t|_6^2+{c\nu^2\over\varepsilon}(|\nabla\eta|_6^2 |\eta_t|_6^2|\Delta\varphi|_2^2+|\nabla\eta_t|_2^2|\Delta\varphi|_3^2+ |\nabla\eta|_6^2|\Delta\varphi_t|_{3/2}^2).
$$
The third term on the r.h.s. of (\ref{3.18}) equals
$$
I_4=-\intop_\Omega(v\cdot\nabla v_t+v_t\cdot\nabla v)\cdot\rot\psi_tdx\equiv I_{41}+I_{42}.
$$
First we examine $I_{41}$. We write it in the form
$$
I_{41}=-\intop_\Omega v\cdot\nabla\rot\psi_t\cdot\rot\psi_tdx-\intop_\Omega v\cdot\nabla\nabla\varphi_t\cdot\rot\psi_tdx\equiv I_{41}^1+I_{41}^2,
$$
where
$$
I_{41}^1=-{1\over2}\intop_\Omega v\cdot\nabla|\rot\psi_t|^2dx={1\over2}\intop_\Omega\Delta\varphi|\rot\psi_t|^2dx
$$
so
$$
|I_{41}^1|\le\varepsilon|\rot\psi_t|_6^2+c/\varepsilon|\Delta\varphi|_3^2 |\rot\psi_t|_2^2.
$$
Integrating by parts in $I_{41}^2$ yields
$$
I_{41}^2=\intop_\Omega\nabla v\cdot\nabla\varphi_t\cdot\rot\psi_tdx.
$$
Hence,we have
$$
|I_{41}^2|\le\varepsilon|\rot\psi_t|_6^2+c/\varepsilon|\nabla v|_2^2 |\nabla\varphi_t|_3^2.
$$
Next, we examine $I_{42}$. Integration by parts implies
$$
I_{42}=\intop_\Omega\Delta\varphi_tv\cdot\rot\psi_tdx+\intop_\Omega v_t\cdot v\cdot\nabla\rot\psi_tdx.
$$
Hence, we obtain
$$
|I_{42}|\le\varepsilon\|\rot\psi_t\|_1^2+c/\varepsilon(|\Delta\varphi_t|_2^2 |v|_3^2+|v_t|_3^2|v|_6^2).
$$
Finally, the last term on the r.h.s. of (\ref{3.18}) is bounded by
$$
\varepsilon|\rot\psi_t|_2^2+c/\varepsilon|f_{rt}|_2^2.
$$
Employing the above estimates in (\ref{3.18}) and using that $\varepsilon$ is sufficiently small we derive the inequality
\begin{equation}\eqal{
&{d\over dt}|\rot\psi_t|_2^2+{\mu\over a}|\nabla\rot\psi_t|_2^2\le c[|\eta|_{2,1}^2(|\eta|_{2,1}^2+1)|\nabla\rot\psi|_3^2\cr
&\quad+(\|\eta\|_2^2+|\Delta\varphi|_3^2)|\rot\psi_t|_3^2+(|\eta|_{2,1}^4+ |\eta|_{2,1}^2)|\Delta\varphi|_3^2+|\nabla\eta|_3^2|\Delta\varphi_t|_2^2\cr
&\quad+|\nabla v|_2^2|\nabla\varphi_t|_3^2+ (|\Delta\varphi_t|_2^2+|v_t|_3^2)D_1^2+|f_{rt}|_2^2]+c\nu^2 [|\eta|_{2,1}^2(|\eta|_{2,1}^2+1)|\Delta\varphi|_3^2\cr
&\quad+ \|\eta\|_2^2|\Delta\varphi_t|_2^2].\cr}
\label{3.19}
\end{equation}
Multiplying (\ref{3.19}) by $1/\nu$ and adding to (\ref{3.17}) gives the inequality
\begin{equation}\eqal{
&{d\over dt}\bigg(|\nabla\varphi_t|_2^2+{1\over\nu}|\rot\psi_t|_2^2\bigg)+{\mu\over a} \|\nabla\varphi_t\|_1^2+{\mu\over\nu a}\|\rot\psi_t\|_1^2+{\nu\over a} |\Delta\varphi_t|_2^2\cr
&\le{c\over\nu}[|\eta|_{2,1}^2+|\eta|_{2,1}^4+|\eta|_{2,1}^6+|\eta|_{2,1}^2(1+ |\eta|_{2,1}^2) (|\nabla\rot\psi|_3^2+|\nabla^2\varphi|_3^2)\cr
&\quad+(\|\eta\|_2^2+|\Delta\varphi|_3^2)(|\nabla\rot\psi_t|_2^2+ \|\nabla\varphi_t\|_1^2)+|\nabla v|_2^2 |\nabla\varphi_t|_3^2\cr
&\quad+|\Delta\varphi_t|_2^2D_1^2+|\Delta\varphi|_3^2D_2^2+|v_t|_3^2D_1^2+ |f_{rt}|_2^2+|f_{gt}|_2^2]\cr
&\quad+c\nu[|\eta|_{2,1}^2(1+|\eta|_{2,1}^2)|\Delta\varphi|_3^2+\|\eta\|_2^2 \|\nabla\varphi_t\|_1^2].\cr}
\label{3.20}
\end{equation}
Introduce the quantity
\begin{equation}
\phi_2=|\eta|_{2,1}^2(1+|\eta|_{2,1}^2).
\label{3.21}
\end{equation}
Then (\ref{3.20}) takes the form
\begin{equation}\eqal{
&{d\over dt}\bigg(|\nabla\varphi_t|_2^2+{1\over\nu}|\rot\psi_t|_2^2\bigg)+ {\mu\over a}\|\nabla\varphi_t\|_1^2+{\mu\over\nu a}\|\rot\psi_t\|_1^2+{\nu\over a}|\Delta\varphi_t|_2^2\cr
&\le{c\over\nu}[\phi_2(1+|\eta|_{2,1}^2)+\phi_2(|\nabla\rot\psi|_3^2+ |\nabla^2\varphi|_3^2+ |\nabla\rot\psi_t|_2^2+\|\nabla\varphi_{,t}\|_1^2)\cr
&\quad+|\Delta\varphi|_3^2(|\nabla\rot\psi_t|_2^2+\|\nabla\varphi_t\|_1^2+D_2^2)+ \|\nabla\varphi_t\|_1^2(|\nabla v|_2^2+D_1^2)\cr
&\quad+|v_t|_3^2D_1^2+|f_t|_2^2]+c\nu[\phi_2|\Delta\varphi|_3^2+\|\eta\|_2^2 \|\nabla\varphi_t\|_1^2].\cr}
\label{3.22}
\end{equation}
Hence (\ref{3.22}) implies (\ref{3.14}). Integrating (\ref{3.22}) with respect to time and using Notation \ref{n2.13} and the fact that $|v|_6\le D_1$, $|v_t|_2\le D_2$, $|v|_{6,2,\Omega^t}\le A_1$, $|v_t|_{2,\Omega^t}\le D_2$, $|\nabla v|_{2,\Omega^t}\le A_1$ implies (\ref{3.15}). This concludes the proof of Lemma \ref{l3.2}.
\end{proof}

\begin{lemma}\label{l3.3}
Let the assumptions of Lemma \ref{l3.1} hold. Then
\begin{equation}\eqal{
&a{d\over dt}\bigg(|\nabla\varphi_x|_2^2+{1\over\nu}|\rot\psi_x|_2^2\bigg)+\mu |\nabla^2\varphi_x|_2^2+\nu|\Delta\varphi_x|_2^2+{\mu\over\nu} |\nabla\rot\psi_x|_2^2\cr
&\le{c\over\nu}[|\eta|_\infty^2D_2^2+(1+|\eta|_\infty^4)|\nabla v|_2^2D_1^4+ |\nabla\eta|_2^2+|\eta|_\infty^2|\nabla\eta|_2^2\cr
&\quad+(1+|\eta|_\infty^2)|f|_2^2],\cr}
\label{3.23}
\end{equation}
\begin{equation}\eqal{
&|\nabla\varphi_x(t)|_2^2+{1\over\nu}|\rot\psi_x(t)|_2^2+{\mu\over a} |\nabla^2\varphi_x|_{2,2,\Omega^t}^2+{\nu\over a} |\Delta\varphi_x|_{2,2,\Omega^t}^2\cr
&\quad+{\mu\over\nu a}|\nabla\rot\psi_x|_{2,2,\Omega^t}^2\le{c\over\nu} [|\eta|_\infty^2D_2^2+(1+|\eta|_\infty^2)A_1^2D_1^4\cr
&\quad+|\nabla\eta|_2^2+|\eta|_\infty^2|\nabla\eta|_2^2+(1+|\eta|_\infty^2) |f|_{2,2,\Omega^t}^2]\cr
&\quad+|\nabla\varphi_x(0)|_2^2+{1\over\nu}|\rot\psi_x(0)|_2^2.\cr}
\label{3.24}
\end{equation}
\end{lemma}

\begin{proof}
Differentiate (\ref{3.2}) with respect to $x$, multiply by $\nabla\varphi_x$ and integrate over $\Omega$. Then we get
\begin{equation}\eqal{
&{a\over2}{d\over dt}|\nabla\varphi_{,x}|_2^2+\mu|\nabla^2\varphi_x|_2^2+\mu |\Delta\varphi_x|_2^2=-\intop_\Omega(\eta v_t)_{,x}\nabla\varphi_xdx\cr
&\quad-a_0\intop_\Omega\nabla\eta_x\nabla\varphi_xdx-\intop_\Omega[(a+\eta)v \cdot\nabla v]_{,x}\cdot\nabla\varphi_xdx\cr
&\quad+\intop_\Omega[(p_\varrho(a)-p_\varrho(a+\eta)) \nabla\eta]_{,x}\cdot\nabla\varphi_xdx\cr
&\quad+\intop_\Omega[(a+\eta)f]_{,x}\cdot\nabla\varphi_xdx.\cr}
\label{3.25}
\end{equation}
After integration by parts in the first term on the r.h.s. of (\ref{3.25}) we bound it by
$$
\varepsilon|\nabla\varphi_{xx}|_2^2+c/\varepsilon|\eta|_\infty^2|v_t|_2^2.
$$
The second term on the r.h.s. of (\ref{3.25}) is bounded by
$$
\varepsilon|\nabla\varphi_{xx}|_2^2+c/\varepsilon|\nabla\eta|_2^2.
$$
After integration by parts the third term on the r.h.s. of (\ref{3.25}) is estimated by
$$
\varepsilon|\nabla\varphi_{xx}|_2^2+c/\varepsilon(1+|\eta|_\infty^2)|v|_6^2 |\nabla v|_3^2\le\varepsilon|\nabla\varphi_{xx}|_2^2+c/\varepsilon|\nabla v|_3^2 D_1^2(1+|\eta|_\infty^2)
$$
and the fourth term by
$$
\varepsilon|\nabla\varphi_{xx}|_2^2+c/\varepsilon|\eta|_\infty^2|\nabla\eta|_2^2.
$$
Finally, the last term on the r.h.s. of (\ref{3.25}) equals
$$
I_1=-a\intop_\Omega f_g\cdot\nabla\varphi_{xx}dx-\intop_\Omega\eta f\cdot\nabla\varphi_{xx}dx,
$$
so it is estimated by
$$
|I_1|\le\varepsilon|\nabla\varphi_{xx}|_2^2+c/\varepsilon(|f_g|_2^2+ |\eta|_\infty^2|f|_2^2).
$$
Employing the above estimates in (\ref{3.25}) and assuming that $\varepsilon$ is sufficiently small we obtain the inequality
\begin{equation}\eqal{
&a{d\over dt}|\nabla\varphi_x|_2^2+\mu|\nabla^2\varphi_x|_2^2+\nu |\Delta\varphi_x|_2^2\cr
&\le{c\over\nu}[|\eta|_\infty^2D_2^2+|\nabla\eta|_2^2+ |\nabla v|_3^2D_1^2(1+|\eta|_\infty^2)+|\eta|_\infty^2|\nabla\eta|_2^2\cr
&\quad+|f|_2^2(1+|\eta|_\infty^2)].\cr}
\label{3.26}
\end{equation}
Differentiate (\ref{3.2}) with respect to $x$, multiply by $\rot\psi_x$ and integrate over $\Omega$. Then we derive
\begin{equation}\eqal{
&{a\over2}{d\over dt}|\rot\psi_x|_2^2+\mu|\nabla\rot\psi_x|_2^2=-\intop_\Omega (\eta v_t)_{,x}\rot\psi_xdx\cr
&\quad-\intop_\Omega[(a+\eta)v\cdot\nabla v]_{,x}\cdot\rot\psi_xdx+\intop_\Omega [(a+\eta)f]_{,x}\cdot\rot\psi_xdx.\cr}
\label{3.27}
\end{equation}
Integrating by parts in the terms from the r.h.s. of the above equality we obtain
\begin{equation}\eqal{
&a{d\over dt}|\rot\psi_x|_2^2+\mu|\nabla\rot\psi_x|_2^2\cr
&\le c[|\eta|_\infty^2 |v_t|_2^2+(1+|\eta|_\infty^2)|v|_6^2|\nabla v|_3^2+(1+|\eta|_\infty^2)|f|_2^2].\cr}
\label{3.28}
\end{equation}
Multiplying (\ref{3.28}) by $1/\nu$ and adding to (\ref{3.26}) yields
\begin{equation}\eqal{
&a{d\over dt}\bigg(|\nabla\varphi_x|_2^2+{1\over\nu}|\rot\psi_x|_2^2\bigg)+\mu |\nabla^2\varphi_x|_2^2+\nu|\Delta\varphi_x|_2^2+{\mu\over\nu} |\nabla\rot\psi_x|_2^2\cr
&\le{c\over\nu}[|\eta|_\infty^2D_2^2+(1+|\eta|_\infty^2)|\nabla v|_3^2D_1^2+ |\nabla\eta|_2^2+|\eta|_\infty^2|\nabla\eta|_2^2\cr
&\quad+(1+|\eta|_\infty^2)|f|_2^2].\cr}
\label{3.29}
\end{equation}
Using the interpolation
$$
|\nabla v|_3\le c|\nabla^2v|_2^{1/2}|\nabla v|_2^{1/2}
$$
We obtain from (\ref{3.29}) the inequality
\begin{equation}\eqal{
&a{d\over dt}\bigg(|\nabla\varphi_x|_2^2+{1\over\nu}|\rot\psi_x|_2^2\bigg)+\mu |\nabla^2\varphi_x|_2^2+\nu|\Delta\varphi|_2^2+{\mu\over\nu} |\nabla\rot\psi_x|_2^2\cr
&{c\over\nu}[|\eta|_\infty^2D_2^2+(1+|\eta|_\infty^4)|\nabla v|_2^2D_1^4+|\nabla\eta|_2^2+|\eta|_\infty^2|\nabla\eta|_2^2\cr
&\quad+(1+|\eta|_\infty^2)|f|_2^2]\cr}
\label{3.30}
\end{equation}
Integrating (\ref{3.30}) with respect to time yields
\begin{equation}\eqal{
&a\bigg(|\nabla\varphi_x(t)|_2^2+{1\over\nu}|\rot\psi_x(t)|_2^2\bigg)+\mu |\nabla^2\varphi_x|_{2,2,\Omega^t}^2+\nu|\Delta\varphi_x|_{2,2,\Omega^t}^2\cr
&\quad+{\mu\over\nu}|\nabla\rot\psi_x|_{2,2,\Omega^t}^2\le{c\over\nu} [|\eta|_{\infty,2,\Omega^t}^2D_2^2+(1+|\eta|_{\infty,2,\Omega^t}^4)A_1^2D_1^4\cr
&\quad+|\nabla\eta|_{2,2,\Omega^t}^2+|\eta|_{\infty,\infty,\Omega^t}^2 |\nabla\eta|_{2,2,\Omega^t}^2+(1+|\eta|_{\infty,\infty,\Omega^t}^2) |f|_{2,2,\Omega^t}^2]\cr
&\quad+a\bigg(|\nabla\varphi_x(0)|_2^2+{1\over\nu}|\rot\psi_x(0)|_2^2\bigg).\cr}
\label{3.31}
\end{equation}
Since $|\eta(t)|_{2,1}$ is bounded by a constant dependent of time the r.h.s. of (\ref{3.31}) is an increasing function of time. To obtain such estimate that the r.h.s. of (\ref{3.31}) does not depend on time needs another approach. However, we shallshow such result that $\nu\to\infty$ implies that $t\to\infty$. This concludes the proof.
\end{proof}

\begin{remark}\label{r3.2}
We simplify (\ref{3.29}). Introduce the quantities
\begin{equation}
\phi_3=|\eta|_\infty^2(D_2^2+|\nabla\eta|_2^2)+|\nabla\eta|_2^2,\quad \phi_4=(1+|\eta|_\infty^2)D_1^2.
\label{3.32}
\end{equation}
Then (\ref{3.29}) takes the form
\begin{equation}\eqal{
&a{d\over dt}\bigg(|\nabla\varphi_x|_2^2+{1\over\nu}|\rot\psi_x|_2^2\bigg)+\mu |\nabla^2\varphi_x|_2^2+\nu|\Delta\varphi_x|_2^2+{\mu\over\nu} |\nabla\rot\psi_x|_2^2\cr
&\le{c\over\nu}[\phi_3+\phi_4|\nabla v|_3^2+(1+|\eta|_\infty^2)|f|_2^2].\cr}
\label{3.33}
\end{equation}
Using the interpolation
\begin{equation}
\alpha|\nabla v|_3^2\le\varepsilon^{4/3}(|\nabla^2\varphi_x|_2^2+ |\nabla\rot\psi_x|_2^2)+{c\over\varepsilon^4}\alpha^4(|\nabla\varphi|_2^2+ |\rot\psi|_2^2)
\label{3.34}
\end{equation}
we obtain from (\ref{3.33}) the inequality
\begin{equation}\eqal{
&a{d\over dt}\bigg(|\nabla\varphi_x|_2^2+{1\over\nu}|\rot\psi_x|_2^2\bigg)+\mu |\nabla^2\varphi_x|_2^2+\nu|\Delta\varphi_x|_2^2+{\mu\over\nu} |\nabla\rot\psi_x|_2^2\cr
&\le{c\over\nu}[\phi_3+\phi_4^4A_1^2+(1+|\eta|_\infty^2)|f|_2^2].\cr}
\label{3.35}
\end{equation}
Adding (\ref{3.12}) and (\ref{3.35}) we have
\begin{equation}\eqal{
&a{d\over dt}\bigg(\|\nabla\varphi\|_1^2+{1\over\nu}\|\rot\psi\|_1^2\bigg)+\mu \|\nabla\varphi\|_2^2+\nu\|\nabla\varphi\|_2^2+{\mu\over\nu}\|\rot\psi\|_2^2\cr
&\le{c\over\nu}[\phi_1+\phi_3+\phi_4^4A_1^2+(1+|\eta|_\infty^2)A_1^2 |\Delta\varphi|_3^2+(1+|\eta|_\infty^2)|f|_2^2].\cr}
\label{3.36}
\end{equation}
Using interpolation (\ref{3.34}) in the r.h.s. of (\ref{3.22}) to terms $|\nabla\rot\psi|_3^2$, $|\nabla^2\varphi|_3^2$, $|\Delta\varphi|_3^2$ we obtain from (\ref{3.36}) and (\ref{3.22}) the inequality
\begin{equation}\eqal{
&a{d\over dt}\bigg(|\nabla\varphi|_{1,1}^2+{1\over\nu}|\rot\psi|_{1,1}^2\bigg)+ \mu|\nabla\varphi|_{2,1}^2+\nu|\nabla\varphi|_{2,1}^2+{\mu\over\nu} |\rot\psi|_{2,1}^2\cr
&\le{c\over\nu}[\phi_1+\phi_2(1+|\eta|_{2,1}^2)+\phi_3+\phi_4^4A_1^2+ (1+|\eta|_\infty^2)A_1^2 |\Delta\varphi|_3^2\cr
&\quad+\phi_2^4+D_2^8+\phi_2(|\nabla\rot\psi_t|_2^2+|\nabla\varphi_t|_3^2+ |\Delta\varphi_t|_3^2)\cr
&\quad+|\Delta\varphi|_3^2(|\nabla\rot\psi_t|_2^2+ \|\nabla\varphi_t\|_1^2)+|\nabla v|_2^2 \|\nabla\varphi_t\|_1^2+\|\nabla\varphi_t\|_1^2D_1^2\cr
&\quad+|v_t|_3^2 D_1^2+(1+|\eta|_\infty^2)|f|_2^2+|f_t|_2^2]+c\nu[\phi_2|\Delta\varphi|_3^2 +\|\eta\|_2^2\|\nabla\varphi_t\|_1^2].\cr}
\label{3.37}
\end{equation}
\end{remark}

\begin{lemma}\label{l3.4}
Let the assumptions of Lemma \ref{l3.1} hold. Let Notation \ref{n2.13} be applied. Then
\begin{equation}\eqal{
&{d\over dt}\bigg(|\nabla\varphi_{xt}|_2^2+{1\over\nu}|\rot\psi_{xt}|_2^2\bigg)+ {\nu\over a}|\nabla^2\varphi_x|_2^2+{\nu\over a}|\Delta\varphi_{xt}|_2^2+ {\mu\over\nu a}|\nabla\rot\psi_{xt}|_2^2\cr
&\le{\varepsilon\over\nu}|v_{xxx}|_2^2+\nu\varepsilon_1 |\nabla\varphi_{,xxx}|_2^2\cr
&{c\over\nu\varepsilon}[\phi_8+|\nabla v|_3^2|\nabla\varphi_{,xt}|_2^2+|\nabla v|_3^2|\nabla v_{,t}|_2^2+|f_t|_2^2]\cr
&\quad+{c\nu\over\varepsilon_1}\phi_7(|\eta|_{2,1})(|\nabla\varphi_{,xx}|_2^2+ |\nabla\varphi_{,xt}|_2^2),\cr}
\label{3.38}
\end{equation}
where $\phi_8$ is introduced below (\ref{3.47}) and $\phi_7$ below (\ref{3.45}).
\begin{equation}\eqal{
&|\nabla\varphi_{,xt}(t)|_2^2+{1\over\nu}|\rot\psi_{,xt}(t)|_2^2+{\mu\over a} |\nabla^2\varphi_{,xt}|_{2,2,\Omega^t}^2+{\nu\over a}|\Delta\varphi_{,xt}|_{2,2,\Omega^t}^2\cr
&\quad+{\mu\over\nu a} |\nabla\rot\psi_{,xt}|_{2,2,\Omega^t}^2\le\varepsilon\bigg({1\over\nu}|v_{,xxx}|_{2,2,\Omega^t}^2+ |\nabla\varphi_{,xxx}|_{2,2,\Omega^t}^2\bigg)\cr
&\quad+{c\over\varepsilon\nu}\bigg[\intop_0^t\phi_8(t')dt'+{\chi_0^2\over\nu} |\nabla v|_{3,2,\Omega^t}^2+|\nabla v|_{3,\infty,\Omega^t}^2D_1^2+ |f_t|_{2,\Omega^t}^2\bigg]\cr
&\quad+c\nu|\eta|_{2,1,\infty,\Omega^t}^2(1+|\eta|_{2,1,\infty,\Omega^t}^2) {\Psi^2\over\nu^2}+|\nabla\varphi_{,xt}(0)|_2^2+ {1\over\nu}|\rot\psi_{,xt}(0)|_2^2,\cr}
\label{3.39}
\end{equation}
where $|\phi_8|_{1,(0,t)}$ is defined by(\ref{3.53}).
\end{lemma}

\begin{proof}
Differentiate (\ref{3.13}) with respect to $x$ and $t$, multiply by $\nabla\varphi_{xt}$ and integrate over $\Omega$. Then we have
\begin{equation}\eqal{
&{1\over2}{d\over dt}|\nabla\varphi_{xt}|_2^2+{\mu\over a} |\nabla^2\varphi_{xt}|_2^2+{\nu\over a}|\Delta\varphi_{xt}|_2^2\cr
&=-\intop_\Omega \bigg({a_0\over a+\eta}\nabla\eta\bigg)_{,xt}\cdot\nabla\varphi_{,xt}dx-{\mu\over a}\intop_\Omega\bigg({\eta\over a+\eta}\Delta v\bigg)_{,xt} \cdot\nabla\varphi_{xt}dx\cr
&\quad-{\nu\over a}\intop_\Omega\bigg({\eta\over a+\eta} \nabla\Delta\varphi\bigg)_{,xt}\cdot\nabla\varphi_{xt}dx-\intop_\Omega(v\cdot\nabla v)_{,xt}\cdot\nabla\varphi_{xt}dx\cr
&\quad+\intop_\Omega \bigg[{1\over a+\eta}(p_\varrho(a)-p_\varrho(a+\eta))\nabla\eta\bigg]_{,xt}\cdot \nabla\varphi_{xt}dx\cr
&\quad+\intop_\Omega f_{,xt}\cdot\nabla\varphi_{,xt}dx.\cr}
\label{3.40}
\end{equation}
Now, we examine the particular terms from the r.h.s. of (\ref{3.40}). Integrating by parts in the first term we get
$$
I_1=\intop_\Omega\bigg[\bigg({a_0\over a+\eta}\bigg)_{,\eta}\eta_t\nabla\eta+ {a_0\over a+\eta}\nabla\eta_t\bigg]\cdot\nabla\varphi_{,xxt}dx.
$$
Hence, we have
$$
|I_1|\le\varepsilon|\nabla\varphi_{,xxt}|_2^2+c/\varepsilon(|\eta_t|_4^2 |\nabla\eta|_4^2+|\nabla\eta_t|_2^2).
$$
We express the second term on the r.h.s. of (\ref{3.40}) in the form
$$\eqal{
I_2&=-{\mu\over a}\intop_\Omega\bigg[\bigg({\eta\over a+\eta}\bigg)_{,\eta} \eta_t\Delta v\bigg]_{,x}\cdot\nabla\varphi_{,xt}dx\cr
&\quad-{\mu\over a}\intop_\Omega \bigg[{\eta\over a+\eta}\Delta v_{,t}\bigg]_{.x}\cdot\nabla\varphi_{,xt}dx\equiv I_{21}+I_{22}.\cr}
$$
Integrating by parts in $I_{21}$ yields
$$
I_{21}={\mu\over a}\intop_\Omega\bigg({\eta\over a+\eta}\bigg)_{,\eta}\eta_t \Delta v\cdot\nabla\varphi_{,xxt}dx
$$
and
$$
|I_{21}|\le\varepsilon|\nabla\varphi_{,xxt}|_2^2+c/\varepsilon|\eta_t|_6^2 |\Delta v|_3^2.
$$
Performing differentiation in $I_{22}$ gives
$$\eqal{
I_{22}&=-{\mu\over a}\intop_\Omega\bigg({\eta\over a+\eta}\bigg)_{,\eta}\eta_x \Delta v_{,t}\nabla\varphi_{,xt}dx\cr
&\quad-{\mu\over a}\intop_\Omega{\eta\over a+\eta} \Delta v_{tx}\cdot\nabla\varphi_{,xt}dx\equiv I_{221}+I_{222}.\cr}
$$
Integrating by parts in $I_{221}$ implies
$$\eqal{
I_{221}&={\mu\over a}\intop_\Omega\bigg({\eta\over a+\eta}\bigg)_{,\eta\eta} \nabla\eta\eta_x\cdot\nabla v_t\cdot\nabla\varphi_{,xt}dx\cr
&\quad+{\mu\over a}\intop_\Omega\bigg({\eta\over a+\eta}\bigg)_{,\eta}\nabla \eta_x\nabla v_t\cdot\nabla\varphi_{,xt}dx\cr
&\quad+{\mu\over a}\intop_\Omega\bigg({\eta\over a+\eta}\bigg)_{,\eta}\eta_{,x} \nabla v_t\cdot\nabla^2\varphi_{,xt}dx\equiv I_{221}^1+I_{221}^2+I_{221}^3,\cr}
$$
where
$$\eqal{
&|I_{221}^1|\le\varepsilon|\nabla\varphi_{,xt}|_6^2+c/\varepsilon|\eta_x|_6^4 |v_{,xt}|_2^2,\cr
&|I_{221}^2|\le\varepsilon|\nabla\varphi_{,xt}|_6^2+c/\varepsilon|\eta_{xx}|_2^2 |v_{,xt}|_3^2,\cr
&|I_{221}^3|\le\varepsilon|\nabla^2\varphi_{,xt}|_2^2+c/\varepsilon|\eta_x|_6^2 |v_{,xt}|_3^2.\cr}
$$
Next we examine $I_{222}$. We express it in the form
$$\eqal{
I_{222}&=-{\mu\over a}\intop_\Omega{\eta\over a+\eta}\Delta\rot\psi_{,xt}\cdot \nabla\varphi_{,xt}dx\cr
&\quad-{\mu\over a}\intop_\Omega{\eta\over a+\eta}\Delta\nabla\varphi_{,xt}\cdot \nabla\varphi_{,xt}dx\equiv J_1+J_2,\cr}
$$
where
$$\eqal{
J_1&={\mu\over a}\intop_\Omega\bigg({\eta\over a+\eta}\bigg)_{,\eta}\nabla\eta \nabla\rot\psi_{,xt}\cdot\nabla\varphi_{,xt}dx\cr
&\quad+{\mu\over a}\intop_\Omega{\eta\over a+\eta}\nabla\rot\psi_{,xt} \nabla^2\varphi_{,xt}dx\cr
&\equiv J_{11}+J_{12}.\cr}
$$
and
$$\eqal{
J_{12}&={\mu\over a}\intop_\Omega{\eta\over a+\eta}\nabla_j(\rot\psi)_{i,xt} \nabla_j\nabla_i\varphi_{,xt}dx\cr
&=-{\mu\over a}\intop_\Omega\bigg({\eta\over a+\eta}\bigg)_{,\eta}\nabla_i\eta \nabla_j(\rot\psi)_{i,xt}\nabla_j\varphi_{,xt}dx.\cr}
$$
Hence,
$$\eqal{
|J_{11}|&\le\varepsilon_1|\nabla\rot\psi_{,xt}|_2^2+c/\varepsilon_1 |\nabla\eta|_6^2|\nabla\varphi_{,xt}|_3^2\cr
&\le\varepsilon_1|\nabla\rot\psi_{xt}|_2^2+c/\varepsilon_1|\nabla\eta|_6^2 |\varphi_{,xxxt}|_2|\varphi_{,xt}|_2\cr
&\le\varepsilon_1|\nabla\rot\psi_{xt}|_2^2+\varepsilon|\varphi_{,xxxt}|_2^2+ c/\varepsilon_1^2\varepsilon|\nabla\eta|_6^4|\varphi_{,xt}|_2^2\cr}
$$
and
$$\eqal{
|J_{12}|&\le\varepsilon_1|\nabla\rot\psi_{,xt}|_2^2+c/\varepsilon_1 |\nabla\eta|_6^2|\nabla\varphi_{,xt}|_3^2\cr
&\le\varepsilon_1|\nabla\rot\psi_{,xt}|_2^2+c/\varepsilon_1|\nabla\eta|_6^2 |\nabla\varphi_{,xxt}|_2|\nabla\varphi_{,xt}|_2\cr
&\le\varepsilon_1|\nabla\rot\psi_{,xt}|_2^2+\varepsilon |\nabla\varphi_{,xxt}|_2^2+c/\varepsilon_1^2\varepsilon|\nabla\eta|_6^4 |\nabla\varphi_{,xt}|_2^2.\cr}
$$
Summarizing,
$$
|J_1|\le\varepsilon_1|\nabla\rot\psi_{,xt}|_2^2+\varepsilon|\nabla\varphi_{,xxt}|_2^2 +c/\varepsilon_1^2\varepsilon\|\eta\|_2^4|\nabla\varphi_{,xt}|_2^2.
$$
Finally, we examine $J_2$. Integrating by parts yields
$$\eqal{
J_2&={\mu\over a}\intop_\Omega\bigg({\eta\over a+\eta}\bigg)_{,\eta}\nabla\eta \nabla\nabla\varphi_{,xt}\cdot\nabla\varphi_{xt}dx\cr
&\quad+{\mu\over a}\intop_\Omega{\eta\over a+\eta}|\nabla^2\varphi_{,xt}|^2dx \equiv J_{21}+J_{22},\cr}
$$
where
$$\eqal{
|J_{21}|&\le\varepsilon|\nabla^2\varphi_{,xt}|_2^2+c/\varepsilon|\nabla\eta|_6^2 |\nabla\varphi_{,xt}|_3^2\cr
&\le\varepsilon|\nabla^2\varphi_{,xt}|_2^2+c/\varepsilon|\nabla\eta|_6^2 |\nabla\varphi_{,xxt}|_2|\nabla\varphi_{,xt}|_2\cr
&\le\varepsilon|\nabla^2\varphi_{,xt}|_2^2+c/\varepsilon^3|\nabla\eta|_6^4 |\nabla\varphi_{xt}|_2^2.\cr}
$$
Summarizing the estimates we have
$$\eqal{
|I_2|&\le\varepsilon_1(|\nabla\rot\psi_{,xt}|_2^2+\varepsilon \|\nabla\varphi_{,xt}\|_1^2)+c/\varepsilon[\|\eta_{,t}\|_1^2|\Delta v|_3^2\cr
&\quad+\|\eta\|_2^4|v_{,xt}|_2^2+\|\eta\|_2^2|v_{,xt}|_3^2]+ (1/\varepsilon_1^2\varepsilon+1/\varepsilon^3) \|\eta\|_2^4|\nabla\varphi_{,xt}|_2^2\cr
&\quad+{\mu\over a}\intop_\Omega{\eta\over a+\eta}|\nabla^2\varphi_{,xt}|^2dx.\cr}
$$
The term $J_{22}$ is absorbed by the second term on the l.h.s. of (\ref{3.40}) in the following way
$$\eqal{
&{\mu\over a}\intop_\Omega|\nabla^2\varphi_{,xt}|^2dx-{\mu\over a}\intop_\Omega {\eta\over a+\eta}|\nabla^2\varphi_{,xt}|^2dx={\mu\over a}\intop_\Omega{a\over a+\eta}|\nabla^2\varphi_{,xt}|^2dx\cr
&\ge{2\over3}{\mu\over a}|\nabla^2\varphi_{,xt}|_2^2,\cr}
$$
where $|\eta|\le a/2$. Now we estimate the third term on the r.h.s. of (\ref{3.40}). Performing differentiation with respect to time yields
$$\eqal{
I_3&=-{\nu\over a}\intop_\Omega\bigg[\bigg({\eta\over a+\eta}\bigg)_{,\eta} \eta_t\nabla\Delta\varphi\bigg]_{,x}\cdot\nabla\varphi_{,xt}dx\cr
&\quad-{\nu\over a}\intop_\Omega\bigg[{\eta\over a+\eta}\nabla\Delta \varphi_{,t}\bigg]_{,x}\nabla\varphi_{,xt}dx\equiv I_{31}+I_{32}.\cr}
$$
Integrating by parts in $I_{31}$ yields
$$
I_{31}={\nu\over a}\intop_\Omega\bigg({\eta\over a+\eta}\bigg)_{,\eta}\eta_t \Delta\nabla\varphi\cdot\nabla\varphi_{,xxt}dx
$$
so
$$\eqal{
|I_{31}|&\le\nu\varepsilon|\nabla\varphi_{,xxt}|_2^2+{c\nu\over\varepsilon} |\eta_t|_6^2|\Delta\nabla\varphi|_3^2\cr
&\le\nu\varepsilon|\nabla\varphi_{,xxt}|_2^2+{c\nu\over\varepsilon}|\eta_t|_6^2 |\Delta\nabla\varphi_x|_2|\Delta\nabla\varphi|_2\cr
&\le\nu\varepsilon|\nabla\varphi_{,xxt}|_2^2+\nu\varepsilon_2 |\nabla\varphi_{,xxx}|_2^2+{c\nu\over\varepsilon^2\varepsilon_2}|\eta_t|_6^4 |\Delta\nabla\varphi|_2^2.\cr}
$$
Performing differentiation with respect to $x$ in $I_{32}$ we have
$$\eqal{
I_{32}&=-{\nu\over a}\intop_\Omega\bigg({\eta\over a+\eta}\bigg)_{,\eta}\eta_x \nabla\Delta\varphi_{,t}\cdot\nabla\varphi_{,xt}dx\cr
&\quad-{\nu\over a}\intop_\Omega{\eta\over a+\eta}\nabla\Delta\varphi_{xt}\cdot \nabla\varphi_{,xt}dx\equiv L_1+L_2,\cr}
$$
where
$$\eqal{
|L_1|&\le\nu\varepsilon|\Delta\nabla\varphi_{,t}|_2^2+{c\nu\over\varepsilon} |\eta_{,x}|_6^2|\nabla\varphi_{,xt}|_3^2\cr
&\le\nu\varepsilon|\nabla\varphi_{,xxt}|_2^2+{c\nu\over\varepsilon} |\eta_{,x}|_6^2|\nabla\varphi_{,xxt}|_2|\nabla\varphi_{,xt}|_2\cr
&\le\nu\varepsilon|\nabla\varphi_{,xxt}|_2^2+{c\nu\over\varepsilon^3} |\eta_{,x}|_6^4|\nabla\varphi_{,xt}|_2^2.\cr}
$$
Integrating by parts in $L_2$ yields
$$\eqal{
L_2&={\nu\over a}\intop_\Omega\bigg({\eta\over a+\eta}\bigg)_{,\eta}\nabla\eta \Delta\varphi_{,xt}\cdot\nabla\varphi_{,xt}dx\cr
&\quad+{\nu\over a}\intop_\Omega{\eta\over a+\eta}|\Delta\varphi_{,xt}|^2dx \equiv L_2^1+L_2^2,\cr}
$$
where
$$\eqal{
|L_2^1|&\le\nu\varepsilon|\nabla\varphi_{,xxt}|_2^2+{\nu c\over\varepsilon}|\nabla\eta|_6^2|\nabla\varphi_{,xt}|_3^2\cr
&\le\nu\varepsilon|\nabla\varphi_{,xxt}|_2^2+{\nu c\over\varepsilon} 
|\nabla\eta|_6^2|\nabla\varphi_{,xxt}|_2|\nabla\varphi_{,xt}|_2\cr
&\le\nu\varepsilon|\nabla\varphi_{,xxt}|_2^2+{\nu c\over\varepsilon^3} |\nabla\eta|_6^4|\nabla\varphi_{,xt}|_2^2\cr}
$$
and $L_2^2$ is absorbed by the third term on the r.h.s. of (\ref{3.40}) in the following way
$$\eqal{
&{\nu\over a}\intop_\Omega|\Delta\varphi_{,xt}|^2dx-{\nu\over a}\intop_\Omega {\eta\over a+\eta}|\Delta\varphi_{,xt}|^2dx={\nu\over a}\intop_\Omega |\Delta\varphi_{,xt}|^2dx\cr
&\ge{2\over3}{\nu\over a}|\Delta\varphi_{,xt}|_2^2\quad {\rm for}\ \ |\eta|_\infty\le a/2.\cr}
$$
Summarizing, the third term on the r.h.s. of (\ref{3.40}) is estimated in the following way
$$\eqal{
|I_3|&\le\nu\varepsilon|\nabla\varphi_{,xxt}|_2^2+\nu\varepsilon_2 |\nabla\varphi_{,xxx}|_2^2+{c\nu\over\varepsilon^2\varepsilon_2}\|\eta_t\|_1^4 |\nabla\varphi_{,xx}|_2^2+{c\nu\over\varepsilon^3}\|\eta\|_2^4 |\nabla\varphi_{xt}|_2^2\cr
&\quad+{\nu\over a}\intop_\Omega{\eta\over a+\eta}|\Delta\varphi_{,xt}|^2dx.\cr}
$$
Consider the fourth term on the r.h.s. of (\ref{3.40}). Performing differentiations it takes the form
$$\eqal{
I_4&=\intop_\Omega v_{,xt}\cdot\nabla v\cdot\nabla\varphi_{,xt}dx+\intop_\Omega v_{,t}\cdot\nabla v_{,x}\cdot\nabla\varphi_{,xt}dx+\intop_\Omega v_{,x}\cdot\nabla v_{,t}\cdot\nabla\varphi_{,xt}dx\cr
&\quad+\intop_\Omega v\cdot\nabla v_{,xt}\cdot\nabla\varphi_{,xt}dx\equiv \sum_{i=1}^4 I_{4i}.\cr}
$$
First we consider
$$\eqal{
I_{44}&=\intop_\Omega v\cdot\!\nabla v_{xt}\cdot\!\nabla\varphi_{xt}dx=\intop_\Omega v\cdot\!\nabla\nabla\varphi_{xt}\cdot\!\nabla\varphi_{xt}dx+\intop_\Omega v\cdot\!\nabla\rot\psi_{xt}\cdot\!\nabla\varphi_{xt}dx\cr
&\equiv I_{44}^1+I_{44}^2,\cr}
$$
where
$$
I_{44}^1={1\over2}\intop_\Omega v\cdot\nabla|\nabla\varphi_{xt}|^2dx=-{1\over2} \intop_\Omega\Delta\varphi|\nabla\varphi_{xt}|^2dx,
$$
so
$$
|I_{44}^1|\le\varepsilon|\nabla\varphi_{xt}|_6^2+c/\varepsilon |\Delta\varphi|_3^2|\nabla\varphi_{xt}|_2^2.
$$
We integrate by parts in $I_{44}^2$. Then we have
$$\eqal{
I_{44}^2&=\intop_\Omega v_i\nabla_i(\rot\psi)_{jxt}\nabla_j\varphi_{xt}dx= -\intop_\Omega\nabla_jv_i\nabla_i(\rot\psi)_{jxt}\varphi_{xt}dx\cr
&=\intop_\Omega\nabla_j\Delta\varphi(\rot\psi)_{jxt}\varphi_{xt}dx+\intop_\Omega \nabla_jv_i(\rot\psi)_{jxt}\nabla_i\varphi_{xt}dx\cr
&\equiv K_1+K_2,\cr}
$$
where
$$
|K_1|\le\varepsilon|\varphi_{xt}|_\infty^2+c/\varepsilon|\Delta\nabla\varphi|_2^2 |\rot\psi_{xt}|_2^2
$$
and
$$\eqal{
|K_2|&\le\varepsilon|\nabla\varphi_{xt}|_6^2+c/\varepsilon|\nabla v|_3^2|\rot\psi_{xt}|_2^2\cr
&\le\varepsilon|\nabla\varphi_{xt}|_6^2+c/\varepsilon|\nabla v|_3^2(|v_{xt}|_2^2+|\nabla\varphi_{xt}|_2^2).\cr}
$$
Next, we examine
$$
I_{41}=-\intop_\Omega\Delta\varphi_{xt}v\cdot\nabla\varphi_{,xt}dx-\intop_\Omega v_{xt}v\cdot\nabla^2\varphi_{xt}dx\equiv I_{41}^1+I_{41}^2,
$$
where
$$\eqal{
&|I_{41}^1|\le\varepsilon|\Delta\varphi_{,xt}|_2^2+c/\varepsilon |\nabla\varphi_{,xt}|_3^2|v|_6^2\le\varepsilon|\nabla\varphi_{,xxt}|_2^2+ c/\varepsilon^3|v|_6^4|\nabla\varphi_{,xt}|_2^2,\cr
&|I_{41}^2|\le\varepsilon|\nabla^2\varphi_{,xt}|_2^2+c/\varepsilon |v_{,xt}|_3^2|v|_6^2.\cr}
$$
The second term in $I_4$ is bounded by
$$
|I_{42}|\le\varepsilon|\nabla\varphi_{,xt}|_6^2+c/\varepsilon|v_t|_2^2 |v_{xx}|_3^2.
$$
Finally, we have the estimate
$$
|I_{43}|\le\varepsilon|\nabla\varphi_{,xt}|_6^2+c/\varepsilon|v_x|_3^2|\nabla v_t|_2^2.
$$
Summarizing, we have the estimate
$$\eqal{
|I_4|&\le\varepsilon\|\nabla\varphi_{,t}\|_2^2+c/\varepsilon[|\Delta\varphi|_3^2 |\nabla\varphi_{,xt}|_2^2+|v|_6^4|\nabla\varphi_{,xt}|_2^2+|v|_6^2|v_{,xt}|_3^2\cr
&\quad+|v_{,t}|_2^2|v_{,xx}|_3^2+|\nabla v|_3^2(|v_{,xt}|_2^2+ |\nabla\varphi_{,xt}|_2^2)+|\nabla\varphi_{,xx}|_2^2|\rot\psi_{,xt}|_2^2].\cr}
$$
The fifth term on the r.h.s. of (\ref{3.40}) is expressed in the form
$$\eqal{
I_5&=-\intop_\Omega\bigg({1\over a+\eta}\bigg)_{,\eta}\eta_t(p_\varrho(a)- p_\varrho(a+\eta))\nabla\eta\cdot\nabla\varphi_{,xxt}dx\cr
&\quad+\intop_\Omega{1\over a+\eta}p_{\varrho\varrho}\eta_t\nabla\eta\cdot \nabla\varphi_{,xxt}dx\cr
&\quad+\intop_\Omega{1\over a+\eta}(p_\varrho(a)-p_\varrho(a+\eta))\nabla\eta_t \cdot\nabla\varphi_{,xxt}dx\equiv I_{51}+I_{52}+I_{53}.\cr}
$$
Hence, we have
$$\eqal{
&|I_{51}|\le\varepsilon|\nabla\varphi_{,xxt}|_2^2+c/\varepsilon|\eta_t|_6^2 |\eta|_6^2|\nabla\eta|_6^2,\cr
&|I_{52}|\le\varepsilon|\nabla\varphi_{,xxt}|_2^2+c/\varepsilon|\eta_t|_6^2 |\nabla\eta|_3^2,\cr
&|I_{53}|\le\varepsilon|\nabla\varphi_{,xxt}|_2^2+c/\varepsilon|\eta|_\infty^2 |\nabla\eta_t|_2^2.\cr}
$$
Finally, the last term on the r.h.s. of (\ref{3.40}) is bounded by
$$
\varepsilon|\nabla\varphi_{,xxt}|_2^2+c/\varepsilon|f_t|_2^2.
$$
Employing the above estimates in (\ref{3.40}) yields
\begin{equation}\eqal{
&{d\over dt}|\nabla\varphi_{,xt}|_2^2+{\mu\over a}|\nabla^2\varphi_{,xt}|_2^2+{\nu\over a}|\Delta\varphi_{,xt}|_2^2\le\varepsilon_1|\nabla\rot\psi_{,xt}|_2^2\cr
&\quad+\nu\varepsilon_2|\nabla\varphi_{,xxx}|_2^2+{c\over\nu}[\|\eta_{,t}\|_1^2 \|\eta\|_2^4+\|\eta_{,t}\|_1^2\|\eta\|_2^2+\|\eta_{,t}\|_1^2+ \|\eta_t\|_1^2|\Delta\varphi|_3^2\cr
&\quad+\|\eta\|_2^4|v_{,xt}|_2^2+(|\Delta\varphi|_3^2+|v|_6^4+\|\eta\|_2^4+ |v_{,x}|_3^3)|\nabla\varphi_{,xt}|_2^2\cr
&\quad+|v|_6^2|v_{,xt}|_3^2+|v_{,t}|_2^2|v_{,xx}|_3^2+|\nabla\varphi_{,xx}|_2^2 |\rot\psi_{,xt}|_2^2\cr
&\quad+|v_{,x}|_3^2|\nabla v_{,t}|_2^2+|f_{,t}|_2^2]+{c\over\varepsilon_1^2\nu} \|\eta\|_2^4|\nabla\varphi_{,xt}|_2^2\cr
&\quad+{\nu c\over\varepsilon_2}\|\eta_{,t}\|_1^4|\nabla\varphi_{,xx}|_2^2.\cr}
\label{3.41}
\end{equation}
Differentiate (\ref{3.13}) with respect to $x$ and $t$, multiply by $\rot\psi_{,xt}$ and integrate over $\Omega$. Then we have
\begin{equation}\eqal{
&{1\over2}{d\over dt}|\rot\psi_{,xt}|_2^2+{\mu\over a}|\nabla\rot\psi_{,xt}|_2^2 =-{\mu\over a}\intop_\Omega\bigg({\eta\over a+\eta}\Delta v\bigg)_{,xt}\cdot \rot\psi_{xt}dx\cr
&\quad-{\nu\over a}\intop_\Omega\bigg({\eta\over a+\eta}\Delta\nabla \varphi\bigg)_{,xt}\cdot\rot\psi_{xt}dx-\intop_\Omega(v\cdot\nabla v)_{,xt}\cdot \rot\psi_{xt}dx\cr
&\quad+\intop_\Omega f_{,xt}\cdot\rot\psi_{,xt}dx.\cr}
\label{3.42}
\end{equation}
Now we estimate the terms from the r.h.s. of (\ref{3.42}). We express the first term in the form
$$\eqal{
I_1&=-{\mu\over a}\intop_\Omega\bigg[\bigg({\eta\over a+\eta}\bigg)_{,\eta}\eta_t \Delta v\bigg]_{,x}\rot\psi_{,xt}dx\cr
&\quad-{\mu\over a}\intop_\Omega\bigg[{\eta\over a+\eta}\Delta v_{,t}\bigg]_{,x} \cdot\rot\psi_{,xt}dx\equiv I_{11}+I_{12}.\cr}
$$
Integrating by parts in $I_{11}$ yields
$$
I_{11}={\mu\over a}\intop_\Omega\bigg({\eta\over a+\eta}\bigg)_{,\eta}\eta_t \Delta v\cdot\rot\psi_{,xxt}dx
$$
and
$$\eqal{
|I_{11}|&\le\varepsilon|\rot\psi_{,xxt}|_2^2+c/\varepsilon|\eta_t|_6^2|\Delta v|_3^2\le\varepsilon|\rot\psi_{,xxt}|_2^2+c/\varepsilon|\eta_t|_6^2 |v_{,xxx}|_2^{5/3}|v|_2^{1/3}\cr
&\le\varepsilon|\rot\psi_{,xxt}|_2^2+\varepsilon_1|v_{,xxx}|_2^2+ c/(\varepsilon\varepsilon_1)|\eta_t|_6^{12}A_1^2.\cr}
$$
Performing differentiation with respect to $x$ in $I_{12}$ implies
$$\eqal{
I_{12}&=-{\mu\over a}\intop_\Omega\bigg({\eta\over a+\eta}\bigg)_{,\eta}\eta_x \Delta v_t\cdot\rot\psi_{,xt}dx\cr
&\quad-{\mu\over a}\intop_\Omega{\eta\over a+\eta}\Delta v_{,xt}\cdot\rot\psi_{,xt}dx\equiv I_{12}^1+I_{12}^2,\cr}
$$
where
$$\eqal{
|I_{12}^1|&\le\varepsilon|v_{,xxt}|_2^2+c/\varepsilon|\eta_{,x}|_6^2 |\rot\psi_{,xt}|_3^2\cr
&\le\varepsilon|v_{,xxt}|_2^2+c/\varepsilon\|\eta\|_2^2|\rot\psi_{,xxt}|_2 |\rot\psi_{,xt}|_2\cr
&\le\varepsilon(v_{,xxt}|_2^2+|\rot\psi_{,xxt}|_2^2)+c/\varepsilon\|\eta\|_2^4 |\rot\psi_{xt}|_2^2.\cr}
$$
To examine $I_{12}^2$ we write it in the form
$$\eqal{
I_{12}^2&=-{\mu\over a}\intop_\Omega{\eta\over a+\eta}\Delta\rot\psi_{,xt}\cdot \rot\psi_{,xt}dx\cr
&\quad-{\mu\over a}\intop_\Omega{\eta\over a+\eta}\Delta\nabla\varphi_{,xt}\cdot \rot\psi_{,xt}dx\equiv J_1+J_2.\cr}
$$
Integrating by parts in $J_1$ yields
$$\eqal{
J_1&={\mu\over a}\intop_\Omega\bigg({\eta\over a+\eta}\bigg)_{,\eta} \nabla\eta\nabla\rot\psi_{,xt}\cdot\rot\psi_{,xt}dx\cr
&\quad+{\mu\over a}\intop_\Omega{\eta\over a+\eta}|\nabla\rot\psi_{,xt}|^2dx \equiv J_{11}+J_{12},\cr}
$$
where
$$\eqal{
|J_{11}|&\le\varepsilon|\nabla\rot\psi_{,xt}|_2^2+c/\varepsilon|\nabla\eta|_6^2 |\rot\psi_{,xt}|_3^2\cr
&\le\varepsilon|\nabla\rot\psi_{,xt}|_2^2+c/\varepsilon\|\eta\|_2^2 |\rot\psi_{,xxt}|_2|\rot\psi_{,xt}|_2\cr
&\le\varepsilon|\nabla\rot\psi_{,xt}|_2^2+c/\varepsilon\|\eta\|_2^4 |\rot\psi_{,xt}|_2^2\cr}
$$
and $J_{12}$ is absorbed by the second term on the l.h.s. of (\ref{3.42}) in the following way
$$\eqal{
&{\mu\over a}\intop_\Omega|\nabla\rot\psi_{,xt}|^2dx-{\mu\over a}\intop_\Omega {\eta\over a+\eta}|\nabla\rot\psi_{,xt}|^2dx\cr
&={\mu\over a}\intop_\Omega{a\over a+\eta}|\nabla\rot\psi_{,xt}|^2dx\ge{2\over3}{\mu\over a}|\nabla\rot\psi_{,xt}|_2^2,\cr}
$$
where the last inequality holds for $|\eta|_\infty\le a/2$.

\noindent
Integrating by parts in $J_2$ yields
$$
J_2={\mu\over a}\intop_\Omega\bigg({\eta\over a+\eta}\bigg)_{,\eta}\nabla\eta \cdot\rot\psi_{,xt}\Delta\varphi_{,xt}dx.
$$
Hence,
$$\eqal{
|J_2|&\le\varepsilon|\Delta\varphi_{,xt}|_2^2+c/\varepsilon|\nabla\eta|_6^2 |\rot\psi_{,xt}|_3^2\cr
&\le\varepsilon(|\Delta\varphi_{,xt}|_2^2+|\nabla\rot\psi_{,xt}|_2^2)+ c/\varepsilon\|\eta\|_2^4|\rot\psi_{,xt}|_2^2.\cr}
$$
Next, we examine the second term on the r.h.s. of (\ref{3.42}). We express it in the form
$$\eqal{
I_2&=-{\nu\over a}\intop_\Omega\bigg[\bigg({\eta\over a+\eta}\bigg)_{,\eta} \eta_t\nabla\Delta\varphi\bigg]_{,x}\cdot\rot\psi_{,xt}dx\cr
&\quad-{\nu\over a}\intop_\Omega\bigg[{\eta\over a+\eta}\nabla\Delta \varphi_{,t}\bigg]_{,x}\rot\psi_{,xt}dx\equiv I_{21}+I_{22}.\cr}
$$
Integrating by parts in $I_{21}$ implies the estimate
$$\eqal{
|I_{21}|&\le\varepsilon|\rot\psi_{,xxt}|_2^2+c\nu^2/\varepsilon|\eta_t|_6^2 |\Delta\nabla\varphi|_3^2\cr
&\le\varepsilon|\rot\psi_{,xxt}|_2^2+{c\nu^2\over\varepsilon}\varepsilon_1 |\nabla\varphi_{,xxx}|_2^2+{c\nu^2\over\varepsilon\varepsilon_1}\|\eta_t\|_1^4 |\nabla\varphi_{xx}|_2^2.\cr}
$$
Performing differentiation in $I_{22}$ yields
$$\eqal{
I_{22}&=-{\nu\over a}\intop_\Omega\bigg({\eta\over a+\eta}\bigg)_{,\eta}\eta_x \Delta\nabla\varphi_{,t}\cdot\rot\psi_{,xt}dx\cr
&\quad-{\nu\over a}\intop_\Omega{\eta\over a+\eta}\Delta\nabla\varphi_{,tx}\cdot \rot\psi_{,xt}dx\equiv L_1+L_2,\cr}
$$
where
$$\eqal{
|L_1|&\le\nu^2\varepsilon_1|\Delta\nabla\varphi_{,t}|_2^2+{c\over\varepsilon_1} |\eta_{,x}|_6^2|\rot\psi_{,xt}|_3^2\cr
&\le\nu^2\varepsilon_1|\Delta\nabla\varphi_{,t}|_2^2+c/\varepsilon_1\|\eta\|_2^2 |\rot\psi_{,xxt}|_2|\rot\psi_{,xt}|_2\cr
&\le\nu^2\varepsilon_1|\Delta\nabla\varphi_{,t}|_2^2+\varepsilon |\rot\psi_{,xxt}|_2^2+c/(\varepsilon\varepsilon_1)\|\eta\|_2^4 |\rot\psi_{,xt}|_2^2.\cr}
$$
Integrating by parts in $L_2$ gives
$$
L_2={\nu\over a}\intop_\Omega\bigg({\eta\over a+\eta}\bigg)_{,\eta}\nabla\eta \cdot\rot\psi_{,xt}\Delta\varphi_{,xt}dx.
$$
Hence,
$$\eqal{
|L_2|&\le\nu^2\varepsilon_1|\Delta\varphi_{,xt}|_2^2+c/\varepsilon_1 |\nabla\eta|_6^2|\rot\psi_{,xt}|_3^2\cr
&\le\nu^2\varepsilon_1|\Delta\varphi_{,xt}|_2^2+\varepsilon |\rot\psi_{,xxt}|_2^2+c/(\varepsilon\varepsilon_1)\|\eta\|_2^4 |\rot\psi_{,xt}|_2^2.\cr}
$$
Now we consider the third term on the r.h.s. of (\ref{3.42}). Performing differentiations we write it in the form
$$\eqal{
I_3&=\intop_\Omega v_{,xt}\cdot\nabla v\cdot\rot\psi_{,xt}dx+\intop_\Omega v_{,x}\cdot\nabla v_{,t}\cdot\rot\psi_{,xt}dx\cr
&\quad+\intop_\Omega v_{,t}\cdot\nabla v_{,x}\cdot\rot\psi_{,xt}dx+\intop_\Omega v\cdot\nabla v_{,xt}\cdot\rot\psi_{,xt}dx\cr
&\equiv I_{31}+I_{32}+I_{33}+I_{34}.\cr}
$$
Consider the sum
$$\eqal{
I_{31}+I_{33}&=-\intop_\Omega v_{,t}\cdot\nabla v_{,x}\cdot\rot\psi_{,xt}dx-\intop_\Omega v_{,t}\cdot\nabla v\cdot\rot\psi_{,xxt}dx\cr
&\quad+\intop_\Omega v_{,t}\cdot\nabla v_{,x}\cdot\rot\psi_{,xt}dx=-\intop_\Omega v_{,t}\cdot\nabla v\cdot\rot\psi_{,xxt}dx\equiv I_0,\cr}
$$
where
$$\eqal{
|I_0|&\le\varepsilon|\rot\psi_{,xxt}|_2^2+c/\varepsilon|v_{,x}|_\infty^2 |v_{,t}|_2^2\le\varepsilon|\rot\psi_{,xxt}|_2^2+c/\varepsilon |v_{,x}|_\infty^2D_2^2\cr
&\le\varepsilon|\rot\psi_{,xxt}|_2^2+c/\varepsilon|v_{,xxx}|_2^{5/3} |v|_2^{1/3}D_2^2\cr
&\le\varepsilon|\rot\psi_{,xxt}|_2^2+\varepsilon_1|v_{,xxx}|_2^2+ c/(\varepsilon\varepsilon_1)|v|_2^2D_2^{12}\cr
&\le\varepsilon|\rot\psi_{,xxt}|_2^2+\varepsilon_1|v_{,xxx}|_2^2+ c/(\varepsilon\varepsilon_1)A_1^2D_2^{12}.\cr}
$$
Next, we consider
$$\eqal{
I_{32}+I_{34}&=-\intop_\Omega v\cdot\nabla v_{,xt}\cdot\rot\psi_{,xt}dx-\intop_\Omega v\cdot\nabla v_{,t}\cdot\rot\psi_{,xxt}dx\cr
&\quad+\intop_\Omega v\cdot\nabla v_{,xt}\cdot\rot\psi_{,xt}dx=-\intop_\Omega v\cdot\nabla v_{,t}\cdot\rot\psi_{,xxt}dx\equiv I_\infty.\cr}
$$
Hence, we have
$$\eqal{
|I_\infty|&\le\varepsilon|\rot\psi_{,xxt}|_2^2+c/\varepsilon|v_{,xt}|_3^2 |v|_6^2\le\varepsilon|\rot\psi_{,xxt}|_2^2+c/\varepsilon|v_{,xt}|_3^2D_1^2\cr
&\le\varepsilon|\rot\psi_{,xxt}|_2^2+c/\varepsilon|v_{xxt}|_2^{3/2}|v_t|_2^{1/2} D_1^2\le\varepsilon|\rot\psi_{,xxt}|_2^2+\varepsilon_1|v_{,xxt}|_2^2\cr
&\quad+c/(\varepsilon_1\varepsilon)D_1^8D_2^2.\cr}
$$
Summarizing, we have
$$
|I_3|\le\varepsilon|\rot\psi_{,xxt}|_2^2+\varepsilon_1(|v_{,xxx}|_2^2+ |v_{,xxt}|_2^2)+c/(\varepsilon\varepsilon_1)(A_1^2D_2^{12}+D_1^8D_2^2).
$$
Finally, the last term on the r.h.s. of (\ref{3.42}) is bounded by
$$
\varepsilon|\rot\psi_{,xxt}|_2^2+c/\varepsilon|f_t|_2^2.
$$
Employing the above estimates in (\ref{3.42}) and assuming that $\varepsilon$ is sufficiently small implies
\begin{equation}\eqal{
&{d\over dt}|\rot\psi_{,xt}|_2^2+{\mu\over a}|\nabla\rot\psi_{,xt}|_2^2\le \varepsilon(|v_{,xxx}|_2^2+|v_{,xxt}|_2^2)\cr
&\quad+\varepsilon_1\nu^2(|\nabla\varphi_{,xxt}|_2^2+|\nabla\varphi_{,xxx}|_2^2)\cr
&\quad+c/\varepsilon[|\eta_t|_6^{12}|v|_2^2+\|\eta\|_2^4 |\rot\psi_{,xt}|_2^2\cr
&\quad+D_2^{12}|v|_2^2+D_1^8|v_t|_2^2+|f_t|_2^2]+ {c\over\varepsilon_1}(\nu^2\|\eta_t\|_1^4|\nabla\varphi_{,xx}|_2^2+ \|\eta\|_2^4|\rot\psi_{xt}|_2^2).\cr}
\label{3.43}
\end{equation}
Multiplying (\ref{3.43}) by $1/\nu$ and adding to (\ref{3.41}) we obtain
\begin{equation}\eqal{
&{d\over dt}\bigg(|\nabla\varphi_{xt}|_2^2+{1\over\nu}|\rot\psi_{xt}|_2^2\bigg)+ {\mu\over a}|\nabla^2\varphi_{,xt}|_2^2+{\nu\over a} |\Delta\varphi_{xt}|_2^2+{\mu\over a\nu}|\nabla\rot\psi_{xt}|_2^2\cr
&\le{\varepsilon\over\nu}|v_{,xxx}|_2^2+\varepsilon_1\nu|\nabla\varphi_{xxx}|_2^2 +{c\over\nu\varepsilon} [|\eta_t|_6^{12}|v|_2^2+ |\eta|_{2,1}^2(|\eta|_{2,1}^4+1)+(\|\eta\|_2^4\cr
&\quad+ |\nabla\varphi_{,xx}|_2^2)|\rot\psi_{,xt}|_2^2+(\|\eta\|_2^4+|\Delta\varphi|_3^2+|\nabla v|_3^2)|\nabla\varphi_{,xt}|_2^2+D_1^2|\nabla\varphi_{,xt}|_3^2\cr
&\quad+(\|\eta_t\|_1^2+ D_2^2)|v_{,xx}|_3^2+(|v|_6^4+\|\eta\|_2^4+\|\eta\|_2^2+D_1^2)|v_{,xt}|_3^2+|\nabla v|_3^2|\nabla v_{,t}|_2^2\cr
&\quad+D_2^{12}|v|_2^2+D_1^8|v_t|_2^2+|f_t|_2^2]+{c\nu\over\varepsilon_1} [\|\eta_t\|_1^4|\nabla\varphi_{,xx}|_2^2+\|\eta\|_2^4|\nabla\varphi_{,xt}|_2^2].\cr}
\label{3.44}
\end{equation}
Let
$$
\phi_5=|\eta_t|_6^{12}|v|_2^2+|\eta|_{2,1}^2(|\eta|_{2,1}^4+1)+D_2^{12}|v|_2^2+ D_1^8|v_{,t}|_2^2,\quad |v|_2\le A_1,\quad |v_{,t}|_2\le D_2.
$$
Simplifying, we write (\ref{3.44}) in the form
\begin{equation}\eqal{
&a{d\over dt}(|\nabla\varphi_{,xt}|_2^2+{1\over\nu}|\rot\psi_{,xt}|_2^2)+\mu \bigg(|\nabla\varphi_{,xxt}|_2^2+{1\over\nu}|\rot\psi_{,xxt}|_2^2\bigg)\cr
&\quad+\nu|\Delta\varphi_{,xt}|_2^2\le{\varepsilon\over\nu}|v_{,xxx}|_2^2+ {\varepsilon_1\over\nu}|\nabla\varphi_{,xxx}|_2^2+{c\over\nu\varepsilon} [\phi_5(|\eta|_{2,1},D_1,D_2,A_1)\cr
&\quad+\phi_6(|\eta|_{2,1},D_1,D_2)(|\rot\psi_{,xt}|_2^2+ |\nabla\varphi_{,xt}|_2^2+|\nabla\varphi_{,xt}|_3^2+|v_{,xx}|_3^2\cr
&\quad+|v_{,xt}|_3^2+|v|_2^2+|v_{,t}|_2^2+|\nabla v|_3^2|\nabla\varphi_{,xt}|_2^2+|\nabla\varphi_{,xx}|_2^2 |\rot\psi_{,xt}|_2^2)\cr
&\quad+|\nabla v|_3^2|\nabla v_{,t}|_2^2+|f_t|_2^2]+ {c\nu\over\varepsilon_1}\phi_7(|\eta|_{2,1})(|\nabla\varphi_{,xx}|_2^2+ |\nabla\varphi_{,xt}|_2^2),\cr}
\label{3.45}
\end{equation}
where
$$\eqal{
\phi_6&=\|\eta\|_2^2(\|\eta\|_2^2+1)+\|\eta_t\|_1^2+D_1^2+D_2^2+D_1^8+D_2^{12},\cr \phi_7&=|\eta|_{2,1,\infty,\Omega^t}^2(|\eta|_{2,1,\infty,\Omega^t}^2+1)\cr}
$$
Using the interpolations
\begin{equation}\eqal{
&\alpha|\rot\psi_{,xt}|_2^2\le c\alpha|\rot\psi_{,xxt}|_2^{2/3}|\rot\psi_t|_2^{4/3}\le\varepsilon|\rot\psi_{,xxt}|_2^2\cr
&\quad+c(1/\varepsilon)\alpha^{3/2}|\rot\psi_{,t}|_2^2,\cr
&\alpha|\nabla\varphi_{,xt}|_3^2\le c\alpha|\nabla\varphi_{,xxt}|_2|\nabla\varphi_{,t}|_2\le\varepsilon |\nabla\varphi_{,xxt}|_2^2+c(1/\varepsilon)\alpha^2|\nabla\varphi_{,t}|_2^2,\cr
&\alpha|v_{,xx}|_3^2\le\alpha|v_{,xxx}|_2^{5/3}|v|_2^{1/3}\le\varepsilon |v_{,xxx}|_2^2+c(1/\varepsilon)\alpha^6|v|_2^2,\cr}
\label{3.46}
\end{equation}
where $\alpha$ is a parameter, in (\ref{3.45}) yields
\begin{equation}\eqal{
&a{d\over dt}(|\nabla\varphi_{,xt}|_2^2+{1\over\nu}|\rot\psi_{,xt}|_2^2)+\mu (|\nabla\varphi_{,xxt}|_2^2+{1\over\nu}|\rot\psi_{,xxt}|_2^2)\cr
&\quad+\nu|\Delta\varphi_{,xt}|_2^2\le{\varepsilon\over\nu}|v_{,xxx}|_2^2+ \nu\varepsilon_1|\nabla\varphi_{,xxx}|_2^2\cr
&\quad+{c\over\nu\varepsilon}[\phi_8+|\nabla v|_3^2|\nabla\varphi_{,xt}|_2^2+|\nabla v|_3^2|\nabla v_{,t}|_2^2+|f_t|_2^2]\cr
&\quad+{c\nu\over\varepsilon_1}\phi_7(|\eta|_{2,1})(|\nabla\varphi_{,xx}|_2^2+ |\nabla\varphi_{,xt}|_2^2),\cr}
\label{3.47}
\end{equation}
where
$$\eqal{
\phi_8&=\phi_5+(\phi_6^{3/2}+\phi_6^2+\phi_6^6)|v_{,t}|_2^2+\phi_6^6|v|_2^2+ |\nabla\varphi_{,xx}|_2^3\cdot\cr
&\quad\cdot(D_2^2+|\nabla\varphi_{,x}|_2^2),\quad |v_{,t}|_2\le D_2,\quad |v|_2\le A_1,\cr}
$$
and $\phi_5$, $\phi_6$, $\phi_7$ are introduced below (\ref{3.44}) and (\ref{3.45}), respectively. Hence (\ref{3.47}) implies (\ref{3.38}). Multiplying (\ref{3.47}) by $\nu$, integrating with respect to time and using that 
$$\eqal{
&\intop_0^t|v|_2^2dt'\le c\intop_0^t|\nabla v|_2^2dt'\le A_1^2,\cr
&\intop_0^t|v_{,t}|_2^2dt'\le c\intop_0^t|\nabla v_{,t}|_2^2dt'\le D_2^2,\cr}
$$
where we used the Poincar\'e inequality because $\intop_\Omega vdx=0$, we obtain the inequality
\begin{equation}\eqal{
X_1^2&\equiv a(\nu|\nabla\varphi_{,xt}(t)|_2^2+|\rot\psi_{,xt}(t)|_2^2)\cr
&\quad+\mu(\nu|\nabla\varphi_{,xxt}|_{2,\Omega^t}^2+|\rot\psi_{,xxt}|_{2,\Omega^t}^2) +\nu^2|\Delta\varphi_{,xt}|_{2,\Omega^t}^2\cr
&\le\varepsilon|v_{,xxx}|_{2,\Omega^t}^2+\nu^2\varepsilon_1 |\nabla\varphi_{,xxx}|_{2,\Omega^t}^2+{c\over\varepsilon}\bigg[\intop_0^t\phi_8dt'\cr
&\quad+|\nabla\varphi_{,xt}|_{2,\infty,\Omega^t}^2|\nabla v|_{3,2,\Omega^t}^2+ |\nabla v|_{3,\infty,\Omega^t}^2D_2^2+|f_t|_{2,\Omega^t}^2\bigg]\cr
&\quad+{c\nu^2\over\varepsilon_1}\phi_7(|\eta|_{2,1,\infty,\Omega^t}) (|\nabla\varphi_{,xx}|_{2,\Omega^t}^2+|\nabla\varphi_{,xt}|_{2,\Omega^t}^2)\cr
&\quad+(\nu|\nabla\varphi_{,xt}(0)|_2^2+|\rot\psi_{,xt}(0)|_2^2),\cr}
\label{3.48}
\end{equation}
where
\begin{equation}\eqal{
\intop_0^t\phi_8dt'&\le\intop_0^t\phi_5dt'+[(|\phi_6|_{\infty,(0,t)}^{3/2}+ |\phi_6|_{\infty,(0,t)}^2+|\phi_6|_{\infty,(0,t)}^6)D_2^2\cr
&\quad+|\phi_6|_{\infty,(0,t)}^6A_1^2]+|\nabla\varphi_{,xx}|_{2,\Omega^t}^2 |\nabla\varphi_{,xx}|_{2,\infty,\Omega^t}(D_2^2+ |\nabla\varphi_{,x}|_{2,\infty,\Omega^t}^2)\cr}
\label{3.49}
\end{equation}
and
\begin{equation}\eqal{
\intop_0^t\phi_5dt'&\le|\eta_t|_{6,\infty,\Omega_t}^{12}A_1^2+ |\eta|_{2,1,2,\Omega^t}^2(|\eta|_{2,1,\infty,\Omega^t}^2+1)\cr
&\quad+D_2^{12}A_1^2+D_1^8D_2^2,\cr}
\label{3.50}
\end{equation}
\begin{equation}\eqal{
|\phi_6|_{\infty,(0,t)}&\le\|\eta\|_{2,\infty,\Omega^t}^2(\|\eta\|_{2,\infty,\Omega^t}^2+1) +\|\eta_t\|_{1,\infty,\Omega^t}^2+D_1^2\cr
&\quad+D_2^2+D_1^8+D_2^{12}.\cr}
\label{3.51}
\end{equation}
In view of Notation \ref{n2.13} we write (\ref{3.48}) in the form
\begin{equation}\eqal{
X_1^2(t)&\le\varepsilon|v_{,xxx}|_{2,\Omega^t}^2+\nu^2\varepsilon_1 |\nabla\varphi_{,xxx}|_{2,\Omega^t}^2+{c\over\varepsilon}\bigg[\intop_0^t\phi_8 dt'\cr
&\quad+{\chi_0^2\over\nu}|\nabla v|_{3,2,\Omega^t}^2+|\nabla v|_{3,\infty,\Omega^t}^2D_2^2+|f_t|_{2,\Omega^t}^2\bigg]\cr
&\quad+{c\nu^2\over\varepsilon_1}\phi_7(|\eta|_{2,1,\infty,\Omega^t}) {\psi^2\over\nu^2}+a\Phi_0^2(0),\cr}
\label{3.52}
\end{equation}
and (\ref{3.49}) is replaced by
\begin{equation}
\intop_0^t\phi_8dt'\le\intop_0^t\phi_5dt'+\phi_9(|\phi_6|_{\infty(0,t)},A_1D_2)+ {\chi_0\over\sqrt{\nu}}{\psi^2\over\nu^2}\cdot\bigg(D_2^2+{\chi_0^2\over\nu}\bigg),
\label{3.53}
\end{equation}
where $\phi_9$ replaces the expression under the square bracket on the r.h.s. of (\ref{3.49}).

\noindent
Inequality (\ref{3.47}) implies (\ref{3.38}). From (\ref{3.52}) we derive (\ref{3.39}). This concludes the proof.
\end{proof}

\begin{remark}\label{r3.3}
Now we want to collect all derived above inequaliteis and write them in the most simple form.
\end{remark}

\noindent
In view of periodicity we have
$$\eqal{
&\intop_\Omega\nabla\varphi dx=\intop_\Omega\rot\psi dx=0,\cr
&\intop_\Omega\nabla\varphi_tdx=\intop_\Omega\rot\psi_tdx=0.\cr}
$$
Hence, we have
$$
\intop_0^t|v|_6^2dt'\le A_1^2+G^2,\quad \intop_0^t|v_t|_2^2dt'\le D_2^2.
$$
From Notation \ref{n2.13} we have
$$\eqal{
&\chi_1^2(t)=\nu|\nabla\varphi|_{2,1,\infty,\Omega^t}^2+ |\rot\psi|_{2,1,\infty,\Omega^t}^2,\cr
&\chi_2^2(t)=\nu|\nabla\varphi|_{3,1,2,\Omega^t}^2+ |\rot\psi|_{3,1,2,\Omega^t}^2,\cr
&\Psi(t)=\nu|\nabla\varphi|_{3,1,2,\Omega^t}.\cr}
$$
We additionally introduce the quantities
$$\eqal{
&\Phi_0^2(t)=\nu|\nabla\varphi(t)|_{2,1}^2+|\rot\psi(t)|_{2,1}^2\cr
&\chi_0^2(t)=\nu|\nabla\varphi|_{2,1,\infty,\Omega^t}^2.\cr}
$$
From (\ref{2.93}) we have
$$\eqal{
&|\eta(t)|_r\le\exp\bigg(t^{1/2}{\Psi\over\nu}\bigg)\bigg[t^{1/2}{\Psi\over\nu}+ |\eta(0)|_r\bigg],\cr
&|\eta(t)|_{2,1}\le\exp(ct^{1/2}\Phi_0(t))\bigg[t^{1/2}{\Psi\over\nu}+ |\eta(0)|_{2,1}\bigg],\cr
&D_1=\exp(ct^{1/2}\Phi_0)\bigg[t^{2/3}{\Psi\over\nu}+t^{1/6}|\eta(0)|_3\bigg]+ \phi\bigg({\Psi\over\nu}\bigg)+A_2,\quad \alpha>0,\cr}
$$
where
$$\eqal{
&A_2=|f|_{3,6,\Omega^t}+|\varrho_0|_\infty^{1/6}|v_0|_6,\cr
&A_1^2=2|\varrho_0|_1|f|_{\infty,1,\Omega^t}^2+{3\over2}\intop_\Omega\bigg( {1\over2}\varrho_0v_0^2+{A\over\varkappa-1}\varrho_0^\varkappa\bigg)dx,\cr
&|\eta|_{2,1,2,\Omega^t}\le\exp(ct^{1/2}\Phi_0)\bigg[t^{3/2}{\Psi\over\nu}+ t^{1/2}|\eta(0)|_{2,1}\bigg],\cr
&D_2=\exp\bigg[\intop_0^t|\eta_t(t')|_3^2dt'+\|\eta\|_{2,\infty,\Omega^t}^2A_1^2+ (1+D_1^2)A_1^2\bigg]\cdot\cr
&\quad\cdot\bigg[\|\eta_t\|_{1,2,\Omega^t}^2+\|\eta_t\|_{1,\infty,\Omega^t}^2 \bigg(1+{\Psi^2\over\nu^2}A_1^2+D_1^2A_1\bigg)+{\Psi^2\over\nu^2}A_1^2\cr
&\quad+(1+\|\eta_t\|_{1,\infty,\Omega^t}^2)|f_t|_{2,\Omega^t}^2\bigg].\cr}
$$
Integrating (\ref{3.37}) with respect to time yields
\begin{equation}\eqal{
&a(\nu|\nabla\varphi(t)|_{1,1}^2+|\rot\psi(t)|_{1,1}^2)+ \mu(\nu|\nabla\varphi|_{2,1,2,\Omega^t}^2+ |\rot\psi|_{2,1,2,\Omega^t}^2)\cr
&\quad+\nu^2|\nabla\varphi|_{2,1,2,\Omega^t}^2\le c[|\eta|_{2,2,\Omega^t}^2+ |\eta|_{3,\infty,\Omega^t}^2D_2^2+(1+|\eta|_{\infty,\Omega^t}^2)A_1^4D_1^2\cr
&\quad+(1+|\eta|_{\infty,\Omega^t}^2)^4D_1^8A_1^2+(1+|\eta|_{\infty,\Omega^t}^2) A_1^2{\Psi^2\over\nu^2}\cr
&\quad+\intop_0^t(\phi_2+\phi_2^4)dt'+D_2^8+\sup_t\phi_2\bigg( |\nabla\rot\psi_t|_{2,\Omega^t}^2+{\Psi^2\over\nu^2}\bigg)\cr
&\quad+{\Psi^2\over\nu^2}\bigg(|\nabla\rot\psi|_{2,\Omega^t}^2+ {\Psi^2\over\nu^2}\bigg)+|\nabla\varphi_t|_{3,\infty,\Omega^t}^2A_1^2+D_1^2 {\Psi^2\over\nu^2}\cr
&\quad+|v_t|_{3,\infty,\Omega^t}^2A_1^2+(1+|\eta|_{\infty,\Omega^t}^2) |f|_{2,\Omega^t}^2+|f_t|_{2,\Omega^t}^2\bigg]\cr
&\quad+c\nu^2(\sup_t\phi_2+\|\eta\|_{2,\infty,\Omega^t}^2){\Psi^2\over\nu^2}+a (\nu|\nabla\varphi(0)|_{1,1,}^2+|\rot\psi(0)|_{1,1,}^2),\cr}
\label{3.54}
\end{equation}
where $\phi_2$ is defined in (\ref{3.21}).

Simplifying, we write (\ref{3.54}) in the form
\begin{equation}\eqal{
X_2^2(t)&\equiv a(\nu|\nabla\varphi(t)|_{1,1}^2+|\rot\psi(t)|_{1,1}^2)+\mu (\nu|\nabla\varphi|_{2,1,2,\Omega^t}^2+|\rot\psi|_{2,1,2,\Omega^t}^2)\cr
&\quad+\nu^2|\nabla\varphi|_{2,1,2,\Omega^t}^2\le c\bigg[\phi_{10}\bigg(|\eta|_{2,1,\infty,\Omega^t},A_1,D_1,D_2,{\Psi\over\nu}\bigg)\cr
&\quad+\sup_t\phi_2|\nabla\rot\psi_t|_{2,\Omega^t}^2+{\Psi^2\over\nu^2} |\nabla\rot\psi|_{2,\Omega^t}^2+|\nabla\varphi_t|_{3,\infty,\Omega^t}^2A_1^2\cr
&\quad+|v_t|_{3,\infty,\Omega^t}^2A_1^2+(1+|\eta|_{\infty,\Omega^t}^2) (|f|_{2,\Omega^t}^2+|f_t|_{2,\Omega^t}^2)\bigg]\cr
&\quad+c\nu^2(\sup_t\phi_2+\|\eta\|_{2,\infty,\Omega^t}^2){\psi^2\over\nu^2}+a (\nu|\nabla\varphi(0)|_{1,1}^2+|\rot\psi(0)|_{1,1}^2).\cr}
\label{3.55}
\end{equation}
Using that
$$
|\nabla\varphi_t|_{3,\infty,\Omega^t}\le c\|\nabla\varphi_t\|_{1,\infty,\Omega^t}\le c{\chi_0\over\sqrt{\nu}}
$$
and
$$
|\nabla\rot\psi|_{2,\Omega^t}\le\varepsilon\|\rot\psi\|_{2,2,\Omega^t}+ c/\varepsilon|\rot\psi|_{2,\Omega^t}
$$
where
$$
|\rot\psi|_{2,\Omega^t}\le|v|_{2,\Omega^t}+|\nabla\varphi|_{2,\Omega^t}\le A_1+{\Psi\over\nu}
$$
we obtain from (\ref{3.55}) the inequality
\begin{equation}\eqal{
X_2^2(t)&\le c\bigg[\phi_{11}\bigg(|\eta|_{2,1,\infty,\Omega^t},A_1,D_1,D_2, {\Psi\over\nu},{\chi_0\over\sqrt{\nu}}\bigg)\cr
&\quad+\sup_t\phi_2|\nabla\rot\psi_t|_{2,\Omega^t}^2+|v_t|_{3,\infty,\Omega^t}^2+ A_1^2\cr
&\quad+(1+|\eta|_{\infty,\Omega^t}^2)(|f|_{2,\Omega^t}^2+|f_t|_{2,\Omega^t}^2) \bigg]\cr
&\quad+c\nu^2(\sup_t\phi_2+\|\eta\|_{2,\infty,\Omega^t}^2){\Psi^2\over\nu^2}\cr
&\quad+a (\nu|\nabla\varphi(0)|_{1,1}^2+|\rot\psi(0)|_{1,1}^2),\cr}
\label{3.56}
\end{equation}
where $\phi_2=|\eta|_{2,1}^2(1+|\eta|_{2,1}^2)$ (see (\ref{3.21})).

\noindent
Adding (\ref{3.52}) and (\ref{3.56}) and using the interpolations
$$\eqal{
&\alpha|\rot\psi_{,xt}|_2^2\le c\alpha|\rot\psi_{,xxt}|_2|\rot\psi_{,t}|_2\le \varepsilon|\rot\psi_{,xxt}|_2^2\cr
&\quad+c/\varepsilon\alpha^2|\rot\psi_{,t}|_2^2\le\varepsilon|\rot\psi_{,xxt}|_2^2+ cc/\varepsilon\alpha^2(D_1^2+|\nabla\varphi_{,t}|_2^2),\cr
&\alpha|v_{,t}|_3^2\le\varepsilon|v_{,xt}|_2^2+c/\varepsilon\alpha^2D_2^2\cr}
$$
we obtain
\begin{equation}\eqal{
&X_1^2(t)+X_2^2(t)\le\varepsilon|v_{,xxx}|_{2,\Omega^t}^2+\nu^2\varepsilon_1 |\nabla\varphi_{,xxx}|_{2,\Omega^t}^2\cr
&\quad+c/\varepsilon\bigg[\phi_{12}+|\nabla v|_{3,2,\Omega^t}^2{\chi_0^2\over\nu}+|\nabla v|_{3,\infty,\Omega^t}^2D_2^2\cr
&\quad+(1+|\eta|_{]\infty,\Omega^t}^2)(|f|_{2,\Omega^t}^2+|f_t|_{2,\Omega^t}^2)\cr
&\quad+{c\nu^2\over\varepsilon_1}\phi_7(|\eta|_{2,1,\infty,\Omega^t}) {\Psi^2\over\nu^2}+a\Phi_0^2(0),\cr}
\label{3.57}
\end{equation}
where
$$
\phi_{12}=\phi_{11}+|\phi_2|_{\infty,(0,t)}^2\bigg(D_1^2+{\Psi^2\over\nu^2}\bigg) +D_2^2A_1^2,
$$
and we used that
$$
|\nabla\varphi_{,t}|_{2,\Omega^t}^2\le{\Psi^2\over\nu^2}.
$$

\begin{lemma}\label{l3.5}
Let the assumptions of Lemma \ref{l3.1}. Let Notation \ref{n2.13} be applied. Then
\begin{equation}\eqal{
&{d\over dt}\bigg(\nu|\nabla\varphi_{,xx}|_2^2+|\rot\psi_{,xx}|_2^2\bigg)+\mu (\nu|\nabla^2\varphi_{,xx}|_2^2+|\nabla\rot\psi_{,xx}|_2^2)+\nu^2 |\Delta\varphi_{,xx}|_2^2\cr
&\le c[\|\eta\|_2^2\|v_{,t}\|_1^2+\|\eta\|_2^2(\|\eta\|_2^2+1)+\phi_{13} |\nabla\varphi_{,xxx}|_2^2\cr
&\quad+\phi_{14}|\nabla\varphi_{,xx}|_3^2+\phi_{15}\|\nabla\varphi_{,x}\|_1^2+ \phi_{16}|v|_2^2+\phi_{17}|v_{,x}|_2^2+(1+\|\eta\|_2^2)\|f\|_1^2],\cr}
\label{3.58}
\end{equation}
where $\phi_{13}-\phi_{17}$ are defined in (\ref{3.76}) and estimated in (\ref{3.78}) and

\begin{equation}\eqal{
&\nu|\nabla\varphi_{xx}(t)|_2^2+|\rot\psi_{,xx}(t)|_2^2+\mu(\nu |\nabla^2\varphi_{,xx}|_{2,2,\Omega^t}^2+|\rot\psi_{,xxx}|_{2,2,\Omega^t}^2)\cr
&\quad+\nu^2 |\Delta\varphi_{,xx}|_{2,2,\Omega^t}^2\le c\bigg[\exp\Big(\sqrt{t}\Phi_*(t)\Big)\bigg(t{\Psi^2\over\nu^2}+{c_0^2\over\nu^2}\bigg) \|v_{,t}\|_{1,2,\Omega^t}^2\cr
&\quad+\phi\bigg(A_1,A_2,\exp\Big(\sqrt{t}\Phi_*(t)\Big)\bigg(\sqrt{t}{\Psi\over\nu}+ {c_0\over\nu}\bigg),{\psi\over\nu},{\chi_0\over\sqrt{\nu}}\bigg)\cr
&\quad+\bigg(1+\exp\Big(\sqrt{t}\Phi_*(t)\Big)\bigg(t{\psi^2\over\nu^2}+{c_0^2\over\nu^2}\bigg)\bigg) \|f\|_{1,2,\Omega^t}^2\bigg]\cr
&\quad+c(\nu|\nabla\varphi_{,xx}(0)|_2^2+|\rot\psi_{,xx}(0)|_2^2).\cr}
\label{3.59}
\end{equation}
\end{lemma}

\begin{proof}
Differentiating (\ref{3.2}) twice with respect to $x$, multiplying by $\nabla\varphi_{,xx}$ and integrating over $\Omega$ yields
\begin{equation}\eqal{
&{a\over2}{d\over dt}|\nabla\varphi_{,xx}|_2^2+\mu|\nabla^2\varphi_{,xx}|_2^2+ \nu|\Delta\varphi_{,xx}|_2^2\cr
&=-\intop_\Omega(\eta v_{,t})_{,xx}\cdot\nabla \varphi_{xx}dx-a_0\intop_\Omega\nabla\eta_{xx}\cdot\nabla\varphi_{,xx}dx\cr
&\quad-\intop_\Omega [(a+\eta)v\cdot\nabla v]_{,xx}\cdot\nabla\varphi_{,xx}dx\cr
&\quad+\intop_\Omega[(p_\varrho(a)-p_\varrho(a+\eta))\nabla\eta]_{,xx}\cdot \nabla\varphi_{,xx}dx\cr
&\quad+\intop_\Omega[(a+\eta)f]_{,xx}\cdot\nabla\varphi_{xx}dx.\cr}
\label{3.60}
\end{equation}
The first term on the r.h.s. is estimated by
$$
\varepsilon|\nabla\varphi_{,xxx}|_2^2+c/\varepsilon\|\eta\|_2^2\|v_t\|_1^2,
$$
the second by
$$
\varepsilon|\nabla\varphi_{,xxx}|_2^2+c/\varepsilon|\nabla\eta_x|_2^2.
$$
Next we examine the third term on the r.h.s. of (\ref{3.60}). We write it in the form
$$
K\equiv a\intop_\Omega(v\cdot\nabla v)_{,xx}\cdot \nabla\varphi_{xx}dx+\intop_\Omega(\eta v\cdot\nabla v)_{,xx}\cdot \nabla\varphi_{xx}dx\equiv H+L.
$$
First we consider $H$. We drop $a$ for simplicity. Then we have
$$\eqal{
H&=\intop_\Omega[v\cdot\nabla(\nabla\varphi+\rot\psi)]_{,xx}\cdot \nabla\varphi_{xx}dx=\intop_\Omega(v\cdot\nabla\nabla\varphi)_{,xx}\cdot \nabla\varphi_{xx}dx\cr
&\quad+\intop_\Omega(v\cdot\nabla\rot\psi)_{,xx}\cdot\nabla\varphi_{xx}dx\equiv I_1+I_2.\cr}
$$
Consider $I_1$. We express it in the form
$$\eqal{
I_1&=\intop_\Omega v_{,xx}\cdot\nabla\nabla\varphi\cdot\nabla\varphi_{,xx}dx+ 2\intop_\Omega v_{,x}\cdot\nabla\nabla\varphi_{,x}\cdot\nabla\varphi_{xx}dx\cr
&\quad+ \intop_\Omega v\cdot\nabla(\nabla\varphi)_{,xx}\cdot\nabla\varphi_{,xx}dx\equiv I_{11}+I_{12}+I_{13},\cr}
$$
where
$$
I_{13}={1\over2}\intop_\Omega v\cdot\nabla|\nabla\varphi_{,xx}|^2dx=-{1\over2} \intop_\Omega\Delta\varphi\cdot|\nabla\varphi_{xx}|^2dx.
$$
Hence
$$
|I_{13}|\le\varepsilon|\nabla\varphi_{,xx}|_6^2+c/\varepsilon |\Delta\varphi|_3^2| \nabla\varphi_{,xx}|_2^2.
$$
Consider $I_{12}$. Integrating by parts yields
$$
I_{12}=-2\intop_\Omega v\cdot\nabla\nabla\varphi_{,xx}\cdot\nabla\varphi_{,xx} dx-2\intop_\Omega v\cdot\nabla\nabla\varphi_{,x}\cdot\nabla\varphi_{,xxx}dx \equiv I_{12}^1+I_{12}^2,
$$
where $I_{12}^1$ is estimated by the same bound as $I_{13}$ and
$$
|I_{12}^2|\le\varepsilon|\nabla\varphi_{,xxx}|_2^2+c/\varepsilon |\nabla\varphi_{,xx}|_3^2|v|_6^2\le\varepsilon|\nabla\varphi_{,xxx}|_2^2+ c/\varepsilon|\nabla\varphi_{,xx}|_3^2D_1^2.
$$
Finally, we examine
$$
I_{11}=-\intop_\Omega\Delta\varphi_{,xx}\nabla\varphi\cdot\nabla\varphi_{,xx}dx -\intop_\Omega v_{,xx}\cdot\nabla\nabla\varphi_{xx}\cdot\nabla\varphi dx\equiv I_{11}^1+I_{11}^2
$$
and
$$\eqal{
&|I_{11}^1|\le\varepsilon|\nabla\varphi_{,xxx}|_2^2+c/\varepsilon |\nabla\varphi_{,xx}|_3^2|\nabla\varphi|_6^2,\cr
&|I_{11}^2|\le\varepsilon|\nabla\varphi_{,xxx}|_2^2+c/\varepsilon|v_{,xx}|_3^2 |\nabla\varphi|_6^2.\cr}
$$
Summarizing, we have
\begin{equation}\eqal{
|I_1|&\le\varepsilon\|\nabla\varphi_{,xx}\|_1^2+c/\varepsilon [|\Delta\varphi|_3^2|\nabla\varphi_{,xx}|_2^2\cr
&\quad+|\nabla\varphi_{,xx}|_3^2D_1^2+|\nabla\varphi_{,xx}|_3^2| \nabla\varphi|_6^2+|v_{xx}|_3^2|\nabla\varphi|_6^2].\cr}
\label{3.61}
\end{equation}
Now we estimate $I_2$,
$$\eqal{
I_2&=\intop_\Omega v_{,xx}\cdot\nabla\rot\psi\cdot\nabla\varphi_{,xx}dx+2 \intop_\Omega v_{,x}\cdot\nabla\rot\psi_{,x}\cdot\nabla\varphi_{,xx}dx\cr
&\quad+\intop_\Omega v\cdot\nabla\rot\psi_{,xx}\cdot\nabla\varphi_{,xx}dx\equiv I_{21}+I_{22}+I_{23}.\cr}
$$
Consider $I_{21}$. Integrating by parts yields
$$
I_{21}=-\intop_\Omega\Delta\varphi_{,xx}\rot\psi\cdot\nabla\varphi_{,xx}dx- \intop_\Omega v_{,xx}\rot\psi\cdot\nabla\nabla\varphi_{,xx}dx\equiv I_{21}^1+I_{21}^2,
$$
where
$$\eqal{
|I_{21}^1|&\le\varepsilon|\Delta\varphi_{,xx}|_2^2+c/\varepsilon |\nabla\varphi_{,xx}|_3^2|\rot\psi|_6^2\cr
&\le\varepsilon|\Delta\varphi_{,xx}|_2^2+ c/\varepsilon|\nabla\varphi_{,xx}|_3^2\cdot(|v|_6^2+|\nabla\varphi|_6^2)\cr
&\le\varepsilon|\Delta\varphi_{,xx}|_2^2+c/\varepsilon|\nabla\varphi_{,xx}|_3^2 (D_1^2+|\nabla\varphi|_6^2)\cr}
$$
and
$$\eqal{
|I_{21}^2|&\le\varepsilon|\nabla\varphi_{xxx}|_2^2+c/\varepsilon|v_{xx}|_3^2 |\rot\psi|_6^2\le\varepsilon|\nabla\varphi_{xxx}|_2^2+c/\varepsilon |v_{,xx}|_3^2(|v|_6^2\cr
&\quad+|\nabla\varphi|_6^2)\le\varepsilon|\nabla\varphi_{,xxx}|_2^2+c/\varepsilon |v_{,xx}|_3^2(D_1^2+|\nabla\varphi|_6^2).\cr}
$$
Integrating by parts in the sum $I_{22}+I_{23}$ we get
$$
I_{22}+I_{23}=-\intop_\Omega v\cdot\nabla\rot\psi_{,xx}\cdot\nabla\varphi_{,xx}dx-2\intop_\Omega v\cdot\nabla\rot\psi_{,x}\cdot\nabla\varphi_{,xxx}dx\equiv J_1+J_2,
$$
where integrating by parts in $J_1$ yields
$$
J_1=\intop_\Omega\Delta\varphi\rot\psi_{,xx}\cdot\nabla\varphi_{,xx}dx+ \intop_\Omega v\rot\psi_{,xx}\cdot\nabla^2\varphi_{,xx}dx\equiv J_{11}+J_{12},
$$
where
$$\eqal{
&|J_{11}|\le\varepsilon|\nabla\varphi_{,xx}|_6^2+c/\varepsilon|\Delta\varphi|_3^2 |\rot\psi_{,xx}|_2^2,\cr
&|J_{12}|\le\varepsilon|\nabla^2\varphi_{,xx}|_2^2+c/\varepsilon |\rot\psi_{,xx}|_3^2|v|_6^2\le\varepsilon|\nabla^2\varphi_{,xx}|_2^2+ c/\varepsilon|\rot\psi_{,xx}|_3^2D_1^2.\cr}
$$
Finally,
$$
|J_2|\le\varepsilon|\nabla\varphi_{,xxx}|_2^2+c/\varepsilon|\rot\psi_{,xx}|_3^2 |v|_6^2\le\varepsilon|\nabla\varphi_{,xxx}|_2^2+c/\varepsilon |\rot\psi_{,xx}|_3^2D_1^2.
$$
Collecting estimates for $I_2$ yields
\begin{equation}\eqal{
|I_2|&\le\varepsilon\|\nabla\varphi_{,xx}\|_1^2+c/\varepsilon [|\nabla\varphi_{,xx}|_3^2(D_1^2+|\nabla\varphi|_6^2)\cr
&\quad+|v_{,xx}|_3^2(D_1^2+|\nabla\varphi|_6^2)+|\Delta\varphi|_3^2 |\rot\psi_{,xx}|_2^2+|\rot\psi_{,xx}|_3^2D_1^2].\cr}
\label{3.62}
\end{equation}
Next, we consider the expression
\begin{equation}\eqal{
L&=\intop_\Omega[\eta v\cdot\nabla v]_{,xx}\cdot\nabla\varphi_{,xx}dx= \intop_\Omega[\eta v\cdot\nabla\nabla\varphi]_{,xx}\cdot\nabla\varphi_{xx}dx\cr
&\quad+\intop_\Omega[\eta v\cdot\nabla\rot\psi]_{,xx}\cdot\nabla\varphi_{xx}dx \equiv L_1+L_2.\cr}
\label{3.63}
\end{equation}
Performing differentiation in $L_1$ we have
$$\eqal{\cr
L_1&=\intop_\Omega[\eta_{xx}v\cdot\nabla\nabla\varphi+\eta v_{xx}\cdot\nabla\nabla\varphi+\eta v\cdot\nabla\nabla\varphi_{xx}+2\eta_xv_x\cdot\nabla\nabla\varphi\cr
&\quad+2\eta_xv\cdot\nabla\nabla\varphi_x+2\eta v_x\nabla\nabla\varphi_x]\cdot\nabla\varphi_{,xx}dx\equiv\sum_{i=1}^6L_{1i}.\cr}
$$
Next, we have
$$\eqal{
|L_{11}|&\le\varepsilon|\nabla\varphi_{,xx}|_6^2+c/\varepsilon|\eta_{,xx}|_2^2 |v|_6^2|\nabla\varphi_{,x}|_6^2,\cr
|L_{12}|&\le\varepsilon|\nabla\varphi_{,xx}|_6^2+c/\varepsilon|\eta|_\infty^2 |v_{,xx}|_3^2|\nabla\varphi_{,x}|_2^2,\cr
L_{13}&={1\over2}\intop_\Omega\eta v\cdot\nabla|\nabla\varphi_{,xx}|^2dx= -{1\over2}\intop_\Omega v\cdot\nabla\eta|\nabla\varphi_{,xx}|^2dx\cr
&\quad-{1\over2}\intop_\Omega\eta\Delta\varphi|\nabla\varphi_{,xx}|^2dx\equiv L_{13}^1+L_{13}^2.\cr}
$$
Continuing, we have
$$\eqal{
&|L_{13}^1|\le\varepsilon|\nabla\varphi_{,xx}|_6^2+c/\varepsilon|v|_6^2 |\nabla\eta|_6^2|\nabla\varphi_{,xx}|_2^2,\cr
&|L_{13}^2|\le\varepsilon|\nabla\varphi_{,xx}|_6^2+c/\varepsilon|\eta|_\infty^2 |\Delta\varphi|_3^2|\nabla\varphi_{,xx}|_2^2,\cr
&|L_{14}|\le\varepsilon|\nabla\varphi_{,xx}|_6^2+c/\varepsilon|\eta_x|_6^2 |v_x|_6^2|\nabla\varphi_{,x}|_2^2,\cr
&|L_{15}|\le\varepsilon|\nabla\varphi_{,xx}|_6^2+c/\varepsilon|\eta_x|_6^2 |v|_6^2|\nabla\varphi_{,xx}|_2^2,\cr
&|L_{16}|\le\varepsilon|\nabla\varphi_{,xx}|_6^2+c/\varepsilon|\eta|_\infty^2 |v_{,x}|_3^2|\nabla\varphi_{,xx}|_2^2.\cr}
$$
Summarizing the above estimates we get
\begin{equation}
|L_1|\le\varepsilon|\nabla\varphi_{,xx}|_6^2+c/\varepsilon\|\eta\|_2^2 \|\nabla\varphi_x\|_1^2(|v|_6^2+|v_{xx}|_3^2+|\Delta\varphi|_3^2+|v_x|_6^2).
\label{3.64}
\end{equation}
Next we consider $L_2$ from (\ref{3.63}). Performing differentiations we obtain
$$\eqal{
L_2&=\intop_\Omega[\eta_{,xx}v\cdot\nabla\rot\psi+\eta v_{xx}\cdot\nabla\rot\psi+\eta v\cdot\nabla\rot\psi_{xx}\cr
&\quad+2\eta_xv_x\cdot\nabla\rot\psi+2\eta_xv\cdot\nabla\rot\psi_x+2\eta v_x\cdot\nabla\rot\psi_x]\cdot\nabla\varphi_{,xx}dx\cr
&\equiv\sum_{i=1}^6L_{2i}.\cr}
$$
Now we estimate the terms from $L_2$.
$$
|L_{21}|\le\varepsilon|\nabla\varphi_{,xx}|_6^2+c/\varepsilon|\eta_{,xx}|_2^2 |v|_6^2|\nabla\rot\psi|_6^2.
$$
Integrating by parts in $L_{22}$ implies
$$\eqal{
L_{22}&=-\intop_\Omega\nabla\eta\cdot v_{xx}\rot\psi\cdot\nabla\varphi_{xx}dx- \intop_\Omega\eta\Delta\varphi_{xx}\rot\psi\cdot\nabla\varphi_{,xx}dx\cr
&\quad-\intop_\Omega\eta v_{,xx}\rot\psi\cdot\nabla\nabla\varphi_{,xx}dx\equiv L_{22}^1+L_{22}^2+L_{22}^3,\cr}
$$
where
$$\eqal{
&|L_{22}^1|\le\varepsilon|\nabla\varphi_{,xx}|_6^2+c/\varepsilon|\nabla\eta|_6^2 |v_{,xx}|_2^2|\rot\psi|_6^2,\cr
&|L_{22}^2|\le\varepsilon|\Delta\varphi_{xx}|_2^2+c/\varepsilon|\eta|_\infty^2 |\nabla\varphi_{xx}|_3^2|\rot\psi|_6^2,\cr
&|L_{22}^3|\le\varepsilon|\nabla\varphi_{,xxx}|_2^2+c/\varepsilon|\eta|_\infty^2 |v_{,xx}|_3^2|\rot\psi|_6^2.\cr}
$$
Next we examine $L_{23}$. Integration by parts yields
$$\eqal{
L_{23}&=-\intop_\Omega\nabla\eta\cdot v\rot\psi_{,xx}\cdot\nabla\varphi_{,xx}dx- \intop_\Omega\eta\Delta\varphi\rot\psi_{,xx}\cdot\nabla\varphi_{,xx}dx\cr
&\quad-\intop_\Omega\eta v\cdot\rot\psi_{,xx}\nabla^2\varphi_{,xx}dx=L_{23}^1+L_{23}^2+L_{23}^3.\cr}
$$
Continuing,
$$\eqal{
&|L_{23}^1|\le\varepsilon|\nabla\varphi_{,xx}|_6^2+c/\varepsilon|\nabla\eta|_6^2 |v|_6^2|\rot\psi_{,xx}|_2^2,\cr
&|L_{23}^2|\le\varepsilon|\nabla\varphi_{,xx}|_6^2+c/\varepsilon|\eta|_\infty^2 |\Delta\varphi|_3^2|\rot\psi_{,xx}|_2^2,\cr
&|L_{23}^3|\le\varepsilon|\nabla\varphi_{,xxx}|_2^2+c/\varepsilon|\eta|_\infty^2 |v|_6^2|\rot\psi_{,xx}|_3^2.\cr}
$$
Next we examine $L_{24}$. Integration by parts implies
$$\eqal{
L_{24}&=-2\intop_\Omega\nabla\eta_xv_x\rot\psi\cdot\nabla\varphi_{,xx}dx-2 \intop_\Omega\eta_x\Delta\varphi_{,x}\rot\psi\cdot\nabla\varphi_{,xx}dx\cr
&\quad-2\intop_\Omega\eta_xv_x\rot\psi\cdot\nabla\varphi_{,xxx}dx\equiv L_{24}^1+L_{24}^2+L_{24}^3.\cr}
$$
Estimating, we get
$$\eqal{
&|L_{24}^1|\le\varepsilon|\nabla\varphi_{,xx}|_6^2+c/\varepsilon|\eta_{,xx}|_2^2 |v_{,x}|_6^2|\rot\psi|_6^2,\cr
&|L_{24}^2|\le\varepsilon|\nabla\varphi_{,xx}|_6^2+c/\varepsilon|\eta_{,x}|_6^2 |\Delta\varphi_{,x}|_2^2|\rot\psi|_6^2,\cr
&|L_{24}^3|\le\varepsilon|\nabla\varphi_{,xxx}|_2^2+c/\varepsilon|\eta_{,x}|_6^2 |v_{,x}|_6^2|\rot\psi|_6^2.\cr}
$$
Next, we have
$$
|L_{25}|\le\varepsilon|\nabla\varphi_{,xx}|_6^2+c/\varepsilon|\eta_{,x}|_6^2 |v|_6^2|\rot\psi_{,xx}|_2^2.
$$
Finally, we integrate by parts in $L_{26}$. Then we have
$$\eqal{
L_{26}&=-2\intop_\Omega\eta_{,x}v\cdot\nabla\rot\psi_{,x}\cdot\nabla\varphi_{,xx} dx-2\intop_\Omega\eta v\cdot\nabla\rot\psi_{,xx}\cdot\nabla\varphi_{,xx}dx\cr
&\quad-2\intop_\Omega\eta v\cdot\nabla\rot\psi_{,x}\cdot\nabla\varphi_{,xxx}dx\equiv L_{26}^1+L_{26}^2+L_{26}^3.\cr}
$$
Consider $L_{26}^2$. Integration by parts yields
$$\eqal{
L_{26}^2&=2\intop_\Omega\nabla\eta\cdot v\rot\psi_{,xx}\cdot\nabla\varphi_{,xx}dx+2 \intop_\Omega\eta\Delta\varphi\rot\psi_{,xx}\cdot\nabla\varphi_{,xx}dx\cr
&\quad+2\intop_\Omega\eta v\cdot\nabla^2\varphi_{xx}\cdot\rot\psi_{xx}dx\equiv F_1+F_2+F_3.\cr}
$$
Continuing, we have
$$\eqal{
&|F_1|\le\varepsilon|\nabla\varphi_{,xx}|_6^2+c/\varepsilon|\nabla\eta|_6^2 |v|_6^2|\rot\psi_{,xx}|_2^2,\cr
&|F_2|\le\varepsilon|\nabla\varphi_{,xx}|_6^2+c/\varepsilon|\eta|_\infty^2 |\Delta\varphi|_3^2|\rot\psi_{,xx}|_2^2,\cr
&|F_3|\le\varepsilon|\nabla^2\varphi_{,xx}|_2^2+c/\varepsilon|\eta|_\infty^2 |v|_6^2|\rot\psi_{,xx}|_3^2.\cr}
$$
Finally, we have
$$\eqal{
&|L_{26}^1|\le\varepsilon|\nabla\varphi_{,xx}|_6^2+c/\varepsilon|\eta_{,x}|_6^2 |v|_6^2|\rot\psi_{,xx}|_2^2,\cr
&|L_{26}^2|\le\varepsilon\|\nabla\varphi_{,xx}\|_1^2+c/\varepsilon\|\eta\|_2^2 (|v|_6^2+|\Delta\varphi|_3^2)|\rot\psi_{,xx}|_3^2,\cr
&|L_{26}^3|\le\varepsilon|\nabla\varphi_{,xxx}|_2^2+c/\varepsilon|\eta|_\infty^2 |v|_6^2|\rot\psi_{,xx}|_3^2.\cr}
$$
Collecting all estimates for $L_2$ we have
\begin{equation}\eqal{
|L_2|&\le\varepsilon\|\nabla\varphi_{,xx}\|_1^2+{c\over\varepsilon}\|\eta\|_2^2 [|v|_6^2|\rot\psi_{,x}|_6^2+(|v_{,xx}|_3^2+|\nabla\varphi_{xx}|_3^2) |\rot\psi|_6^2\cr
&\quad+(|v|_6^2+|\Delta\varphi|_3^2)|\rot\psi_{,xx}|_3^2+(|v_x|_6^2+ |\Delta\varphi_x|_2^2)|\rot\psi|_6^2].\cr}
\label{3.65}
\end{equation}
Now we collect all estimates for the third term on the r.h.s. of (\ref{3.60}). Using (\ref{3.61}), (\ref{3.62}), (\ref{3.63}) and (\ref{3.64}) we have
$$\eqal{
|H|&\le\varepsilon\|\nabla\varphi_{,xx}\|_1^2+c/\varepsilon[ |\nabla\varphi_{,xx}|_3^2(|\Delta\varphi|_3^2+|\nabla\varphi|_6^2+D_1^2)\cr
&\quad+|v_{,xx}|_3^2(D_1^2+|\nabla\varphi|_6^2)+|\rot\psi_{,xx}|_3^2(D_1^2+ |\Delta\varphi|_3^2)]\cr}
$$
and
$$\eqal{
|L|&\le
\varepsilon\|\nabla\varphi_{,xx}\|_1^2+c/\varepsilon\|\eta\|_2^2[ \|\nabla\varphi_{,x}\|_1^2(|v|_6^2+|v_{,x}|_6^2+|v_{,xx}|_3^2+ |\Delta\varphi|_3^2)\cr
&\quad+|\rot\psi_{xx}|_3^2(|v|_6^2+|\Delta\varphi|_3^2)+|\rot\psi|_6^2 (|v_{,x}|_6^2+|v_{,xx}|_3^2+|\nabla\varphi_{,xx}|_3^2)].\cr}
$$
Using that $v=\nabla\varphi+\rot\psi$ we derive
$$\eqal{
|K|&\le|H|+|L|\le\varepsilon\|\nabla\varphi_{,xx}\|_1^2+c/\varepsilon [|\nabla\varphi_{,xx}|_3^2(|\Delta\varphi|_3^2+|\nabla\varphi|_6^2+D_1^2)\cr
&\quad+(1+\|\eta\|_2^2)|\rot\psi_{,xx}|_3^2(D_1^2+|\Delta\varphi|_3^2)\cr
&\quad+\|\eta\|_2^2\|\nabla\varphi_{,x}\|_1^2(D_1^2+|v_{,x}|_6^2+|v_{,xx}|_3^2+ |\Delta\varphi|_3^2)\cr
&\quad+\|\eta\|_2^2|\rot\psi|_6^2(|v_{,x}|_6^2+|v_{,xx}|_3^2+ |\nabla\varphi_{,xx}|_3^2)].\cr}
$$
The fourth term on the r.h.s. of (\ref{3.60}) is estimated by
$$
\varepsilon|\nabla\varphi_{,xxx}|_2^2+c/\varepsilon\|\eta\|_2^4.
$$
Finally, the last term on the r.h.s. of (\ref{3.60}) is bounded by
$$
\varepsilon|\nabla\varphi_{,xxx}|_2^2+c/\varepsilon[|f_{g,x}|_2^2+ |\eta_{,x}|_6^2|f|_3^2+|\eta|_\infty^2|f_x|_2^2].
$$
Using the above estimates in (\ref{3.60}) and assuming that $\varepsilon\le\nu/2$ we obtain the inequality
\begin{equation}\eqal{
&a{d\over dt}|\nabla\varphi_{,xx}|_2^2+\mu|\nabla^2\varphi_{,xx}|_2^2+\nu |\Delta\varphi_{,xx}|_2^2\le{c\over\nu}[\|\eta\|_2^2\|v_t\|_1^2\cr
&\quad+\|\eta\|_2^2+\|\eta\|_2^4+|\nabla\varphi_{,xx}|_3^2(|\Delta\varphi|_3^2+ |\nabla\varphi|_6^2+D_1^2)\cr
&\quad+(1+\|\eta\|_2^2)|\rot\psi_{,xx}|_3^2(D_1^2+|\Delta\varphi|_3^2)+ \|\eta\|_2^2\|\nabla\varphi_{,x}\|_1^2(D_1^2+|v_{,x}|_6^2\cr
&\quad+|v_{,xx}|_3^2+|\Delta\varphi|_3^2)+\|\eta\|_2^2(D_1^2+|\nabla\varphi|_6^2) (|v_{,x}|_6^2+|v_{,xx}|_3^2+|\nabla\varphi_{,xx}|_3^2)\cr
&\quad+(1+\|\eta\|_2^2)\|f\|_1^2].\cr}
\label{3.66}
\end{equation}
Simplifying, we write (\ref{3.66}) in the form
\begin{equation}\eqal{
&a{d\over dt}|\nabla\varphi_{,xx}|_2^2+\mu|\nabla\varphi_{,xxx}|_2^2+\nu |\Delta\varphi_{,xx}|_2^2\le{c\over\nu}[\|\eta\|_2^2\|v_{,t}\|_1^2\cr
&\quad+\|\eta\|_2^2(\|\eta\|_2^2+1)+|\nabla\varphi_{,xx}|_3^2(1+\|\eta\|_2^2) (D_1^2+\|\nabla\varphi\|_2^2)\cr
&\quad+|\rot\psi_{,xx}|_3^2(1+\|\eta\|_2^2)(D_1^2+\|\nabla\varphi\|_2^2)\cr
&\quad+\|\eta\|_2^2\|\nabla\varphi_{,x}\|_1^2(D_1^2+|v_{,x}|_6^2+ |\Delta\varphi|_3^2)+\|\eta\|_2^2(D_1^2+|\nabla\varphi|_6^2)|v_{,x}|_6^2\cr
&\quad+(1+\|\eta\|_2^2)\|f\|_1^2].\cr}
\label{3.67}
\end{equation}
Differentiating (\ref{3.2}) twice with respect to $x$, multiplying by $\rot\psi_{,xx}$ and integrating over $\Omega$ implies
\begin{equation}\eqal{
&{a\over2}{d\over dt}|\rot\psi_{,xx}|_2^2+\mu|\nabla\rot\psi_{,xx}|_2^2= -\intop_\Omega[\eta v_{,t}]_{,xx}\cdot\rot\psi_{,xx}dx\cr
&\quad-\intop_\Omega[(a+\eta)v\cdot\nabla v]_{,xx}\cdot\rot\psi_{,xx}dx+ \intop_\Omega[(a+\eta)f]_{,xx}\cdot\rot\psi_{,xx}dx.\cr}
\label{3.68}
\end{equation}
The first term on the r.h.s. of (\ref{3.68}) is bounded by
\begin{equation}
\varepsilon|\rot\psi_{,xxx}|_2^2+c/\varepsilon\|\eta\|_2^2\|v_{,t}\|_1^2.
\label{3.69}
\end{equation}
The second term on the r.h.s. of (\ref{3.68}) is expressed in the form
$$
a\intop_\Omega(v\cdot\nabla v)_{,xx}\cdot\rot\psi_{,xx}dx+\intop_\Omega(\eta v\cdot\nabla v)_{,xx}\cdot\rot\psi_{,xx}dx\equiv J+K.
$$
First we consider $J$. We drop $a$ for simplicity. Then we have
$$
J=\intop_\Omega(v\cdot\nabla\nabla\varphi)_{,xx}\cdot\rot\psi_{,xx}dx+ \intop_\Omega(v\cdot\nabla\rot\psi)_{,xx}\cdot\rot\psi_{xx}dx\equiv J_1+J_2.
$$
Consider $J_1$,
$$
J_1=\intop_\Omega[v_{,xx}\cdot\nabla\nabla\varphi+v\cdot\nabla\nabla\varphi_{xx} +2v_{,x}\nabla\nabla\varphi_{,x}]\cdot\rot\psi_{,xx}dx\equiv J_{11}+J_{12}+J_{13}.
$$
Integrating by parts in $J_{11}$ yields
$$
J_{11}=-\intop_\Omega\Delta\varphi_{,xx}\nabla\varphi\cdot\rot\psi_{,xx}dx- \intop_\Omega v_{,xx}\nabla\varphi\cdot\nabla\rot\psi_{,xx}dx\equiv J_{11}^1+J_{11}^2,
$$
where
$$\eqal{
&|J_{11}^1|\le\varepsilon|\rot\psi_{,xx}|_6^2+c/\varepsilon |\nabla\varphi_{,xxx}|_2^2|\nabla\varphi|_3^2,\cr
&|J_{11}^2|\le\varepsilon|\rot\psi_{,xxx}|_2^2+c/\varepsilon|v_{,xx}|_3^2 |\nabla\varphi|_6^2.\cr}
$$
Integrating by parts in $J_{13}$ yields
$$
J_{12}+J_{13}=-\intop_\Omega v\cdot\nabla\nabla\varphi_{,xx}\cdot\rot\psi_{,xx}dx-2\intop_\Omega v\cdot\nabla\nabla\varphi_{,x}\cdot\rot\psi_{,xxx}dx\equiv K_1+K_2,
$$
where
$$\eqal{
&|K_1|\le\varepsilon|\rot\psi_{,xx}|_6^2+c/\varepsilon|\nabla\varphi_{,xxx}|_2^2 |v|_3^2,\cr
&|K_2|\le\varepsilon|\rot\psi_{,xxx}|_2^2+c/\varepsilon |\nabla\varphi_{,xx}|_3^2|v|_6^2.\cr}
$$
Summarizing,
$$\eqal{
|J_1|&\le\varepsilon\|\rot\psi_{,xx}\|_1^2+c/\varepsilon|\nabla\varphi_{,xxx}|_2^2 (|\nabla\varphi|_3^2+|v|_3^2)+c/\varepsilon(|v_{,xx}|_3^2|\nabla\varphi|_6^2\cr
&\quad+|\nabla\varphi_{,xx}|_3^2|v|_6^2).\cr}
$$
Consider $J_2$. We express it in the form
$$
J_2=\intop_\Omega[v_{,xx}\cdot\nabla\rot\psi+2v_{,x}\cdot\nabla\rot\psi_{,x}+ v\cdot\nabla\rot\psi_{,xx}]\cdot\rot\psi_{,xx}dx\equiv J_{21}+J_{22}+J_{23}.
$$
Integration by parts in $J_{21}$ gives
$$
J_{21}=-\intop_\Omega\Delta\varphi_{,xx}\rot\psi\cdot\rot\psi_{,xx}dx- \intop_\Omega v_{,xx}\rot\psi\cdot\nabla\rot\psi_{,xx}dx\equiv J_{21}^1+J_{21}^2,
$$
where
$$\eqal{
&|J_{21}^1|\le\varepsilon|\rot\psi_{,xx}|_6^2+c/\varepsilon |\Delta\varphi_{,xx}|_2^2|\rot\psi|_3^2,\cr
&|J_{21}^2|\le\varepsilon|\rot\psi_{,xxx}|_2^2+c/\varepsilon|v_{,xx}|_3^2 |\rot\psi|_6^2.\cr}
$$
Integrating by parts in $J_{22}$ implies
$$\eqal{
J_{22}+J_{23}&=-\intop_\Omega v\cdot\nabla\rot\psi_{,xx}\cdot\rot\psi_{,xx}dx-2\intop_\Omega v\cdot\nabla\rot\psi_{,x}\cdot\rot\psi_{,xxx}dx\cr
&\equiv M_1+M_2,\cr}
$$
where
$$
M_1=-{1\over2}\intop_\Omega v\cdot\nabla|\rot\psi_{,xx}|^2dx={1\over2}\intop_\Omega\Delta\varphi |\rot\psi_{,xx}|^2dx
$$
and
$$
|M_1|\le\varepsilon|\rot\psi_{,xx}|_6^2+c/\varepsilon|\Delta\varphi|_2^2 |\rot\psi_{,xx}|_3^2.
$$
Finally,
$$
|M_2|\le\varepsilon|\rot\psi_{,xxx}|_2^2+c/\varepsilon|\rot\psi_{,xx}|_3^2 |v|_6^2.
$$
Summarizing,
$$\eqal{
|J_2|&\le\varepsilon\|\rot\psi_{,xx}\|_1^2+c/\varepsilon [|\Delta\varphi_{,xx}|_2^2|\rot\psi|_3^2+|v_{,xx}|_3^2|\rot\psi|_6^2\cr
&\quad+|\rot\psi_{,xx}|_3^2(|\Delta\varphi|_2^2+|v|_6^2)].\cr}
$$
Collecting all estimates for $J$ we get
\begin{equation}\eqal{
|J|&\le\varepsilon\|\rot\psi_{,xx}\|_1^2+c/\varepsilon[|\nabla\varphi_{,xxx}|_2^2 (|\nabla\varphi|_3^2+|v|_3^2+|\rot\psi|_3^2)\cr
&\quad+(|\nabla\varphi_{,xx}|_3^2+|\rot\psi_{,xx}|_3^2)(|v|_6^2+ |\nabla\varphi|_6^2+|\Delta\varphi|_2^2)]\cr}
\label{3.70}
\end{equation}
Finally, we consider $K$. We write it in the form
$$\eqal{
K&=\intop_\Omega[\eta v\cdot\nabla\nabla\varphi]_{,xx}\cdot\rot\psi_{,xx}dx+ \intop_\Omega[\eta v\cdot\nabla\rot\psi]_{,xx}\cdot\rot\psi_{,xx}dx\cr
&\equiv K_1+K_2.\cr}
$$
Performing differentiations we have
$$\eqal{
K_1&=\intop_\Omega[\eta_{xx}v\cdot\nabla\nabla\varphi+\eta v_{,xx}\cdot\nabla\nabla\varphi+\eta v\cdot\nabla\nabla\varphi_{,xx}+2\eta_{,x}v_{,x}\cdot\nabla\nabla\varphi\cr
&\quad+2\eta_{,x}v\cdot\nabla\nabla\varphi_{,x}+2\eta v_{,x}\nabla\nabla\varphi_{,x}]\cdot\rot\psi_{,xx}dx\equiv\sum_{i=1}^6K_{1i}.\cr}
$$
Now we estimate the terms $K_{1i}$, $i=1,\dots,6$. We have
$$\eqal{
&|K_{11}|\le\varepsilon|\rot\psi_{,xx}|_6^2+c/\varepsilon|\eta_{,xx}|_2^2 |v|_6^2|\nabla\varphi_{,x}|_6^2,\cr
&|K_{12}|\le\varepsilon|\rot\psi_{,xx}|_6^2+c/\varepsilon|\eta|_\infty^2 |v_{,xx}|_2^2|\nabla\varphi_{,x}|_3^2,\cr
&|K_{13}|\le\varepsilon|\rot\psi_{,xx}|_6^2+c/\varepsilon|\eta|_\infty^2|v|_3^2 |\nabla\varphi_{,xxx}|_2^2.\cr}
$$
Integrating by parts in $K_{14}$ yields
$$\eqal{
K_{14}&=-2\intop_\Omega\eta_{xx}v\cdot\nabla\nabla\varphi\cdot\rot\psi_{,xx}dx- 2\intop_\Omega\eta_{,x}v\cdot\nabla\nabla\varphi_x\cdot\rot\psi_{,xx}dx\cr
&\quad-2\intop_\Omega\eta_{,x}v\cdot\nabla\nabla\varphi\cdot\rot\psi_{,xxx}dx \equiv K_{14}^1+K_{14}^2+K_{14}^3,\cr}
$$
where
$$\eqal{
&|K_{14}^1|\le\varepsilon|\rot\psi_{,xx}|_6^2+c/\varepsilon|\eta_{,xx}|_2^2 |v|_6^2|\nabla\varphi_{,x}|_6^2,\cr
&|K_{14}^2|\le\varepsilon|\rot\psi_{,xx}|_6^2+c/\varepsilon|\eta_{,x}|_6^2 |v|_6^2|\nabla\varphi_{,xx}|_2^2,\cr
&|K_{14}^3|\le\varepsilon|\rot\psi_{,xxx}|_2^2+c/\varepsilon|\eta_{,x}|_6^2 |v|_6^2|\nabla\varphi_{,x}|_6^2.\cr}
$$
Next, we have
$$
|K_{15}|\le\varepsilon|\rot\psi_{,xx}|_6^2+c/\varepsilon|\eta_{,x}|_6^2|v|_6^2 |\nabla\varphi_{,xx}|_2^2.
$$
Integrating by parts in $K_{16}$ gives
$$\eqal{
K_{16}&=-2\intop_\Omega\eta_{,x}v\cdot\nabla\nabla\varphi_{,x}\cdot \rot\psi_{,xx}dx\cr
&\quad-2\intop_\Omega\eta v\cdot\nabla\nabla\varphi_{,xx}\cdot\rot\psi_{,xx}dx\cr
&\quad-2\intop_\Omega\eta v\cdot\nabla\nabla\varphi_{,x}\cdot\rot\psi_{,xxx}dx\equiv K_{16}^1+K_{16}^2+K_{16}^3,\cr}
$$
where
$$\eqal{
&|K_{16}^1|\le\varepsilon|\rot\psi_{,xx}|_6^2+c/\varepsilon|\eta_{,x}|_6^2 |v|_6^2|\nabla\varphi_{,xx}|_2^2,\cr
&|K_{16}^2|\le\varepsilon|\rot\psi_{,xx}|_6^2+c/\varepsilon|\eta|_\infty^2 |v|_3^2|\nabla\varphi_{,xxx}|_2^2,\cr
&|K_{16}^3|\le\varepsilon|\rot\psi_{,xxx}|_2^2+c/\varepsilon|\eta|_\infty^2 |v|_6^2|\nabla\varphi_{,xx}|_3^2.\cr}
$$
Summarizing, the above estimates yields
\begin{equation}\eqal{
|K_1|&\le\varepsilon\|\rot\psi_{,xx}\|_1^2+c/\varepsilon[\|\eta\|_2^2|v|_6^2 \|\nabla\varphi_{,x}\|_1^2+|\eta|_\infty^2|v_{,xx}|_2^2|\nabla\varphi_{,x}|_3^2\cr
&\quad+|\eta|_\infty^2|v|_6^2\|\nabla\varphi_{,xx}\|_1^2].\cr}
\label{3.71}
\end{equation}
Finally, we estimate $K_2$. Performing differentiations we have
$$\eqal{
K_2&=\intop_\Omega[\eta_{,xx}v\cdot\nabla\rot\psi+\eta v_{,xx}\cdot\nabla\rot\psi+\eta v\cdot\nabla\rot\psi_{,xx}+2\eta_{,x}v_{,x}\cdot\nabla\rot\psi\cr
&\quad+2\eta_{,x}v\cdot\nabla\rot\psi_{,x}+2\eta v_{,x}\cdot\nabla\rot\psi_{,x}]\cdot\rot\psi_{,xx}dx\equiv\sum_{i=1}^6K_{2i}.\cr}
$$
Now, we estimate $K_{2i}$, $i=1,\cdots,6$.
$$
|K_{21}|\le\varepsilon|\rot\psi_{,xx}|_6^2+c/\varepsilon|\eta_{,xx}|_2^2|v|_6^2 |\nabla\rot\psi|_6^2.
$$
To estimate $K_{22}$ we integrate by parts. Then we have
$$\eqal{
K_{22}&=-\intop_\Omega\nabla\eta\cdot v_{,xx}\rot\psi\cdot\rot\psi_{,xx}dx- \intop_\Omega\eta\Delta\varphi_{,xx}\rot\psi\cdot\rot\psi_{,xx}dx\cr
&\quad-\intop_\Omega\eta v_{,xx}\cdot\rot\psi\cdot\nabla\rot\psi_{,xx}dx\equiv K_{22}^1+K_{22}^2+K_{22}^3,\cr}
$$
where
$$\eqal{
&|K_{22}^1|\le\varepsilon|\rot\psi_{,xx}|_6^2+c/\varepsilon|\nabla\eta|_6^2 |v_{,xx}|_2^2|\rot\psi|_6^2,\cr
&|K_{22}^2|\le\varepsilon|\rot\psi_{,xx}|_6^2+c/\varepsilon|\eta|_\infty^2 |\nabla\varphi_{,xxx}|_2^2|\rot\psi|_3^2,\cr
&|K_{22}^3|\le\varepsilon|\rot\psi_{,xxx}|_2^2+c/\varepsilon|\eta|_\infty^2 |v_{,xx}|_3^2|\rot\psi|_6^2.\cr}
$$
Consider $K_{23}$. Integration by parts yields
$$\eqal{
K_{23}&={1\over2}\intop_\Omega\eta v\cdot\nabla|\rot\psi_{,xx}|^2dx=-{1\over2} \intop_\Omega v\cdot\nabla\eta|\rot\psi_{,xx}|^2dx\cr
&\quad-{1\over2}\intop_\Omega\eta\Delta\varphi|\rot\psi_{,xx}|^2dx\equiv K_{23}^1+K_{23}^2,\cr}
$$
where
$$\eqal{
&|K_{23}^1|\le\varepsilon|\rot\psi_{,xx}|_6^2+c/\varepsilon|\nabla\eta|_6^2 |v|_6^2|\rot\psi_{,xx}|_2^2,\cr
&|K_{23}^2|\le\varepsilon|\rot\psi_{,xx}|_6^2+c/\varepsilon|\eta|_\infty^2 |\Delta\varphi|_3^2|\rot\psi_{,xx}|_2^2.\cr}
$$
Integrating by parts in $K_{24}$ gives
$$\eqal{
K_{24}&=-2\intop_\Omega\eta_{,xx}v\cdot\nabla\rot\psi\cdot\rot\psi_{,xx}dx- 2\intop_\Omega\eta_{,x}v\cdot\nabla\rot\psi_{,x}\cdot\rot\psi_{,xx}dx\cr
&\quad-2\intop_\Omega\eta_{,x}v\cdot\nabla\rot\psi\cdot\rot\psi_{,xxx}dx\equiv K_{24}^1+K_{24}^2+K_{24}^3,\cr}
$$
where
$$\eqal{
&|K_{24}^1|\le\varepsilon|\rot\psi_{,xx}|_6^2+c/\varepsilon|\eta_{,xx}|_2^2 |v|_6^2|\rot\psi_{,x}|_6^2,\cr
&|K_{24}^2|\le\varepsilon|\rot\psi_{,xx}|_6^2+c/\varepsilon|\eta_{,x}|_6^2 |v|_6^2|\rot\psi_{,xx}|_2^2,\cr
&|K_{24}^3|\le\varepsilon|\rot\psi_{,xxx}|_2^2+c/\varepsilon|\eta_{,x}|_6^2 |v|_6^2|\rot\psi_{,x}|_6^2.\cr}
$$
Next, we have
$$
|K_{25}|\le\varepsilon|\rot\psi_{,xx}|_6^2+c/\varepsilon|\eta_{,x}|_6^2|v|_6^2 |\rot\psi_{,xx}|_2^2.
$$
Integrating by parts in $K_{26}$ implies
$$\eqal{
K_{26}&=-2\intop_\Omega\eta_{,x}v\cdot\nabla\rot\psi_{,x}\cdot\rot\psi_{,xx}dx- 2\intop_\Omega\eta v\cdot\nabla\rot\psi_{,xx}\cdot\rot\psi_{xx}dx\cr
&\quad-2\intop_\Omega\eta v\cdot\nabla\rot\psi_{,x}\cdot\rot\psi_{,xxx}dx\equiv K_{26}^1+K_{26}^2+K_{26}^3,\cr}
$$
where
$$
|K_{26}^1|\le\varepsilon|\rot\psi_{,xx}|_6^2+c/\varepsilon|\eta_{,x}|_6^2|v|_6^2 |\rot\psi_{,xx}|_2^2,
$$
$$\eqal{
K_{26}^2&=-\intop_\Omega\eta v\cdot\nabla|\rot\psi_{,xx}|^2dx\cr
&=\intop_\Omega \nabla\eta\cdot v|\rot\psi_{,xx}|^2dx+\intop_\Omega\eta\Delta\varphi|\rot\psi_{,xx}|^2dx\cr
&\equiv D_1+D_2.\cr}
$$
Continuing, we have
$$\eqal{
&|D_1|\le\varepsilon|\rot\psi_{,xx}|_6^2+c/\varepsilon|\nabla\eta|_6^2|v|_6^2 |\rot\psi_{,xx}|_2^2,\cr
&|D_2|\le\varepsilon|\rot\psi_{,xx}|_6^2+c/\varepsilon|\eta|_\infty^2 |\Delta\varphi|_3^2|\rot\psi_{,xx}|_2^2.\cr}
$$
Finally, we estimate
$$
|K_{26}^3|\le\varepsilon|\rot\psi_{,xxx}|_2^2+c/\varepsilon|\eta|_\infty^2 |v|_6^2|\rot\psi_{,xx}|_3^2.
$$
Summarizing the estimates we get
\begin{equation}\eqal{
|K_2|&\le\varepsilon\|\rot\psi_{,xx}\|_1^2+c/\varepsilon\|\eta\|_2^2 [(|v|_6^2+|\Delta\varphi|_3^2)\|\rot\psi_{,x}\|_1^2\cr
&\quad+(D_1^2+|\nabla\varphi|_6^2)(|\rot\psi_{,xx}|_3^2+|\nabla\varphi_{,xx}|_3^2 +|\nabla\varphi_{,xxx}|_2^2)].\cr}
\label{3.72}
\end{equation}
The last term on the r.h.s. of (\ref{3.68}) is bounded by
\begin{equation}
\varepsilon|\rot\psi_{,xxx}|_2^2+c/\varepsilon(\|f_r\|_1^2+\|\eta\|_2^2\|f\|_1^2).
\label{3.73}
\end{equation}
Using estimates (\ref{3.69})--(\ref{3.73}) in (\ref{3.68}) and assuming that $\varepsilon$ is sufficiently small we get
\begin{equation}\eqal{
&a{d\over dt}|\rot\psi_{,xx}|_2^2+\mu|\nabla\rot\psi_{,xx}|_2^2\le c[\|\eta\|_2^2\|v_{,t}\|_1^2\cr
&\quad+|\nabla\varphi_{,xxx}|_2^2(1+\|\eta\|_2^2)(|\nabla\varphi|_6^2+|v|_6^2)\cr
&\quad+(|\nabla\varphi_{,xx}|_3^2+|\rot\psi_{,xx}|_3^2)(1+\|\eta\|_2^2)(|v|_6^2+ \|\nabla\varphi\|_1^2)\cr
&\quad+\|\nabla\varphi_{,x}\|_1^2\|\eta\|_2^2(|v|_6^2+|v_{,xx}|_2^2)+ \|\rot\psi_{,x}\|_1^2\|\eta\|_2^2(|v|_6^2+\|\nabla\varphi\|_2^2)\cr
&\quad+(1+\|\eta\|_2^2)\|f\|_1^2].\cr}
\label{3.74}
\end{equation}
Multiplying (\ref{3.74}) by $1/\nu$ and adding to (\ref{3.67}) yields
\begin{equation}\eqal{
&a{d\over dt}\bigg(|\nabla\varphi_{,xx}|_2^2+{1\over\nu}|\rot\psi_{,xx}|_2^2\bigg)+\mu \bigg(|\nabla\varphi_{,xxx}|_2^2+{1\over\nu}|\rot\psi_{,xxx}|_2^2\bigg)\cr
&\quad+\nu|\Delta\varphi_{,xx}|_2^2\le{c\over\nu}[\|\eta\|_2^2\|v_{,t}\|_1^2+ \|\eta\|_2^2(\|\eta\|_2^2+1)\cr
&\quad+|\nabla\varphi_{,xxx}|_2^2(1+\|\eta\|_2^2)(|\nabla\varphi|_6^2+|v|_6^2)\cr
&\quad+(1+\|\eta\|_2^2)(|\nabla\varphi_{,xx}|_3^2+|\rot\psi_{,xx}|_3^2) (|v|_6^2+\|\nabla\varphi\|_2^2)\cr
&\quad+\|\eta\|_2^2\|\nabla\varphi_{,x}\|_1^2(|v|_6^2+\|\nabla\varphi\|_2^2+ \|v_{,x}\|_1^2)+\|\eta\|_2^2\|\rot\psi_{,x}\|_1^2(|v|_6^2+\|\nabla\varphi\|_2^2)\cr
&\quad+\|\eta\|_2^2(|v|_6^2+|\nabla\varphi|_6^2)|v_{,x}|_6^2+(1+\|\eta\|_2^2) \|f\|_1^2].\cr}
\label{3.75}
\end{equation}
Now we simplify (\ref{3.75}). First we multiply it by $\nu$. Next we have to eliminate $\rot\psi$ from the r.h.s. of (\ref{3.75}). For this purpose we use the interpolations:
$$
|\rot\psi_{,xx}|_3^2\le c|\rot\psi_{,xxx}|_2^{5/3}|\rot\psi|_2^{1/3}.
$$
Then
$$\eqal{
\alpha|\rot\psi_{,xx}|_3^2&\le\varepsilon|\rot\psi_{,xxx}|_2^2+c(1/\varepsilon) \alpha^6|\rot\psi|_2^2\cr
&\le\varepsilon|\rot\psi_{,xxx}|_2^2+c(1/\varepsilon)\alpha^6|v|_2^2,\cr}
$$
where $\alpha=(1+\|\eta\|_2^2)(|v|_6^2+\|\nabla\varphi\|_2^2)$.

\noindent
Next
$$
|\rot\psi_{,xx}|_2^2\le c|\rot\psi_{,xxx}|_2^{4/3}|\rot\psi|_2^{2/3}.
$$
Hence
$$
\beta|\rot\psi_{,xx}|_2^2\le\varepsilon|\rot\psi_{,xxx}|_2^2+c(1/\varepsilon) \beta^3|v|_2^2,
$$
where $\beta=\|\eta\|_2^2(|v|_6^2+\|\nabla\varphi\|_2^2)$.

\noindent
Finally we have that
$$
|v_{,xx}|_2^2\le|\rot\psi_{,xx}|_2^2+|\nabla\varphi_{,xx}|_2^2
$$
and
$$
|v_{,x}|_6^2\le c(|v_{,xx}|_2^2+|v_{,x}|_2^2).
$$
Then (\ref{3.75}) takes the form
\begin{equation}\eqal{
&a{d\over dt}(\nu|\nabla\varphi_{,xx}|_2^2+|\rot\psi_{,xx}|_2^2)+\mu(\nu |\nabla\varphi_{,xxx}|_2^2+|\rot\psi_{,xxx}|_2^2)\cr
&\quad+\nu^2|\Delta\varphi_{,xx}|_2^2\le c[\|\eta\|_2^2\|v_{,t}\|_1^2+ \|\eta\|_2^2(\|\eta\|_2^2+1)\cr
&\quad+|\nabla\varphi_{,xxx}|_2^2(1+\|\eta\|_2^2)(|\nabla\varphi|_6^2+|v|_6^2)\cr
&\quad+(1+\|\eta\|_2^2)(|v|_6^2+\|\nabla\varphi\|_2^2)|\nabla\varphi_{,xx}|_3^2\cr
&\quad+(1+\|\eta\|_2^2)^6(|v|_6^2+\|\nabla\varphi\|_2^2)^2|v|_2^2\cr
&\quad+\|\eta\|_2^2(|v|_6^2+\|\nabla\varphi\|_2^2)\|\nabla\varphi_{,x}\|_1^2\cr
&\quad+\|\eta\|_2^6(|v|_6^2+ \|\nabla\varphi\|_2^2)^3|v|_2^2\cr
&\quad+\|\eta\|_2^2(|v|_6^2+|\nabla\varphi|_6^2)|v_{,x}|_2^2+(1+\|\eta\|_2^2) \|f\|_1^2]\cr
&\equiv c[\|\eta\|_2^2\|v_{,t}\|_1^2+\|\eta\|_2^2(\|\eta\|_2^2+1)+\phi_{13} |\nabla\varphi_{,xxx}|_2^2\cr
&\quad+\phi_{14}|\nabla\varphi_{,xx}|_3^2+\phi_{15}\|\nabla\varphi_{,x}\|_1^2+ \phi_{16}|v|_2^2+\phi_{17}|v_{,x}|_2^2\cr
&\quad+(1+\|\eta\|_2^2)\|f\|_1^2].\cr}
\label{3.76}
\end{equation}
Integrating (\ref{3.76}) with respect to time yields
\begin{equation}\eqal{
&a(\nu|\nabla\varphi_{,xx}(t)|_2^2+|\rot\psi_{,xx}(t)|_2^2)+\mu(\nu |\nabla\varphi_{,xxx}|_{2,\Omega^t}^2+|\rot\psi_{,xxx}|_{2,\Omega^t}^2)\cr
&\quad+\nu^2|\Delta\varphi_{,xx}|_{2,\Omega^t}^2\le c[\|\eta\|_{2,\infty,\Omega^t}^2\|v_{,t}\|_{1,2,\Omega^t}^2\cr
&\quad+\|\eta\|_{2,2,\Omega^t}^2(\|\eta\|_{2,\infty,\Omega^t}^2+1)+\sup_t\phi_{13} |\nabla\varphi_{,xxx}|_{2,\Omega^t}^2\cr
&\quad+\sup_t\phi_{14}|\nabla\varphi_{,xx}|_{3,2,\Omega^t}^2+\sup_t\phi_{15} \|\nabla\varphi\|_{2,2,\Omega^t}^2\cr
&\quad+(\sup_t\phi_{16}+\sup_t\phi_{17})A_1^2+(1+\|\eta\|_{2,\infty,\Omega^t}^2) \|f\|_{1,2,\Omega^t}^2]\cr
&\quad+a(\nu|\nabla\varphi_{,xx}(0)|_2^2+|\rot\psi_{,xx}(0)|_2^2),\cr}
\label{3.77}
\end{equation}
where
\begin{equation}\eqal{
&\sup_t\phi_{13}\le(1+\|\eta\|_{2,\infty,\Omega^t}^2) (|\nabla\varphi|_{6,\infty,\Omega^t}^2+D_1^2),\cr
&\sup_t\phi_{14}\le(1+\|\eta\|_{2,\infty,\Omega^t}^2) (\|\nabla\varphi\|_{2,\infty,\Omega^t}^2+D_1^2)\cr
&\sup_t\phi_{15}\le\|\eta\|_{2,\infty,\Omega^t}^2 (\|\nabla\varphi\|_{2,\infty,\Omega^t}^2+D_1^2)\cr
&\sup_t\phi_{16}\le\phi(\|\eta\|_{2,\infty,\Omega^t}^2, \|\nabla\varphi\|_{2,\infty,\Omega^t},D_1)\cr
&\sup_t\phi_{17}\le\|\eta\|_{2,\infty,\Omega^t}^2(D_1^2+ \|\nabla\varphi\|_{1,\infty,\Omega^t}^2).\cr}
\label{3.78}
\end{equation}
Using the inequalities
\begin{equation}\eqal{
&|\nabla\varphi_{,xxx}|_{2,\Omega^t}\le{\Psi\over\nu}, |\nabla\varphi_{,xx}|_{3,2,\Omega^t}\le{\Psi\over\nu},\cr
&\|\eta\|_{2,\infty,\Omega^t}\le\exp(t^{1/2}\Phi_*(t))\bigg(t^{1/2}{\Psi\over\nu} +\|\eta(0)\|_2\bigg),\cr
&\|\eta\|_{2,2,\Omega^t}\le\exp(t^{1/2}\Phi_*(t))\bigg(t^{3/2}{\Psi\over\nu}+ \|\eta(0)\|_2\bigg),\cr}
\label{3.79}
\end{equation}
where $\|\eta(0)\|_2\le{c_0\over\nu}$,
$$\eqal{
&|D_1|\le\phi\bigg({\Psi\over\nu^\alpha}\bigg)+c|\eta|_{3,6,\Omega^t}+A_2,\cr
&\|\nabla\varphi\|_{2,\infty,\Omega^t}\le c{\chi_0\over\sqrt{\nu}},\cr}
$$
we write (\ref{3.77}) in the form
\begin{equation}\eqal{
&a(\nu|\nabla\varphi_{,xx}(t)|_2^2+|\rot\psi_{,xx}(t)|_2^2)+\mu(\nu |\nabla\varphi_{,xxx}|_{2,\Omega^t}^2+|\rot\psi_{,xxx}|_{2,\Omega^t}^2)\cr
&\quad+\nu^2|\Delta\varphi_{,xx}|_{2,\Omega^t}^2\le\exp(\sqrt{t}\Phi_*(t)) \bigg(t{\Psi^2\over\nu^2}+{c_0^2\over\nu^2}\bigg)\|v_{,t}\|_{1,2,\Omega^t}^2\cr
&\quad+\phi\bigg(A_1,A_2,\exp(\sqrt{t}\Phi_*(t))\bigg(\sqrt{t}{\Psi\over\nu}+ {c_0\over\nu}\bigg),{\Psi\over\nu},{\chi_0\over\sqrt{\nu}}\bigg)\cr
&\quad+\bigg(1+\exp(\sqrt{t}\Phi_*(t))\bigg(t{\Psi^2\over\nu^2}+ {c_0^2\over\nu^2}\bigg)\bigg)\|f\|_{1,2,\Omega^t}^2\cr
&\quad+a(\nu|\nabla\varphi_{,xx}(0)|_2^2+|\rot\psi_{,xx}(0)|_2^2).\cr}
\label{3.80}
\end{equation}
Inequality (\ref{3.76}) implies (\ref{3.58}) and (\ref{3.80}) yields (\ref{3.59}). This concludes the proof.
\end{proof}

\begin{remark}\label{r3.4}
Finally, we summarize estimates derived in this section. First we derive the differential inequality. From (\ref{3.57}) and (\ref{3.58}) we have
\begin{equation}\eqal{
&a{d\over dt}(\nu|\nabla\varphi|_{2,1}^2+|\rot\psi|_{2,1}^2)+(\nu |\nabla\varphi|_{3,1}^2+|\rot\psi|_{3,1}^2)\cr
&\quad+\nu^2|\nabla\varphi|_{3,1}^2\le\phi_{18}(\exp(\sqrt{t}\Phi_*(t)) \bigg(\sqrt{t}{\Psi\over\nu}+{c_0\over\nu}\bigg),D_1,D_2,A_1, |\nabla\varphi|_{2,1})\cdot\cr
&\quad\cdot(\|\nabla\varphi_{,x}\|_2^2+|v|_2^2+|\nabla v|_2^2)+c(1+|\eta|_{2,1}^2)|f|_{1,1}^2,\cr}
\label{3.81}
\end{equation}
where we used that
\begin{equation}
|\eta|_{2,1}^2\le\phi\bigg(\exp(\sqrt{t}\Phi_*(t))\cdot\bigg(\sqrt{t}{\Psi\over\nu}+ {c_0\over\nu}\bigg)\bigg).
\label{3.82}
\end{equation}
Next, (\ref{3.52}) and (\ref{3.80}) imply
\begin{equation}\eqal{
&a(\nu|\nabla\varphi(t)|_{2,1}^2+|\rot\psi(t)|_{2,1}^2)+\mu(\nu |\nabla\varphi|_{3,1,2,\Omega^t}^2+|\rot\psi|_{3,1,2,\Omega^t}^2)\cr
&\quad+\nu^2|\nabla\varphi|_{3,1,2,\Omega^t}^2\le\phi_{19}(D_1,D_2,A_1, {\Psi\over\nu},{\chi_0\over\sqrt{\nu}},|\eta|_{2,1,2,\Omega^t}, |\eta|_{2,1,\infty,\Omega^t})\cr
&\quad+c(1+|\eta|_{2,1,\infty,\Omega^t}^2)|f|_{1,1,2,\Omega^t}^2+c(\nu |\nabla\varphi(0)|_{2,1}^2+|\rot\psi(0)|_{2,1,}^2).\cr}
\label{3.83}
\end{equation}
\end{remark}

\section{Estimate and Existence}\label{s4}

\begin{lemma}\label{l4.1}
Let Notation \ref{n2.13} hold. Let $\nu>0$ be given. Assume that $\eta(0)\in L_\infty(\Omega)$, $\eta(0),\nabla\varphi(0),\rot\psi(0)\in\Gamma_1^2(\Omega)$, $f\in L_2(0,T;H^1(\Omega))$, $f_t\in L_2(\Omega^T)$, $f\in L_6(0,T;L_3(\Omega))$, $f\in L_1(0,T;L_\infty(\Omega))$, $\nu|\nabla\varphi(0)|_{2,1}^2\le\const$.\\
Then for sufficiently large $\nu$ there exists a constant $A$ depending on all above assumptions such that
\begin{equation}
\chi_1(t)+\chi_2(t)+\Psi(t)\le A,\quad t\le T,
\label{4.1}
\end{equation}
where $T$ is proportional to $\nu$.
\end{lemma}

\begin{proof}
Using (\ref{2.92})--(\ref{2.96}), (\ref{3.79}) and Notation \ref{n2.13} in (\ref{3.83}) we have
\begin{equation}\eqal{
&\chi_1^2+\chi_2^2+\Psi^2\le\phi_{20}(\exp(t^{1/2}\Phi_*)\bigg(t^{1/2} {\Psi\over\nu}+c_0/\nu\bigg),\cr
&\phi\bigg({\Psi\over\nu^\alpha}\bigg),{\Psi\over\nu},{\chi_0\over\sqrt{\nu}}, A_1,A_2|f|_{1,1,2,\Omega^t}\bigg)\cr
&\quad+c(\nu|\nabla\varphi(0)|_{2,1}^2+|\rot\psi(0)|_{2,1,}^2),\cr}
\label{4.2}
\end{equation}
where $\phi(0)=0$.

\noindent
To show (\ref{4.1}) we take constant $A$ so large that
\begin{equation}
\phi_{20}(\exp(t^{1/2}\Phi_*){c_0\over\nu},0,0,0,A_1,A_2,|f|_{1,1,2,\Omega^t})+ c\Phi_0^2(0)<A.
\label{4.3}
\end{equation}
Then for $\nu$ sufficiently large there exists a constant $A$ such that
\begin{equation}\eqal{
&\phi_{20}(\exp(t^{1/2}\Phi_*)\bigg(t^{1/2}{A\over\nu}+{c_0\over\nu}\bigg),\phi \bigg({A\over\nu^\alpha}\bigg),{A\over\nu},{A\over\sqrt{\nu}},\cr
&A_1,A_2,|f|_{1,1,2,\Omega^t})+c\Phi_0^2(0)\le A.\cr}
\label{4.4}
\end{equation}
Hence (\ref{4.1}) holds. This concludes the proof.
\end{proof}

\begin{corollary}\label{c4.2}
Let the assumptions of Lemma \ref{l4.1} hold. Then (\ref{4.4}) implies
\begin{equation}\eqal{
&|\nabla\varphi(t)|_{2,1}^2+{1\over\nu}|\rot\psi(t)|_{2,1}^2+\mu |\nabla\varphi|_{3,1,2,\Omega^t}^2+\nu|\nabla\varphi|_{3,1,2,\Omega^t}^2+ {\mu\over\nu}|\rot\psi|_{3,1,2,\Omega^t}^2\cr
&\le{1\over\nu}\phi(A)+|\nabla\varphi(0)|_{2,1}^2+{1\over\nu} |\rot\psi(0)|_{2,1}^2.\cr}
\label{4.5}
-\end{equation}
From (\ref{4.5}) and estimates for $\eta$ (see Lemmas \ref{2.9}, \ref{2.10}) we have
\begin{equation}
|\eta|_{2,1,\infty,\Omega^t}\le{1\over\nu}\phi(A)
\label{4.6}
\end{equation}
where
$$
|\eta(0)|_{2,1}\le c_0/\nu
$$
and
\begin{equation}
|\nabla\varphi(t)|_{2,1}^2+\nu|\nabla\varphi|_{3,1,2,\Omega^t}^2\le{1\over\nu} \phi(A)+|\nabla\varphi(0)|_{2,1}^2+{1\over\nu}|\rot\psi(0)|_{2,1}^2,
\label{4.7}
\end{equation}
so quantities $\eta$ and $\varphi$ can be made in these norms as small as we want for sufficiently large $\nu$.
\end{corollary}

\begin{theorem}\label{t4.3}
Let the assumptions of Lemma \ref{l4.1} hold.
Then there exists a solution to problem (\ref{1.7})--(\ref{1.9}) such that $\psi,\varphi\in L_\infty(0,t;\Gamma_1^2(\Omega))\cap L_2(0,t;\Gamma_1^3(\Omega))$, $t\le T$ as long as estimate (\ref{4.1}) holds.
\end{theorem}

\begin{proof}
For $T$ sufficiently small there exists a solution to problem (\ref{1.7})--(\ref{1.9}) in the above spaces. Having estimate (\ref{4.1}) for $t\le T$ the local solution can be extended in time up to time $t=T$. This ends the proof.
\end{proof}

\section{Global estimate and existence}\label{s5}

In Section \ref{s4} we proved long time estimate for solutions to problem (\ref{1.7})--(\ref{1.9}), where the estimate time is proportional to $\nu$. Since $\nu$ is finite we have only finite time estimate. Hence there exist finite time solutions. Therefore, to prove global existence we need additional differential inequality. This is derived in this Section.

\begin{lemma}\label{l5.1}
Assume that $\varphi$, $\psi$, $\eta$ are solutions to problem (\ref{3.1})--(\ref{3.3}). Let the assumptions of Corollary \ref{c4.2} hold. Then
\begin{equation}\eqal{
&{d\over dt}\bigg(|\nabla\varphi|_2^2+{1\over\nu}|\rot\psi|_2^2\bigg)+{\mu\over a}|\nabla^2\varphi|_2^2+{\nu\over a}|\Delta\varphi|_2^2+{\mu\over a\nu}|\nabla\rot\psi|_2^2\cr
&\le{c\over\nu}[|\eta|_2^2+|\eta|_3^2|v_{,t}|_2^2+|\Delta\varphi|_3^2|v|_2^2+ |v|_3^2|\nabla v|_2^2+ |f|_2^2].\cr}
\label{5.1}
\end{equation}
\end{lemma}

\begin{proof}
Multiplying (\ref{3.2}) by $\nabla\varphi$ and integrating over $\Omega$ yields
\begin{equation}\eqal{
&{a\over2}{d\over dt}|\nabla\varphi|_2^2+\mu|\nabla^2\varphi|_2^2+\nu |\Delta\varphi|_2^2=-a_0\intop_\Omega\nabla\eta\cdot\nabla\varphi dx- \intop_\Omega\eta v_{,t}\cdot\nabla\varphi dx\cr
&\quad-\intop_\Omega(a+\eta)v\cdot\nabla v\cdot\nabla\varphi dx+\intop_\Omega[p_\varrho(a)-p_\varrho(a+\eta)]\nabla\eta\cdot\nabla\varphi dx\cr
&\quad+\intop_\Omega(a+\eta)f\cdot\nabla\varphi dx.\cr}
\label{5.2}
\end{equation}
Integrating by parts the first term on the r.h.s. is bounded by
$$
\varepsilon|\Delta\varphi|_2^2+c/\varepsilon|\eta|_2^2,
$$
the second by
$$
\varepsilon|\nabla\varphi|_6^2+c/\varepsilon|\eta|_3^2|v_{,t}|_2^2
$$
the third term on the r.h.s. of (\ref{5.2}) is expressed in the form
$$
-a\intop_\Omega v\cdot\nabla v\cdot\nabla\varphi dx-\intop_\Omega\eta v\cdot\nabla v\cdot\nabla\varphi dx\equiv I_1+I_2.
$$
integrating by parts in $I_1$ yields
$$
I_1=a\intop_\Omega\Delta\varphi v\cdot\nabla\varphi dx+\intop_\Omega v\cdot v\cdot\nabla\nabla\varphi dx\equiv I_{11}+I_{12},
$$
where
$$\eqal{
&|I_{11}|\le\varepsilon|\nabla\varphi|_6^2+c/\varepsilon|\Delta\varphi|_3^2 |v|_2^2,\cr
&|I_{12}|\le\varepsilon|\nabla^2\varphi|_2^2+c/\varepsilon|v|_6^2|v|_3^2.\cr}
$$
Next we examine $I_2$. Then we have
$$
|I_2|\le\varepsilon|\nabla\varphi|_6^2+c/\varepsilon|\eta|_\infty^2|v|_3^2 |\nabla v|_2^2.
$$
The fourth term on the r.h.s. of (\ref{5.2}) is estimated by
$$
\varepsilon|\nabla\varphi|_6^2+c/\varepsilon|\eta|_3^2|\nabla\eta|_2^2.
$$
Finally, the last term on the r.h.s. of (\ref{5.2}) equals
$$
I_3=a\intop_\Omega f_g\cdot\nabla\varphi dx+\intop_\Omega\eta f\cdot\nabla\varphi dx.
$$
Hence
$$
|I_3|\le\varepsilon|\nabla\varphi|_2^2+c/\varepsilon|f_g|_2^2+\varepsilon |\nabla\varphi|_6^2+c/\varepsilon|\eta|_3^2|f|_2^2.
$$
Employing the above estimates in (\ref{5.2}) and assuming that $\varepsilon$ is sufficiently small yields
\begin{equation}\eqal{
&a{d\over dt}|\nabla\varphi|_2^2+\mu|\nabla^2\varphi|_2^2+\nu|\Delta\varphi|_2^2 \le{c\over\nu}[|\eta|_2^2+|\eta|_3^2|v_{,t}|_2^2+|\Delta\varphi|_3^2|v|_2^2\cr
&\quad+|v|_6^2|v|_3^2+|\eta|_\infty^2|v|_3^2|\nabla v|_2^2+|\eta|_3^2 |\nabla\eta|_2^2+|f_g|_2^2+|\eta|_3^2|f|_2^2].\cr}
\label{5.3}
\end{equation}
Multiplying (\ref{3.2}) by $\rot\psi$ and integrating the result over $\Omega$ implies
\begin{equation}\eqal{
&{a\over2}{d\over dt}|\rot\psi|_2^2+\mu|\nabla\rot\psi|_2^2=-\intop_\Omega\eta v_{,t}\cdot\rot\psi dx-\intop_\Omega(a+\eta)v\cdot\nabla v\cdot\rot\psi dx\cr
&\quad+\intop_\Omega(a+\eta)f\cdot\rot\psi dx.\cr}
\label{5.4}
\end{equation}
We estimate the first term on the r.h.s. of (\ref{5.4}) by
$$
\varepsilon|\rot\psi|_6^2+c/\varepsilon|\eta|_3^2|v_{,t}|_2^2.
$$
The second term on the r.h.s. of (\ref{5.4}) is expressed in the form
$$
-a\intop_\Omega v\cdot\nabla v\cdot\rot\psi dx-\intop_\Omega\eta v\cdot\nabla v\cdot\rot\psi dx\equiv J_1+J_2,
$$
where
$$\eqal{
&|J_1|\le\varepsilon|\rot\psi|_6^2+c/\varepsilon|v|_3^2|\nabla v|_2^2,\cr
&|J_2|\le\varepsilon|\rot\psi|_6^2+c/\varepsilon|\eta|_\infty^2|v|_3^2|\nabla v|_2^2.\cr}
$$
Finally, the last term on the r.h.s. of (\ref{5.4}) takes the form
$$
a\intop_\Omega f_r\cdot\rot\psi dx+\intop_\Omega\eta f\cdot\rot\psi dx\equiv J_3.
$$
Hence,
$$
|J_3|\le\varepsilon|\rot\psi|_2^2+c/\varepsilon|f_r|_2^2+\varepsilon |\rot\psi|_6^2+c/\varepsilon|\eta|_3^2|f|_2^2.
$$
Employing the above estimates in (\ref{5.4}) and using that $\varepsilon$ is sufficiently small we derive the inequality
\begin{equation}\eqal{
&{a\over2}{d\over dt}|\rot\psi|_2^2+\mu|\nabla\rot\psi|_2^2\le c[|\eta|_3^2|v_{,t}|_2^2+|v|_3^2|\nabla v|_2^2+|\eta|_\infty^2|v|_3^2|\nabla v|_2^2\cr
&\quad+|f_r|_2^2+|\eta|_3^2|f|_2^2].\cr}
\label{5.5}
\end{equation}
Multiplying (\ref{5.5}) by $1/\nu$, adding to (\ref{5.3}) and using that $|\eta|\le a/2$ we obtain (\ref{5.1}). This concludes the proof.
\end{proof}

\begin{lemma}\label{l5.2}
Assume that $\varphi$, $\psi$, $\eta$ are solutions to problem (\ref{3.1})--(\ref{3.3}). Use (\ref{3.2}) in the form (\ref{3.11}). Let the assumptions of Corollary \ref{c4.2} hold.\\
Then
\begin{equation}\eqal{
&{d\over dt}\bigg(|\nabla\varphi_{,t}|_2^2+{1\over\nu}|\rot\psi_{,t}|_2^2\bigg)+ {\mu\over a}|\nabla^2\varphi_{,t}|_2^2+{\mu\over a}|\Delta\varphi_{,t}|_2^2+{\mu\over a\nu}|\nabla\rot\psi_{,t}|_2^2\cr
&\le{c\over\nu}[|\eta_{,t}|_2^2+|\eta|_{2,1}^2(|\nabla\rot\psi|_3^2+ |\Delta\varphi|_3^2)\cr
&\quad+|\Delta\varphi|_3^2|v_{,t}|_2^2+|v_{,t}|_3^2|v|_6^2+|f_t|_2^2]+ c\nu|\eta|_{2,1}^2|\Delta\varphi|_3^2.\cr}
\label{5.6}
\end{equation}
\end{lemma}

\begin{proof}
Differentiate (\ref{3.13}) with respect to $t$, multiply by $\nabla\varphi_{,t}$ and integrate over $\Omega$. Then we have
\begin{equation}\eqal{
&{1\over2}{d\over dt}|\nabla\varphi_{,t}|_2^2+{\mu\over a}|\nabla^2\varphi_{,t}|_2^2+{\nu\over a}|\Delta\varphi_{,t}|_2^2=-\intop_\Omega \bigg({a_0\over a+\eta}\nabla\eta\bigg)_{,t}\cdot\nabla\varphi_{,t}dx\cr
&\quad-{\mu\over a}\intop_\Omega\bigg({\eta\over a+\eta}\Delta v\bigg)_{,t}\cdot \nabla\varphi_{,t}dx-{\nu\over a}\intop_\Omega\bigg({\eta\over a+\eta} \nabla\Delta\varphi\bigg)_{,t}\cdot\nabla\varphi_{,t}dx\cr
&\quad-\intop_\Omega(v\cdot\nabla v)_{,t}\cdot\nabla\varphi_{,t}dx+\intop_\Omega \bigg[{1\over a+\eta}(p_\varrho(a)-p_\varrho(a+\eta))\nabla\eta\bigg]_{,t}\cdot \nabla\varphi_{,t}dx\cr
&\quad+\intop_\Omega f_{g,t}\cdot\nabla\varphi_tdx.\cr}
\label{5.7}
\end{equation}
The first term on the r.h.s. is bounded by
$$\eqal{
&c\intop_\Omega(|\nabla\eta_{,t}|\,|\nabla\varphi_{,t}|+|\eta_{,t}|\, |\nabla\eta|\nabla\varphi_{,t}|)dx\cr
&\le\varepsilon|\nabla^2\varphi_{,t}|_2^2+c/\varepsilon |\eta_{,t}|_2^2+\varepsilon|\nabla\varphi_{,t}|_6^2+c/\varepsilon |\eta_{,t}|_3^2|\nabla\eta|_2^2.\cr}
$$
The second term on the r.h.s. of (\ref{5.7}) equals
$$
-{\mu\over a}\intop_\Omega\bigg({\eta\over a+\eta}\Delta\rot\psi\bigg)_{,t}\cdot \nabla\varphi_{,t}dx-{\mu\over a}\intop_\Omega\bigg({\eta\over a+\eta} \Delta\nabla\varphi\bigg)_{,t}\cdot\nabla\varphi_{,t}dx\equiv I_1+I_2.
$$
Using estimates of $I_1$ and $I_2$ in the proof of Lemma \ref{l3.2} we have
$$
|I_1|\le\varepsilon\|\nabla\varphi_{,t}\|_1^2+c/\varepsilon(|\eta|_{2,1}^4 |\nabla\rot\psi|_2^2+|\eta|_{2,1}^2|\nabla\rot\psi|_3^2+\|\eta\|_2^2 |\nabla\rot\psi_{,t}|_2^2)
$$
and
$$\eqal{
&|I_2|\le\varepsilon\|\nabla\varphi_{,t}\|_1^2+c/\varepsilon(|\eta|_{2,1}^4 |\nabla^2\varphi|_2^2+|\eta|_{2,1}^2|\nabla^2\varphi|_3^2+\|\eta\|_2^2 |\nabla\varphi_{,t}|_2^2)\cr
&\quad+{\mu\over a}\intop_\Omega{\eta\over a+\eta} |\nabla^2\varphi_{,t}|^2dx,\cr}
$$
where the last term is absorbed by the second term on the l.h.s. of (\ref{5.7}). Consider the third term on the r.h.s. of (\ref{5.7}). Repeating the proof of estimate of $I_3$ in the proof of Lemma \ref{l3.2} we obtain
$$\eqal{
&|I_3|\le\nu\varepsilon\|\nabla\varphi_{,t}\|_1^2+{\nu c\over\varepsilon} (|\eta|_{2,1}^4|\Delta\varphi|_2^2+|\eta|_{2,1}^2|\Delta\varphi|_3^2+ \|\eta\|_2^2|\nabla\varphi_t|_3^2)\cr
&\quad+{\nu\over a}\intop_\Omega{\eta\over a+\eta} |\Delta\varphi_{,t}|^2dx,\cr}
$$
where the last term is absorbed by the last term on the l.h.s. of (\ref{5.7}).

\noindent
We write the fourth term on the r.h.s. of (\ref{5.7}) in the form
$$\eqal{
I_4&=\intop_\Omega v_{,t}\cdot\nabla v\cdot\nabla\varphi_{,t}dx+\intop_\Omega v\cdot\nabla v_{,t}\cdot\nabla\varphi_{,t}dx=-\intop_\Omega\Delta\varphi_{,t}v \cdot\nabla\varphi_{,t}dx\cr
&\quad-\intop_\Omega v_{,t}v\cdot\nabla^2\varphi_{,t}dx-\intop_\Omega\Delta\varphi v_{,t}\cdot\nabla\varphi_{,t}dx-\intop_\Omega vv_{,t}\cdot\nabla^2\varphi_{,t}dx.\cr}
$$
Hence, we have
$$\eqal{
|I_4|&\le\varepsilon|\nabla\varphi_{,t}|_6^2+c/\varepsilon (|\Delta\varphi_{,t}|_2^2|v|_3^2+|\Delta\varphi|_3^2|v_{,t}|_2^2)+\varepsilon |\nabla^2\varphi_{,t}|_2^2\cr
&\quad+c/\varepsilon|v_{,t}|_3^2|v|_6^2.\cr}
$$
We estimate the fifth term on the r.h.s. of (\ref{5.7}) by
$$
|I_5|\le c\intop_\Omega(|\eta_{,t}|\,|\nabla\eta|+|\eta|\,|\nabla\eta_t|) |\nabla\varphi_{,t}|dx\le\varepsilon|\nabla\varphi_{,t}|_6^2+c/\varepsilon |\eta|_{2,1}^4.
$$
Finally, we estimate the last term on the r.h.s. of (\ref{5.7}) by
$$
\varepsilon|\nabla\varphi_{,t}|_6^2+c/\varepsilon|f_{g,t}|_2^2.
$$
Employing the above estimates in (\ref{5.7}) and assuming that $\varepsilon$ is sufficiently small we have
\begin{equation}\eqal{
&{d\over dt}|\nabla\varphi_{,t}|_2^2+{\mu\over a}\|\nabla\varphi_{,t}\|_1^2+ {\nu\over a}|\Delta\varphi_{,t}|_2^2\le {c\over\nu}[|\eta_{,t}|_2^2+|\eta|_{2,1}^4\cr
&\quad+|\eta|_{2,1}^4(|\nabla\rot\psi|_2^2+|\Delta\varphi|_2^2)+|\eta|_{2,1}^2(|\nabla\rot\psi|_3^2+|\Delta\varphi|_3^2)\cr
&\quad+\|\eta\|_2^2 (|\nabla\rot\psi_{,t}|_2^2+|\nabla^2\varphi_{,t}|_2^2)+|\Delta\varphi_{,t}|_2^2 |v|_3^2\cr
&\quad+|\Delta\varphi|_3^2|v_{,t}|_2^2+|v_{,t}|_3^2|v|_6^2+|f_t|_2^2]+c\nu (|\eta|_{2,1}^4|\Delta\varphi|_2^2+|\eta|_{2,1}^2|\Delta\varphi|_3^2\cr
&\quad+\|\eta\|_2^2|\nabla\varphi_{,t}|_3^2).\cr}
\label{5.8}
\end{equation}
Differentiate (\ref{3.13}) with respect to $t$, multiply by $\rot\psi_{,t}$ and integrate over $\Omega$. Then we have
\begin{equation}\eqal{
&{1\over2}{d\over dt}|\rot\psi_{,t}|_2^2+{\mu\over a}|\nabla\rot\psi_{,t}|_2^2= -{\mu\over a}\intop_\Omega\bigg({\eta\over a+\eta}\Delta v\bigg)_{,t}\cdot \rot\psi_{,t}dx\cr
&\quad-{\nu\over a}\intop_\Omega\bigg({\eta\over a+\eta}\Delta\nabla\varphi\bigg)_{,t}\cdot\rot\psi_{,t}dx- \intop_\Omega(v\cdot\nabla v)_{,t}\cdot\rot\psi_{,t}dx\cr
&\quad+\intop_\Omega f_{r,t}\cdot\rot\psi_{,t}dx.\cr}
\label{5.9}
\end{equation}
Now we examine the particular terms from the r.h.s. of (\ref{5.9}). Looking for the estimate of the first term on the r.h.s. of (\ref{3.18}) we see that the first term on the r.h.s. of (\ref{5.9}) is estimated by
$$\eqal{
&\varepsilon\|\rot\psi_{,t}\|_1^2+c/\varepsilon[|\eta|_{2,1}^4 (|\nabla\rot\psi|_2^2+|\Delta\varphi|_2^2)+|\eta|_{2,1}^2(|\nabla\rot\psi|_3^2+ |\Delta\varphi|_3^2)\cr
&\quad+\|\eta\|_2^2(|\rot\psi_{,t}|_3^2+|\Delta\varphi_{,t}|_3^2)]+{\mu\over a} \intop_\Omega{\eta\over a+\eta}|\nabla\rot\psi_{,t}|^2dx,\cr}
$$
where we used Corollary \ref{c4.2} and the last term is absorbed by the second term on the l.h.s. of (\ref{5.9}).

\noindent
Looking for the estimate of the second term on the r.h.s. of (\ref{3.18}) (term $I_3$) we see that the second term on the r.h.s. of (\ref{5.9}) is bounded by
$$
\varepsilon|\rot\psi_{,t}|_6^2+{c\nu^2\over\varepsilon}(|\eta|_{2,1}^4 |\Delta\varphi|_2^2+|\nabla\eta_{,t}|_2^2|\Delta\varphi|_3^2+\|\eta\|_2^2 |\Delta\varphi_t|_2^2).
$$
The third term on the r.h.s. of (\ref{5.9}) equals
$$
J=-\intop_\Omega(v\cdot\nabla v_t+v_t\cdot\nabla v)\cdot\rot\psi_{,t}dx\equiv J_1+J_2.
$$
First we examine $J_1$. We write it in the form
$$
J_1=-\intop_\Omega v\cdot\nabla\rot\psi_{,t}\cdot\rot\psi_{,t}dx-\intop_\Omega v\cdot\nabla\nabla\varphi_{,t}\cdot\rot\psi_{,t}dx\equiv J_{11}+J_{12},
$$
where
$$
J_{11}=-{1\over2}\intop_\Omega v\cdot\nabla|\rot\psi_{,t}|^2dx={1\over2} \intop_\Omega\Delta\varphi|\rot\psi_{,t}|^2dx
$$
so
$$
|J_{11}|\le\varepsilon|\rot\psi_{,t}|_6^2+c/\varepsilon|\Delta\varphi|_3^2 |\rot\psi_{,t}|_2^2.
$$
Integrating by parts in $J_{12}$ yields
$$J_{12}=\intop_\Omega\nabla v\cdot\nabla\varphi_{,t}\cdot\rot\psi_{,t}dx
$$
so
$$
|J_{12}|\le\varepsilon|\rot\psi_{,t}|_6^2+c/\varepsilon|\nabla v|_3^2 |\nabla\varphi_{,t}|_2^2.
$$
Next, we examine $J_2$. Integration by parts implies
$$
J_2=\intop_\Omega\Delta\varphi_{,t}v\cdot\rot\psi_{,t}dx+\intop_\Omega v_{,t}v\cdot\nabla\rot\psi_{,t}dx.
$$
Hence, we have
$$
|J_2|\le\varepsilon\|\rot\psi_{,t}\|_1^2+c/\varepsilon(|\Delta\varphi_{,t}|_2^2 |v|_3^2+|v_{,t}|_3^2|v|_6^2).
$$
Finally, the last term on the r.h.s. of (\ref{5.9}) is bounded by
$$
\varepsilon|\rot\psi_{,t}|_2^2+c/\varepsilon|f_{r,t}|_2^2.
$$
Employing the above estimates in (\ref{5.9}) and using that $\varepsilon$ is sufficiently small we derive the inequality
\begin{equation}\eqal{
&{d\over dt}|\rot\psi_{,t}|_2^2+{\mu\over a}|\nabla\rot\psi_{,t}|_2^2\le c[|\eta|_{2,1}^2(|\eta|_{2,1}^2+1)|\nabla\rot\psi|_3^2\cr
&\quad+(\|\eta\|_2^2+|\Delta\varphi|_3^2)|\rot\psi_{,t}|_3^2+|\eta|_{2,1}^4 |\Delta\varphi|_3^2+|\nabla\eta|_3^2|\Delta\varphi_{,t}|_2^2+|\nabla v|_2^2 |\nabla\varphi_{,t}|_3^2\cr
&\quad+|v_{,t}|_3^2|v|_6^2+|f_t|_2^2]+c\nu^2[|\eta|_{2,1}^2(|\eta|_{2,1}^2+1) |\Delta\varphi|_3^2+\|\eta\|_2^2|\Delta\varphi_{,t}|_2^2].\cr}
\label{5.10}
\end{equation}
Multiplying (\ref{5.10}) by $1/\nu$, adding to (\ref{5.8}) and using that $c(\|\eta\|_2^2+|\Delta\varphi|_3^2)<{\mu\over a}$ in view of Corollary \ref{c4.2} we obtain the inequality
\begin{equation}\eqal{
&{d\over dt}\bigg(|\nabla\varphi_{,t}|_2^2+{1\over\nu}|\rot\psi_{,t}|_2^2\bigg)+ {\mu\over a}|\nabla^2\varphi_{,t}|_2^2+{\nu\over a}|\Delta\varphi_{,t}|_2^2+ {\mu\over a\nu}|\nabla\rot\psi_{,t}|_2^2\cr
&\le{c\over\nu}[|\eta_{,t}|_2^2+|\eta|_{2,1}^2(|\Delta\varphi|_2^2+ |\nabla\rot\psi|_3^2+|\Delta\varphi|_3^2)\cr
&\quad+|\Delta\varphi|_3^2|v_t|_2^2+ |v_t|_3^2|v|_6^2+|f_t|_2^2]+c\nu|\eta|_{2,1}^2|\Delta\varphi|_3^2,\cr}
\label{5.11}
\end{equation}
where we used that $c\|\eta\|_2^2<1$, $c\|v(t)\|_1^2<c^2$.
Then (\ref{5.11}) implies inequality (\ref{5.6}). This concludes the proof.
\end{proof}

\begin{remark}\label{r5.3}
From (\ref{5.1}) and (\ref{5.6}) under the restrictions
$$
c\|v(t)\|_1^2<c^2,\quad c(\|\eta\|_2^2+|\Delta\varphi|_3^2)<{\mu\over a},\quad c\|\eta\|_2^2<1
$$
which holds in view of Corollary \ref{c4.2} we obtain the inequality
\begin{equation}\eqal{
&{d\over dt}\bigg(|\nabla\varphi|_2^2+|\nabla\varphi_{,t}|_2^2+ {1\over\nu}|\rot\psi|_2^2+{1\over\nu}|\rot\psi|_2^2+ {1\over\nu}|\rot\psi_{,t}|_2^2\bigg)\cr
&\quad+{\mu\over a}(|\nabla^2\varphi|_2^2+ |\nabla^2\varphi_{,t}|_2^2)+{\nu\over a}(|\Delta\varphi|_2^2+|\Delta\varphi_{,t}|_2^2)\cr
&\quad+{\mu\over a\nu}(|\nabla\rot\psi|_2^2+|\nabla\rot\psi_{,t}|_2^2)\cr
&\le{c\over\nu}\Big[|\eta|_3^2|v_{,t}|_2^2+|\eta|_2^2+|\eta_t|_2^2+|\eta|_{2,1}^2 (|\nabla\rot\psi|_3^2+|\Delta\varphi|_3^2)\cr
&\quad+|\Delta\varphi|_3^2|v_{,t}|_2^2+ |v|_3^2|\nabla v|_2^2+|v_{,t}|_3^2 |v|_6^2+|f|_2^2+|f_t|_2^2\Big]\cr
&\quad+c\nu|\eta|_{2,1}^2|\Delta\varphi|_3^2.\cr}
\label{5.12}
\end{equation}
\end{remark}

\begin{lemma}\label{l5.4}
Assume that $\varphi$, $\psi$, $\eta$ are solutions to problem (\ref{3.1})--(\ref{3.3}). Let the assumptions of Corollary \ref{c4.2} hold. Then
\begin{equation}\eqal{
&{d\over dt}\bigg(|\nabla\varphi_{,x}|_2^2+{1\over\nu}|\rot\psi_{,x}|_2^2\bigg)+ {\mu\over a}|\nabla^2\varphi_{,x}|_2^2+{\nu\over a}|\Delta\varphi_{,x}|_2^2+ {\mu\over\nu a}|\nabla\rot\psi_{,x}|_2^2\cr
&\le{c\over\nu}[|\eta|_\infty^2|v_{,t}|_2^2+|\nabla\eta|_2^2+|v|_6^2|\nabla v|_3^2+|f|_2^2].\cr}
\label{5.13}
\end{equation}
\end{lemma}

\begin{proof}
Differentiate (\ref{3.2}) with respect to $x$, multiply by $\nabla\varphi_x$ and integrate over $\Omega$. Then we get
\begin{equation}\eqal{
&{a\over2}{d\over dt}|\nabla\varphi_{,x}|_2^2+\mu|\nabla^2\varphi_{,x}|_2^2+\nu |\Delta\varphi_{,x}|_2^2=-\intop_\Omega(\eta v_{,t})_{,x}\cdot\nabla\varphi_{,x}dx\cr
&\quad-a_0\intop_\Omega\nabla\eta_{,x}\cdot\nabla\varphi_{,x}dx+\intop_\Omega [(a+\eta)v\cdot\nabla v]_{,x}\cdot\nabla\varphi_{,x}dx\cr
&\quad+\intop_\Omega[(p_\varrho(a)-p_\varrho(a+\eta))\nabla\eta]_{,x}\cdot \nabla\varphi_{,x}dx+\intop_\Omega[(a+\eta)f]_{,x}\cdot\nabla\varphi_{,x}dx.\cr}
\label{5.14}
\end{equation}
After integration by parts in the first term on the r.h.s. of (\ref{5.14}) we bound it by
$$
\varepsilon|\nabla\varphi_{,xx}|_2^2+c/\varepsilon|\eta|_\infty^2|v_{,t}|_2^2,
$$
the second term is bounded by
$$
\varepsilon|\nabla\varphi_{,xx}|_2^2+c/\varepsilon|\nabla\eta|_2^2.
$$
After integration by parts the third term on the r.h.s. of (\ref{5.14}) is estimated by
$$
\varepsilon|\nabla\varphi_{,xx}|_2^2+c/\varepsilon(1+|\eta|_\infty^2) |v|_6^2|\nabla v|_3^2
$$
and the fourth term by
$$
\varepsilon|\nabla\varphi_{,xx}|_2^2+c/\varepsilon|\eta|_\infty^2|\nabla\eta|_2^2.
$$
Finally, the last term on the r.h.s. of (\ref{5.14}) equals
$$
I_1=-a\intop_\Omega f_g\cdot\nabla\varphi_{,xx}dx-\intop_\Omega\eta f\nabla\varphi_{,xx}dx
$$
which is estimated by
$$
|I_1|\le\varepsilon|\nabla\varphi_{,xx}|_2^2+c/\varepsilon(|f_g|_2^2+ |\eta|_\infty^2|f|_2^2).
$$
Employing the above estimates in (\ref{5.14}) and assuming that $\varepsilon$ is sufficiently small we obtain the inequality
\begin{equation}\eqal{
&{d\over dt}|\nabla\varphi_{,x}|_2^2+{\mu\over a}|\nabla^2\varphi_{,x}|_2^2+{\nu\over a}|\Delta\varphi_{,x}|_2^2\le{c\over\nu}[|\eta|_\infty^2|v_{,t}|_2^2+ |\nabla\eta|_2^2\cr
&\quad+|v|_6^2|\nabla v|_3^2+|\eta|_\infty^2|\nabla\eta|_2^2+(1+|\eta|_\infty^2) |f|_2^2].\cr}
\label{5.15}
\end{equation}
Differentiate (\ref{3.2}) with respect to $x$, multiply by $\rot\psi_{,x}$ and integrate over $\Omega$. Then we derive
\begin{equation}\eqal{
&{a\over2}{d\over dt}|\rot\psi_{,x}|_2^2+\mu|\nabla\rot\psi_{,x}|_2^2= -\intop_\Omega(\eta v_{,t})_{,x}\cdot\rot\psi_{,x}dx\cr
&\quad+\intop_\Omega[(a+\eta)v\cdot\nabla v]_{,x}\cdot\rot\psi_{,x}dx+ \intop_\Omega[(a+\eta)f]_{,x}\cdot\rot\psi_{,x}dx.\cr}
\label{5.16}
\end{equation}
Integrating by parts in the terms from the r.h.s. of the above inequality we obtain
\begin{equation}\eqal{
&{d\over dt}|\rot\psi_{,x}|_2^2+{\mu\over a}|\rot\psi_{,x}|_2^2\le c[|\eta|_\infty^2|v_{,t}|_2^2+(1+|\eta|_\infty^2)|v|_6^2|\nabla v|_3^2\cr
&\quad+(1+|\eta|_\infty^2)|f|_2^2].\cr}
\label{5.17}
\end{equation}
Multiplying (\ref{5.17}) by $1/\nu$ and adding to (\ref{5.16}) yields
\begin{equation}\eqal{
&{d\over dt}\bigg(|\nabla\varphi_{,x}|_2^2+{1\over\nu}|\rot\psi_{,x}|_2^2\bigg)+ {\mu\over a}|\nabla^2\varphi_{,x}|_2^2+{\nu\over a}|\Delta\varphi_{,x}|_2^2+{\mu\over\nu a}|\nabla\rot\psi_{,x}|_2^2\cr
&\le{c\over\nu}[|\eta|_\infty^2|v_t|_2^2+(1+|\eta|_\infty^2)|v|_6^2|\nabla v|_3^2+(1+|\eta|_\infty^2)|\nabla\eta|_2^2\cr
&\quad+(1+|\eta|_\infty^2)|f|_2^2].\cr}
\label{5.18}
\end{equation}
Using that $|\eta|_\infty\le a/2$ the above inequality implies (\ref{5.13}). This concludes the proof.
\end{proof}

\begin{remark}\label{r5.5}
We use the restriction introduced in Remark \ref{r5.3}. 
Then (\ref{5.12}) and (\ref{5.13}) imply the inequality
\begin{equation}\eqal{
&{d\over dt}\bigg(\|\nabla\varphi\|_1^2+|\nabla\varphi_t|_2^2+{1\over\nu} \|\rot\psi\|_1^2+{1\over\nu}|\rot\psi_{,t}|_2^2\bigg)\cr
&\quad+{\mu\over a}(\|\nabla\varphi\|_2^2+\|\nabla\varphi_{,t}\|_1^2)+{\nu\over a}(\|\Delta\varphi\|_1^2+|\Delta\varphi_{,t}|_2^2)\cr
&\quad+{\mu\over a\nu}(\|\rot\psi\|_2^2+|\rot\psi_{,t}|_2^2)\cr
&\le{c\over\nu}[|\eta|_\infty^2|v_{,t}|_2^2+|\eta|_{1,1}^2+|v|_3^2|\nabla v|_2^2+|v_{,t}|_3^2|v|_6^2+|v|_6^2|\nabla v|_3^2\cr
&\quad+|f|_2^2+|f_t|_2^2].\cr}
\label{5.19}
\end{equation}
\end{remark}

\begin{lemma}\label{l5.6}
Assume that $\varphi$, $\psi$, $\eta$ are solutions to problem (\ref{3.1})--(\ref{3.3}) and Corollary \ref{c4.2} holds. Then
\begin{equation}\eqal{
&{d\over dt}\bigg(|\nabla\varphi_{,xt}|_2^2+{1\over\nu}|\rot\psi_{,xt}|_2^2\bigg) +{\mu\over a}|\nabla^2\varphi_{,xt}|_2^2+{\nu\over a}|\Delta\varphi_{,xt}|_2^2+ {\mu\over\nu a}|\nabla\rot\psi_{,xt}|_2^2\cr
&\le\nu\varepsilon|\nabla\varphi_{,xxx}|_2^2+{c\nu\over\varepsilon} (\|\eta_{,t}\|_1^2|\nabla\varphi_{,xx}|_2^2+\|\eta\|_2^2|\nabla\varphi_{,xt}|_2^2) +{c\over\nu}[|\eta_{,xt}|_2^2+|\eta|_{2,1}^4\cr
&\quad+\|\eta_t\|_1^2|\Delta v|_3^2+\|v\|_2^2\|v_{,t}\|_1^2+ |f_t|_2^2]+c\nu\|\eta_{,t}\|_1^2|\nabla\varphi_{,xx}|_2^2.\cr}
\label{5.20}
\end{equation}
\end{lemma}

\begin{proof}
Differentiate (\ref{3.11}) with respect to $x$ and $t$, multiply by $\nabla\varphi_{,xt}$ and integrate over $\Omega$. Then we have
\begin{equation}\eqal{
&{1\over2}{d\over dt}|\nabla\varphi_{,xt}|_2^2+{\mu\over a}|\nabla^2\varphi_{,xt}|_2^2+{\nu\over a}|\Delta\varphi_{,xt}|_2^2=-\intop_\Omega\bigg({a_0\over a+\eta}\nabla\eta\bigg)_{,xt}\cdot\nabla\varphi_{,xt}dx\cr
&\quad-{\mu\over a}\intop_\Omega\bigg({\eta\over a+\eta}\Delta v\bigg)_{,xt}\cdot\nabla\varphi_{,xt}dx-{\nu\over a}\intop_\Omega\bigg({\eta\over a+\eta}\nabla\Delta\varphi\bigg)_{,xt}\cdot\nabla\varphi_{,xt}dx\cr
&\quad+\intop_\Omega(v\cdot\nabla v)_{,xt}\cdot\nabla\varphi_{,xt}dx+\intop_\Omega\bigg[{1\over a+\eta}(p_\varrho(a)-p_\varrho(a+\eta))\nabla\eta\bigg]_{,xt}\!\!\cdot\nabla\varphi_{,xt}dx\cr
&\quad+\intop_\Omega f_{g,xt}\cdot\nabla\varphi_{,xt}dx.\cr}
\label{5.21}
\end{equation}
Now, we examine the particular terms from the r.h.s. of (\ref{5.21}). We repeat the proof of Lemma \ref{l3.4}. Let $I_1$ be the first term on the r.h.s. of (\ref{5.21}). Integrating by parts we estimate it by
$$
|I_1|\le\varepsilon|\nabla\varphi_{,xxt}|_2^2+c/\varepsilon(|\eta_{,t}|_4^2 |\nabla\eta|_4^2+|\nabla\eta_t|_2^2).
$$
Let $I_2$ be the second term on the r.h.s. of (\ref{5.21}). In view of Corollary \ref{c4.2} we can assume that $\|\eta\|_2\le1$. Then from the proof of Lemma \ref{l3.4} we have
$$\eqal{
|I_2|&\le\varepsilon(|\nabla\rot\psi_{,xt}|_2^2+\|\nabla\varphi_{,xt}\|_1^2)+ c/\varepsilon[\|\eta_t\|_1^2|\Delta v|_3^2+\|\eta\|_2^2(|v_{,xt}|_3^2\cr
&\quad+|\varphi_{,xt}|_2^2+|\nabla\varphi_{,xt}|_2^2)]+{\mu\over a}\intop_\Omega {\eta\over a+\eta}|\nabla^2\varphi_{,xt}|^2dx,\cr}
$$
where the last term is absorbed by the second term on the l.h.s. of (\ref{5.21}).

\noindent
Let $I_3$ be the third term on the r.h.s. of (\ref{5.21}). Then, from the proof of Lemma \ref{l3.4}, we have
$$\eqal{
|I_3|&\le\nu\varepsilon(|\nabla\varphi_{,xxt}|_2^2+|\nabla\varphi_{,xxx}|_2^2)+ {c\nu\over\varepsilon}(\|\eta_t\|_1^2|\nabla\varphi_{,xx}|_2^2+\|\eta\|_2^2 |\nabla\varphi_{,xt}|_2^2)\cr
&\quad+{\nu\over a}\intop_\Omega{\eta\over a+\eta}|\Delta\varphi_{,xt}|^2dx,\cr}
$$
where the last term is absorbed by the third term on the l.h.s. of (\ref{5.21}) and Corollary \ref{c4.2} is used.

\noindent
Consider the fourth term on the r.h.s. of (\ref{5.21}). We can express it in the form
$$
I_4=-\intop_\Omega v_{,t}\cdot\nabla v\cdot\nabla\varphi_{,xt}dx-\intop_\Omega v\cdot\nabla v_{,t}\cdot\nabla\varphi_{,xxt}dx\equiv I_4^1+I_4^2,
$$
where
$$\eqal{
&|I_4^1|\le\varepsilon|\nabla\varphi_{,xt}|_2^2+c/\varepsilon|v_{,t}|_6^2 |v_{,x}|_3^2,\cr
&|I_4^2|\le\varepsilon|\nabla\varphi_{,xxt}|_2^2+c/\varepsilon|v|_\infty^2 |v_{,xt}|_2^2.\cr}
$$
The fifth term on the l.h.s. of (\ref{5.21}) is estimated by
$$
|I_5|\le\varepsilon|\nabla\varphi_{,xxt}|_2^2+c/\varepsilon|\eta|_{2,1}^4 (1+|\eta|_{2,1}^2).
$$
Finally, the last term on the r.h.s. of (\ref{5.21}) is bounded by
$$
|I_6|\le\varepsilon|\nabla\varphi_{,xxt}|_2^2+c/\varepsilon|f_{,t}|_2^2.
$$
Employing the above estimates in (\ref{5.21}), assuming that $\varepsilon$ is sufficiently small and that $|\eta|_{2,1}\le1$ in view of Corollary \ref{c4.2}, we obtain
\begin{equation}\eqal{
&{d\over dt}|\nabla\varphi_{,xt}|_2^2+{\mu\over a}|\nabla^2\varphi_{,xt}|_2^2+{\nu\over a}|\Delta\varphi_{,xt}|_2^2\le {c\over\nu}[|\eta_{,xt}|_2^2+|\eta|_{2,1}^4\cr
&\quad+\|\eta_t\|_1^2|\Delta v|_3^2+\|\eta\|_2^2(|v_{,xt}|_3^2+\|\varphi_{,xt}\|_1^2)\cr
&\quad+\|v\|_2^2\|v_{,t}\|_1^2+|f_t|_2^2]+\nu\varepsilon |\nabla\varphi_{,xxx}|_2^2+{c\nu\over\varepsilon}(\|\eta_{,t}\|_1^2 |\nabla\varphi_{,xx}|_2^2\cr
&\quad+\|\eta\|_2^2|\nabla\varphi_{,xt}|_2^2).\cr}
\label{5.22}
\end{equation}
Differentiate (\ref{3.13}) with respect to $x$ and $t$, multiply the result by $\rot\psi_{,xt}$ and integrate over $\Omega$. Then we have
\begin{equation}\eqal{
&{1\over2}{d\over dt}|\rot\psi_{,xt}|_2^2+{\mu\over a}|\nabla\rot\psi_{,xt}|_2^2 =-{\mu\over a}\intop_\Omega\bigg({\eta\over a+\eta}\Delta v\bigg)_{,xt}\cdot \rot\psi_{,xt}dx\cr
&\quad-{\nu\over a}\intop_\Omega\bigg({\eta\over a+\eta}\Delta\nabla\varphi\bigg)_{,xt}\cdot\rot\psi_{,xt}dx+\intop_\Omega(v\cdot \nabla v)_{,xt}\cdot\rot\psi_{,xt}dx\cr
&\quad+\intop_\Omega f_{,xt}\cdot\rot\psi_{,xt}dx.\cr}
\label{5.23}
\end{equation}
Let $I_1$ be the first term on the r.h.s. of (\ref{5.23}). Repeating the proof of estimate of $I_1$ from the proof of Lemma \ref{l3.4} we have
$$\eqal{
|I_1|&\le\varepsilon(|\rot\psi_{,xxt}|_2^2+|\nabla\varphi_{,xxt}|_2^2)+ c/\varepsilon(|\eta_t|_6^2|\Delta v|_3^2+\|\eta\|_2^4|\rot\psi_{,xt}|_2^2)\cr
&\quad+{\mu\over a}\intop_\Omega{\eta\over a+\eta}|\nabla\rot\psi_{,xt}|^2dx,\cr}
$$
where the last term is absorbed by the last term on the l.h.s. of (\ref{5.23}).

\noindent
Let $I_2$ be the second term on the r.h.s. of (\ref{5.23}). Then from the proof of Lemma \ref{l3.4} we have
$$\eqal{
|I_2|&\le\varepsilon|\rot\psi_{,xxt}|_2^2+{c\nu^2\over\varepsilon}|\eta_t|_6^2 |\Delta\nabla\varphi|_3^2+\varepsilon_1\nu^2|\nabla\varphi_{,xxt}|_2^2\cr
&\quad+c/\varepsilon_1\|\eta\|_2^4|\rot\psi_{,xt}|_2^2.\cr}
$$
Next, we examine the third term on the r.h.s. ov (\ref{5.23}). We express it in the form
$$
I_3=-\intop_\Omega v_{,t}\cdot\nabla v\cdot\rot\psi_{,xxt}dx-\intop_\Omega v\cdot\nabla v_{,t}\cdot\rot\psi_{,xxt}dx\equiv I_3^1+I_3^2,
$$
where
$$\eqal{
&|I_3^1|\le\varepsilon|\rot\psi_{,xxt}|_2^2+c/\varepsilon|v_{,t}|_6^2|\nabla v|_3^2,\cr
&|I_3^2|\le\varepsilon|\rot\psi_{,xxt}|_2^2+c/\varepsilon|v|_\infty^2 |v_{,xt}|_2^2.\cr}
$$
Finally, the last term on the r.h.s. of (\ref{5.23}) is bounded by
$$
\varepsilon|\rot\psi_{,xxt}|_2^2+c/\varepsilon|f_t|_2^2.
$$
Employing the above estimates in (\ref{5.23}) and assuming that $\varepsilon$ is sufficiently small we derive the inequality
\begin{equation}\eqal{
&{d\over dt}|\rot\psi_{,xt}|_2^2+{\mu\over a}|\nabla\rot\psi_{,xt}|_2^2\le \varepsilon|\nabla\varphi_{,xxt}|_2^2+c/\varepsilon(|\eta_{,t}|_6^2|\Delta v|_3^2\cr
&\quad+\|\eta\|_2^4|\rot\psi_{,xt}|_2^2)+\varepsilon_1\nu^2 |\nabla\varphi_{,xxt}|_2^2+c/\varepsilon_1\|\eta\|_2^4|\rot\psi_{,xt}|_2^2+c\nu^2 |\eta_t|_6^2|\Delta\nabla\varphi|_2^2\cr
&\quad+c\|v_t\|_1^2\|v\|_2^2+c|f_t|_2^2.\cr}
\label{5.24}
\end{equation}
Multiplying (\ref{5.24}) by $1/\nu$, adding to (\ref{5.22}) and choosing $\varepsilon$ and $\varepsilon_1$ small we obtain
\begin{equation}\eqal{
&{d\over dt}\bigg(|\nabla\varphi_{,xt}|_2^2+{1\over\nu}|\rot\psi_{,xt}|_2^2\bigg)+ {\mu\over a}|\nabla^2\varphi_{,xt}|_2^2 +{\nu\over a}|\Delta\varphi_{,xt}|_2^2\cr
&\quad+{\mu\over a\nu}|\nabla\rot\psi_{,xt}|_2^2\le\nu\varepsilon |\nabla\varphi_{,xxx}|_2^2\cr &\quad+{c\nu\over\varepsilon}(\|\eta_{,t}\|_1^2 |\nabla\varphi_{,xx}|_2^2+\|\eta\|_2^2|\nabla\varphi_{,xt}|_2^2)\cr
&\quad+{c\over\nu}\Big[|\eta_{,xt}|_2^2+|\eta|_{2,1}^4+\|\eta_{,t}\|_1^2|\Delta v|_3^2+\|\eta\|_2^2|v_{,xt}|_3^2 +|f_t|_2^2\Big]\cr
&\quad+c\nu\|\eta_{,t}\|_1^2|\nabla\varphi_{,xx}|_2^2,\cr}
\label{5.25}
\end{equation}
where terms $\|\eta\|_2^2\|\nabla\varphi_{,xt}\|_1^2$, $\|\eta\|_2^4|\rot\psi_{,xt}|_2^2$ are absorbed by the l.h.s. terms in view of Corollary \ref{c4.2}. Hence we derive (\ref{5.20}) from (\ref{5.25}). This concludes the proof.
\end{proof}

\begin{lemma}\label{l5.7}
Assume that $\varphi$, $\psi$, $\eta$ are solutions to problem (\ref{3.1})--(\ref{3.3}). Let the assumptions of Corollary \ref{c4.2} hold. Then
\begin{equation}\eqal{
&{d\over dt}\bigg(|\nabla\varphi_{,xx}|_2^2+{1\over\nu}|\rot\psi_{,xx}|_2^2\bigg)+ {\mu\over a}|\nabla^2\varphi_{,xx}|_2^2+{\nu\over a}|\Delta\varphi_{,xx}|_2^2+ {\mu\over\nu a}|\nabla\rot\psi_{,xx}|_2^2\cr
&\le{c\over\nu}[\|\eta\|_2^2\|v_{,t}\|_1^2+(1+\|\eta\|_2^2)\|v\|_2^4+ \|f\|_1^2].\cr}
\label{5.26}
\end{equation}
\end{lemma}

\begin{proof}
Differentiating (\ref{3.2}) twice with respect to $x$, multiplying by $\nabla\varphi_{,xx}$ and integrating over $\Omega$ yields
\begin{equation}\eqal{
&{a\over2}{d\over dt}|\nabla\varphi_{,xx}|_2^2+\mu|\nabla^2\varphi_{,xx}|_2^2+ \nu|\Delta\varphi_{,xx}|_2^2=-\intop_\Omega[\eta v_{,t}]_{,xx}\cdot\nabla \varphi_{,xx}dx\cr
&\quad-a_0\intop_\Omega\nabla\eta_{,xx}\cdot\nabla\varphi_{xx}dx-\intop_\Omega [(a+\eta)v\cdot\nabla v]_{,xx}\cdot\nabla\varphi_{,xx}dx\cr
&\quad+\intop_\Omega[(p_\varrho(a)-p_\varrho(a+\eta))\nabla\eta]_{,xx}\cdot \nabla\varphi_{,xx}dx+\intop_\Omega[(a+\eta)f]_{,xx}\cdot\nabla\varphi_{xx}dx.\cr}
\label{5.27}
\end{equation}
The first term on the r.h.s. of (\ref{5.27}) is bounded by
$$
\varepsilon|\nabla\varphi_{,xxx}|_2^2+c/\varepsilon\|\eta\|_2^2\|v_{,t}\|_1^2,
$$
the second by
$$
\varepsilon|\nabla\varphi_{,xxx}|_2^2+c/\varepsilon|\nabla\eta_{,x}|_2^2.
$$
We write the third term in the form
$$
a\intop_\Omega(v\cdot\nabla v)_{,x}\cdot\nabla\varphi_{,xxx}dx+\intop_\Omega (\eta v\cdot\nabla v)_{,x}\cdot\nabla\varphi_{,xxx}dx\equiv I_1+I_2,
$$
where
$$
|I_1|\le\varepsilon|\nabla\varphi_{,xxx}|_2^2+c/\varepsilon(|v_{,x}|_3^2 |v_{,x}|_6^2+|v|_\infty^2|v_{,xx}|_2^2)
$$
and
$$\eqal{
|I_2|&\le\varepsilon|\nabla\varphi_{,xxx}|_2^2+c/\varepsilon(|\eta_{,x}|_6^2 |v|_\infty^2|v_{,x}|_3^2+|\eta|_\infty^2|v_{,x}|_3^2|v_{,x}|_6^2\cr
&\quad+|\eta|_\infty^2|v|_\infty^2|v_{,xx}|_2^2).\cr}
$$
The fourth term is estimated by
$$
\varepsilon|\nabla\varphi_{,xxx}|_2^2+c/\varepsilon\|\eta\|_2^4.
$$
Finally, the last term is bounded by
$$
\varepsilon|\nabla\varphi_{,xxx}|_2^2+c/\varepsilon[|f_{g,x}|_2^2+|\eta_{,x}|_6^2 |f|_3^2+|\eta|_\infty^2|f_{,x}|_2^2].
$$
Utilizing the estimates in (\ref{5.27}) yields
\begin{equation}\eqal{
&{d\over dt}|\nabla\varphi_{,xx}|_2^2+{\mu\over a}|\nabla^2\varphi_{,xx}|_2^2+ {\nu\over a}|\Delta\varphi_{,xx}|_2^2\le {c\over\nu}[\|\eta\|_2^2\|v_{,t}\|_1^2\cr
&\quad+\|\eta\|_2^2+\|\eta\|_2^4+|\eta_{,x}|_6^2|v|_\infty^2|v_{,x}|_3^2+ (1+|\eta|_\infty^2)(|v_{,x}|_3^2|v_{,x}|_6^2+|v|_\infty^2|v_{,xx}|_2^2)\cr
&\quad+\|f\|_1^2].\cr}
\label{5.28}
\end{equation}
Differentiating (\ref{3.2}) twice with respect to $x$, multiplying by $\rot\psi_{,xx}$ and integrating over $\Omega$ implies
\begin{equation}\eqal{
&{1\over2}{d\over dt}|\rot\psi_{,xx}|_2^2+{\mu\over a}|\nabla\rot\psi_{,xx}|_2^2 =-\intop_\Omega[\eta v_{,t}]_{,xx}\cdot\rot\psi_{,xx}dx\cr
&\quad-\intop_\Omega[(a+\eta)v\cdot\nabla v]_{,xx}\cdot\rot\psi_{,xx}dx+ \intop_\Omega[(a+\eta)f]_{,xx}\cdot\rot\psi_{,xx}dx\cr}
\label{5.29}
\end{equation}
The first term on the r.h.s. is bounded by
$$
\varepsilon|\rot\psi_{,xxx}|_2^2+c/\varepsilon\|\eta\|_2^2\|v_t\|_1^2
$$
We express the second term on the r.h.s. of (\ref{5.29}) in the form
$$
a\intop_\Omega(v\cdot\nabla v)_{,x}\cdot\rot\psi_{,xxx}dx+\intop_\Omega(\eta v\cdot\nabla v)_{,x}\cdot\rot\psi_{,xxx}dx\equiv J_1+J_2,
$$
where
$$\eqal{
|J_1|&\le\varepsilon|\rot\psi_{,xx}|_2^2+c/\varepsilon(|v_{,x}|_3^2|v_{,x}|_6^2+ |v|_\infty^2|v_{,xx}|_2^2),\cr
|J_2|&\le\varepsilon|\rot\psi_{,xxx}|_2^2+c/\varepsilon(|\eta_{,x}|_6^2|v|_6^2 |v_{,x}|_6^2+|\eta|_\infty^2|v_{,x}|_3^2|v_{,x}|_6^2\cr
&\quad+|\eta|_\infty^2|v|_\infty^2|v_{,xx}|_2^2).\cr}
$$
Finally, the last term on the r.h.s. of (\ref{5.29}) is estimated by
$$
\varepsilon|\rot\psi_{,xxx}|_2^2+c/\varepsilon(|f_{r,x}|_2^2+|\eta_{,x}|_6^2 |f|_3^2+|\eta|_\infty^2|f_{,x}|_2^2).
$$
Utilizing the estimates in (\ref{5.29}) and choosing that $\varepsilon$ is sufficiently small we have
\begin{equation}
{d\over dt}|\rot\psi_{,xx}|_2^2+{\mu\over a}|\nabla\rot\psi_{,xx}|_2^2\le c[\|\eta\|_2^2\|v_{,t}\|_1^2+\|v\|_2^4(1+\|\eta\|_2^2)+\|f\|_1^2].
\label{5.30}
\end{equation}
Multiply (\ref{5.30}) by $1/\nu$ and add to (\ref{5.28}). Then we obtain (\ref{5.26}). This concludes the proof.
\end{proof}

\begin{theorem}\label{t5.8}
Let $\varphi$, $\psi$, $\eta$ be solutions to problem (\ref{3.1})--(\ref{3.3}). Let the assumptions of Corollary \ref{c4.2} hold. Assume that there exists such time $T$ that
\begin{equation}
-{\mu\over a}T+c\intop_0^T\|v(t)\|_2^2dt<0.
\label{5.31}
\end{equation}
Assume that there exists such relation between $T$, $\intop_0^T\|v(t)\|_2^2dt$, $\eta(t)$, $f(t)$, $|\nabla\varphi(0)|_{2,1}+{1\over\sqrt{\nu}}|\rot\psi(0)|_{2,1}$ that
\begin{equation}\eqal{
&\exp\bigg(c\intop_0^T\|v(t)\|_2^2dt\bigg){c\over\nu}\intop_0^T(|\eta|_{2,1}^2+ |f|_{2,1}^2)dt\cr
&\quad+\exp\bigg(-{\mu\over a}T+c\intop_0^T\|v(t)\|_2^2\bigg)X^2(0)\le X^2(0),\cr}
\label{5.32}
\end{equation}
where $X^2(t)=\nu|\nabla\varphi(t)|_{2,1}^2+|\rot\psi(t)|_{2,1}^2+\|\eta(t)\|_2^2$. 
Then for any $k\in\N_0$ we have
\begin{equation}
X^2(kT)\le X^2(0).
\label{5.33}
\end{equation}
Moreover, assuming that $|\eta(0)|_{2,1}\le c_0/\nu$, $\|f_g(t)\|_1^2\le f_0^2e^{-\alpha t}$, $c_0$, $f_0$, $\alpha$ constants there exists a global solution $(\varphi,\psi,\eta)$ to problem (\ref{3.1})--(\ref{3.3}) such that in any time interval $[kT,(k+1)T]$ it is a solution described by Corollary \ref{c4.2}.
\end{theorem}

\begin{proof}
Adding (\ref{5.20}) and (\ref{5.26}) we have
\begin{equation}\eqal{
&{d\over dt}\bigg(|\nabla\varphi_{,x}|_2^2+|\nabla\varphi_{,xx}|_2^2+{1\over\nu} |\rot\psi_{,xt}|_2^2+{1\over\nu}|\rot\psi_{,xx}|_2^2\bigg)\cr
&\quad+{\mu\over a}(|\nabla^2\varphi_{,xt}|_2^2+|\nabla^2\varphi_{,xx}|_2^2)+ {\nu\over a}(|\Delta\varphi_{,xt}|_2^2+|\Delta\varphi_{,xx}|_2^2)\cr
&\quad+{\mu\over\nu a}(|\nabla\rot\psi_{,xt}|_2^2+|\nabla\rot\psi_{,xx}|_2^2)\cr
&\le c\nu(\|\eta_t\|_1^2|\nabla\varphi_{,xx}|_2^2+\|\eta\|_2^2|\nabla\varphi_{,xt}|_2^2)\cr
&\quad+c\nu\|\eta_t\|_1^2|\Delta\nabla\varphi|_2^2+{c\over\nu}[|\eta|_{2,1}^2+ \|\eta_t\|_1^2|\Delta v|_3^2+\|\eta\|_2^2\|v_t\|_1^2\cr
&\quad+(1+\|\eta\|_2^2)(\|v\|_2^2\|v_{,t}\|_1^2+\|v\|_2^4)+\|f\|_1^2+ |f_t|_2^2].\cr}
\label{5.34}
\end{equation}
In view of Corollary \ref{c4.2} we obtain from (\ref{5.34}) the inequality
\begin{equation}\eqal{
&{d\over dt}\bigg(|\nabla\varphi_{,xt}|_2^2+|\nabla\varphi_{,xx}|_2^2+{1\over\nu} |\rot\psi_{,xt}|_2^2+{1\over\nu}|\rot\psi_{,xx}|_2^2\bigg)\cr
&\quad+{\mu\over a}(|\nabla^2\varphi_{,xt}|_2^2+|\nabla^2\varphi_{,xx}|_2^2)+ {\nu\over a}(|\Delta\varphi_{,xt}|_2^2+|\Delta\varphi_{,xx}|_2^2)\cr
&\quad+{\mu\over\nu a}(|\nabla\rot\psi_{,xt}|_2^2+|\nabla\rot\psi_{,xx}|_2^2)\cr
&\le{c\over\nu}[|\eta|_{2,1}^2+\|\eta_t\|_1^2|\Delta v|_3^2+\|\eta\|_2^2 \|v_{,t}\|_1^2+(1+\|\eta\|_2^2)(\|v\|_2^2\|v_t\|_1^2\cr
&\quad+\|v\|_2^4)+\|f\|_1^2+|f_t|_2^2].\cr}
\label{5.35}
\end{equation}
From (\ref{5.19}), (\ref{5.35}) and Corollary \ref{c4.2} we have
\begin{equation}\eqal{
&{d\over dt}\bigg(|\nabla\varphi|_{2,1}^2+{1\over\nu}|\rot\psi|_{2,1}^2\bigg)+ {\mu\over a}|\nabla\varphi|_{3,1}^2+{\nu\over a}|\nabla\varphi|_{3,1}^2+{1\over a\nu}|\rot\psi|_{3,1}^2\cr
&\le{c\over\nu}[|\eta|_{2,1}^2+\|v\|_2^4+\|v\|_2^2\|v_t\|_1^2+|f|_{1,1}^2].\cr}
\label{5.36}
\end{equation}
Introduce the quantity
\begin{equation}
X_1^2=\nu|\nabla\varphi|_{2,1}^2+|\rot\psi|_{2,1}^2
\label{5.37}
\end{equation}
We also need a similar differential inequality for $\eta$. From (\ref{3.1}) and (\ref{3.2}) we have
\begin{equation}\eqal{
&\nabla\eta_t+a_0\nabla\eta=-\nabla(v\cdot\nabla\eta)-a\nabla\Delta\varphi- \nabla(\eta\Delta\varphi)-(a+\eta)v_t\cr
&\quad+\mu\Delta v+\nu\nabla\Delta\varphi-(a+\eta)v\cdot\nabla v+(p_\varrho(a)-p_\varrho(a+\eta))\nabla\eta\cr
&\quad+(a+\eta)f.\cr}
\label{5.38}
\end{equation}
Multiplying (\ref{5.38}) by $\nabla\eta$ and integrating over $\Omega$ yields
\begin{equation}\eqal{
&{1\over2}{d\over dt}|\nabla\eta|_2^2+a_0|\nabla\eta|_2^2=-\intop_\Omega\nabla (v\cdot\nabla\eta)\cdot\nabla\eta dx\cr
&\quad+\varepsilon|\nabla\eta|_2^2+c/\varepsilon[|\nabla\Delta\varphi|_2^2+ |\nabla\varphi_t|_2^2+|v_t|_2^2|\eta|_\infty^2+|\Delta\nabla\varphi|_2^2+\nu^2 |\nabla\Delta\varphi|_2^2\cr
&\quad+|v|_6^2X_1^2+|v|_6^2|\nabla v|_3^2|\eta|_\infty^2+|f_g|_2^2+ |f|_2^2|\eta|_\infty^2]+c|\eta|_\infty^2|\nabla\eta|_2^2.\cr}
\label{5.39}
\end{equation}
In view of Corollary \ref{c4.2} $|\eta|_\infty$ is bounded and small. Then (\ref{5.39}) takes the form
\begin{equation}\eqal{
&{d\over dt}|\nabla\eta|_2^2+a_0|\nabla\eta|_2^2\le-\intop_\Omega\nabla (v\cdot\nabla\eta)\cdot\nabla\eta dx+c[|\nabla\Delta\varphi|_2^2+|\nabla\varphi_t|_2^2\cr
&\quad+\nu^2|\nabla\Delta\varphi|_2^2+|f_g|_2^2]+c [|v_t|_2^2+|v|_6^2|\nabla v|_3^2+|f|_2^2]\|\eta\|_2^2\cr
&\quad+c|v|_6^2X_1^2.\cr}
\label{5.40}
\end{equation}
Consider the first term on the r.h.s. of (\ref{5.40}). Performing differentiation it takes the form
$$
I=-\intop_\Omega v\cdot\nabla\nabla\eta\nabla\eta dx-\intop_\Omega\nabla v\cdot\nabla\eta\nabla\eta dx\equiv I_1+I_2,
$$
where
$$
I_1=-{1\over2}\intop_\Omega v\cdot\nabla|\nabla\eta|^2dx={1\over2}\intop_\Omega \Delta\varphi|\nabla\eta|^2dx.
$$
Hence
$$
|I_1|\le\varepsilon|\nabla\eta|_2^2+c/\varepsilon|\Delta\varphi|_\infty^2 |\nabla\eta|_2^2.
$$
Similarly, we have
$$
|I_2|\le\varepsilon|\nabla\eta|_2^2+c/\varepsilon|\nabla v|_\infty^2 |\nabla\eta|_2^2.
$$
Employing the estimates in (\ref{5.40}) and using that $|v|_2\le c$ in view of the energy estimate (see Lemma \ref{l2.1})we have
\begin{equation}\eqal{
&{d\over dt}|\nabla\eta|_2^2+a_0|\nabla\eta|_2^2\le c[|\nabla v|_\infty^2+|\Delta\varphi|_\infty^2+|v_t|_2^2+|v|_6^2|\nabla v|_3^2\cr
&\quad+|f|_2^2] \|\eta\|_2^2+c[|\nabla\Delta\varphi|_2^2+ |\nabla\varphi_t|_2^2+\nu^2|\nabla\Delta\varphi|_2^2+|f_g|_2^2]+ c|v|_6^2X_1^2.\cr}
\label{5.41}
\end{equation}
Differentiate (\ref{5.38}) with respect to $x$, multiply by $\nabla\eta_x$ and integrate over $\Omega$. Then we have
\begin{equation}\eqal{
&{1\over2}{d\over dt}|\nabla\eta_x|_2^2+a_0|\nabla\eta_x|_2^2=-\intop_\Omega (\nabla(v\cdot\nabla\eta))_{,x}\cdot\nabla\eta_xdx\cr
&\quad+c[|\nabla\Delta\varphi_{,x}|_2^2+\|\nabla\varphi_t\|_1^2+\|v_t\|_1^2 \|\eta\|_2^2+\nu^2|\nabla\Delta\varphi_{,x}|_2^2\cr
&\quad+|[(a+\eta)v\cdot\nabla v]_{,x}|^2]+\intop_\Omega [(p_\varrho(a)-p_\varrho(a+\eta))\nabla\eta]_{,x}\cdot\nabla\eta_xdx\cr
&\quad+\intop_\Omega[(a+\eta)f]_{,x}\cdot\nabla\eta_{,x}dx.\cr}
\label{5.42}
\end{equation}
Carrying out differentiations in the first term on the r.h.s. of (\ref{5.42}) yields
$$
I_1=-\intop_\Omega(v\cdot\nabla\nabla\eta_x+\nabla v\cdot\nabla\eta_x+v_x\cdot\nabla\nabla\eta+\nabla v_x\cdot\nabla\eta)\cdot\nabla\eta_xdx.
$$
Hence
$$\eqal{
|I_1|&\le\varepsilon|\nabla\eta_x|_2^2+c/\varepsilon[|\nabla v|_\infty^2|\nabla\eta_x|_2^2+|v_{,x}|_\infty^2|\nabla^2\eta|_2^2+|\nabla v_x|_3^2|\nabla\eta|_6^2\cr
&\quad+|\Delta\varphi|_\infty^2|\nabla\eta_x|_2^2]\le\varepsilon |\nabla\eta_x|_2^2+c/\varepsilon(|\nabla v|_\infty^2+|\nabla v_x|_3^2+|\Delta\varphi|_\infty^2)\|\nabla\eta\|_1^2.\cr}
$$
Consider the fifth term under the square bracket on the r.h.s. of (\ref{5.42}). Then we obtain
$$\eqal{
I_3&=\intop_\Omega[(a+\eta)v\cdot\nabla v]_{,x}^2dx\le a\intop_\Omega(v\cdot\nabla v)_{,x}^2dx+\intop_\Omega(\eta v\cdot\nabla v)_{,x}^2dx\cr
&\le a\intop_\Omega|\nabla v|^4dx+a\intop_\Omega|v|^2|\nabla v_x|^2dx+\intop_\Omega|\eta_x|^2|v|^2|\nabla v|^2dx\cr
&\quad+\intop_\Omega|\eta|^2|\nabla v|^4dx+\intop_\Omega|\eta|^2|v|^2|\nabla v_x|^2dx\equiv\sum_{i=1}^5I_{3i}.\cr}
$$
Continuing, we have
$$\eqal{
&I_{31}\le c|\nabla v|_6^2|\nabla v|_3^2\le c\|v\|_2^4\le c\|v\|_2^2X_1^2\cr
&I_{32}\le c|v|_\infty^2|\nabla v_x|_2^2\le c|v|_\infty^2\|v\|_2^2\le c|v|_\infty^2X_1^2\cr
&I_{33}\le c|\eta_x|_6^2|v|_\infty^2|\nabla v|_3^2\le\|\nabla\eta\|_1^2|v|_\infty^2\|v\|_2^2,\cr
&I_{34}+I_{35}\le|\eta|_\infty^2(\|v\|_2^4+|v|_\infty^2\|v\|_2^2)\le c\|v\|_2^4\|\eta\|_2^2\cr}
$$
Hence
$$
I_3\le c\|v\|_2^2X_1^2+c\|v\|_2^4\|\eta\|_2^2.
$$
The last but one term on the r.h.s. of (\ref{5.42}) is bounded by
$$\eqal{
I_4&\le c\intop_\Omega|\eta|\,|\nabla\eta_x|^2dx+c\intop_\Omega|\nabla\eta|^2 |\nabla\eta_x|dx\cr
&\le\varepsilon|\nabla\eta_x|_2^2+c/\varepsilon|\eta|_\infty^2|\nabla\eta_x|_2^2+ c|\nabla\eta|_4^4.\cr}
$$
Finally, the last term on the r.h.s. of (\ref{5.42}) is estimated by
$$
I_5\le\varepsilon|\nabla\eta_x|_2^2+c/\varepsilon(|f_{g,x}|_2^2+|\eta_x|_4^2 |f|_4^2+|\eta|_\infty^2|f_{,x}|_2^2).
$$
Employing the above estimates in (\ref{5.42}) and assuming that $\varepsilon$ is sufficiently small implies
\begin{equation}\eqal{
&{d\over dt}|\nabla\eta_x|_2^2+a_0|\nabla\eta_x|_2^2\le c[|\nabla v|_\infty^2+|\nabla v_x|_3^2+|\Delta\varphi|_\infty^2+\|v_t\|_1^2+\|v\|_2^4\cr
&\quad+\|f\|_1^2]\|\eta\|_2^2+c(\|\nabla\varphi_t\|_1^2+ |\Delta\nabla\varphi_x|_2^2+\nu^2\|\Delta\nabla\varphi\|_1^2+|f_{g,x}|^2)\cr
&\quad+c\|v\|_2^2X_1^2,\cr}
\label{5.43}
\end{equation}
where the terms $|\eta|_\infty^2|\nabla\eta_x|_2^2+|\nabla\eta|_4^4$ are absorbed by the second term on the l.h.s. in view of Corollary \ref{c4.2}.

\noindent
From (\ref{5.41}) and (\ref{5.43}) we obtain (where we used that $\intop_\Omega\eta dx=0$)
\begin{equation}\eqal{
&{d\over dt}\|\eta\|_2^2+a_0\|\eta\|_2^2\le c[|\nabla v|_\infty^2+|v_{,xx}|_3^2+|\Delta\varphi|_\infty^2+\|v_t\|_1^2+\|v\|_2^4\cr
&\quad+\|f\|_1^2]\|\eta\|_2^2+c(\|\nabla\varphi_t\|_1^2+ \|\nabla\varphi_{,xx}\|_1^2+\nu^2\|\nabla\varphi_{,xx}\|_1^2+\|f_g\|_1^2)\cr
&\quad+c\|v\|_2^2X_1^2\cr}
\label{5.44}
\end{equation}
Using definition of $X_1$ (see (\ref{5.37})) we write (\ref{5.36}) in the form
\begin{equation}
{d\over dt}X_1^2+{\mu\over a}X_1^2+\nu^2|\nabla\varphi|_{3,1}^2\le c\|v\|_2^2X_1^2+c(|\eta|_{2,1}^2+|f|_{1,1}^2).
\label{5.45}
\end{equation}
Introduce the notation
$$
X_2^2=X_1^2+\|\eta\|_2^2,\quad a_*=\min\bigg\{a_0,{\mu\over a}\bigg\}.
$$
Then (\ref{5.44}) and (\ref{5.45}) imply the inequality after appropriate summing
\begin{equation}\eqal{
&{d\over dt}X_2^2+a_*X_2^2+\nu^2|\nabla\varphi|_{3,1}^2\le c(|v|_{3,1}^2+|\Delta\varphi|_\infty^2+\|v\|_2^4+\|f\|_1^2)X_2^2\cr
&\quad+c(|\eta|_{2,1}^2+\|f_g\|_1^2).\cr}
\label{5.46}
\end{equation}
Let
$$\eqal{
&G^2(t)=|v(t)|_{3,1}^2+|\Delta\varphi(t)|_\infty^2+\|v(t)\|_2^4+\|f(t)\|_1^2,\cr
&K^2(t)=|\eta|_{2,1}^2+\|f_g(t)\|_1^2.\cr}
$$
Then (\ref{5.46}) takes the following form because the integral $|\nabla\varphi|_{3,1}$ is absorbed by the last term on the l.h.s. of (\ref{5.46})
\begin{equation}
{d\over dt}X_2^2+a_*X_2^2\le cG^2X_2^2+K^2.
\label{5.47}
\end{equation}
Integrating (\ref{5.47}) from 0 to $t$ we have
\begin{equation}\eqal{
X_2^2(t)&\le\exp\bigg[c\intop_0^tG^2(t')dt'\bigg]\intop_0^tK^2dt'\cr
&\quad+\exp\bigg[-a_*t+c\intop_0^tG^2(t')dt'\bigg]X_2^2(0).\cr}
\label{5.48}
\end{equation}
Setting $t=T$ and assuming that
\begin{equation}
-{a_*\over2}T+c\intop_0^TG^2(t)dt\le 0
\label{5.49}
\end{equation}
we obtain from (\ref{5.48}) the inequality
\begin{equation}
X_2^2(T)\le\exp\bigg[c\intop_0^TG^2(t)dt\bigg]\intop_0^TK^2dt+\exp \bigg(-{a_*T\over2}\bigg)X_2^2(0).
\label{5.50}
\end{equation}
Consider (\ref{5.50}). In view of Corollary \ref{c4.2} we have that
\begin{equation}\eqal{
X_2^2(T)&\le\exp(\chi_1^2+\chi_2^2+\Psi^2)\bigg({\Psi^2\over\nu^2}+ T|\eta(0)|_{2,1}^2+e^{-{a_*\over2}T}\intop_0^T\|f_g(t)\|_1^2 e^{{a_*\over2}t}dt\bigg)\cr
&\quad+\exp\bigg(-{a_*\over2}T\bigg)X_2^2(0).\cr}
\label{5.51}
\end{equation}
Assuming that
\begin{equation}
|\eta(0)|_{2,1}\le{c_0\over\nu},\quad T\le\nu,\quad \|f_g(t)\|_1^2\le f_0^2e^{-\alpha t},\quad t\le T
\label{5.52}
\end{equation}
we obtain
$$
e^{-{a_*\over2}T}\intop_0^T\|f_g(t)\|_1^2e^{{a_*\over2}t}dt\le f_0^2\intop_0^Te^{-\alpha t+{a_*\over2}t}dt\le cf_0^2e^{-\alpha T}.
$$
Therefore, for $\nu$ sufficiently large we obtain from (\ref{5.51}) that
\begin{equation}
X_2^2(T)\le c/\nu
\label{5.53}
\end{equation}
Since $X_2^2(T)\ge|v(T)|_{2,1}^2+|\eta(T)|_{2,1}^2$ we can consider problem (\ref{3.1})--(\ref{3.3}) with small initial data at $t=T$. Let $X=X_2$.

\noindent
Hence, in view of \cite{BSZ, Z1, Z2, VZ} we have existence of global regular solutions to (\ref{3.1})--(\ref{3.3}) with small initial data at $t=T$ described by (\ref{5.53}).
\end{proof}

\noindent
The existence of solutions to problem (\ref{3.1})--(\ref{3.3}) in the time interval $[T,\infty)$ can be made by the step by step in time approach presented in \cite{Z3}.

\noindent
Therefore, we have

\begin{theorem}\label{t5.9}
Let the assumptions of Lemma \ref{l4.1} hold. Let $T$ be so large and estimate in Lemma \ref{l4.1} so appropriate that (\ref{5.49}) hold. Assume restrictions (\ref{5.52}). Assume that $\nu$ is sufficiently large. Then there exists a global solution to problem (\ref{3.1})--(\ref{3.3}) such that
$$
\eta\in L_\infty(T_1,T_2;\Gamma_1^2(\Omega)),\quad \nabla\varphi,\rot\psi\in L_\infty(T_1,T_2;\Gamma_1^2(\Omega))\cap L_2(T_1,T_2;\Gamma_1^3(\Omega)),
$$
where for $(T_1,T_2)\subset(0,T)$, the solution is described by Lemma \ref{l4.1} and for $T_1>T$, we have that there exists $T_0>0$ that $(T_1,T_2)=(T+kT_0,T+(k+1)T_0)$ for any $k\in\N_0$.
\end{theorem}

\begin{remark}\label{r5.10}
In view of Lemma \ref{l4.1} the constant $A$ (see (\ref{4.1})) may depend on time with a growth less than $T$. Then (\ref{5.49}) may always hold.
\end{remark}

\bibliographystyle{amsplain}

\end{document}